\documentclass[a4paper,11pt]{amsart}
\RequirePackage{etex}

\usepackage{amssymb}
\usepackage{amsmath}
\usepackage{amsthm}
\usepackage{amscd}
\usepackage{verbatim}
\usepackage{graphicx}
\usepackage[all]{xy}
\usepackage{url}

\usepackage{mathtools}
\usepackage{dsfont}

\usepackage{comment}
\usepackage{color}
\usepackage{url}

\usepackage[T2A,T1]{fontenc}
\usepackage[utf8]{inputenc}
\usepackage[russian,english]{babel}

\usepackage{dynkin-diagrams}
\tikzset{big arrow/.style={
    -Stealth,line cap=round,line width=1mm,
    shorten <=1mm,shorten >=1mm}}

\DeclareFontEncoding{OT2}{}{} 
\DeclareTextFontCommand{\textcyr}{\fontencoding{OT2}\fontfamily{wncyr}\fontseries{m}\fontshape{n}\selectfont}

\definecolor{brown}{RGB}{150,100,0}

\definecolor{purple}{RGB}{128,0,128}

\definecolor{grey}{RGB}{128,128,128}

\definecolor{gray}{RGB}{32,32,32}

\newcommand{\SmallMatrix}[1]{{\tiny\arraycolsep=0.4\arraycolsep\ensuremath
{\begin{pmatrix}#1\end{pmatrix}}}}

\newcommand{\0}[2]{{}_{#2}#1}
\newcommand{\ii}{{\boldsymbol{i}}}

\DeclareMathOperator{\Aut}{Aut}
\newcommand{\ad}{{\rm ad}}

\numberwithin{equation}{section}

\newcommand{\subsec}[1]{\setcounter{subsection}{\value{equation}}\subsection{#1}\setcounter{equation}{\value{subsection}}}
\swapnumbers

\theoremstyle{plain}

\newtheorem{theorem}[equation]{Theorem}
\newtheorem{proposition}[equation]{Proposition}

\newtheorem{lemma}[equation]{Lemma}
\newtheorem{corollary}[equation]{Corollary}

\theoremstyle{definition}

\newtheorem{remark}[equation]{Remark}
\newtheorem{example}[equation]{Example}
\newtheorem{examples}[equation]{Examples}
\newtheorem{definition}[equation]{Definition}
\newtheorem{notation}[equation]{Notation}
\newtheorem{construction}[equation]{Construction}

\title[Real Galois cohomology and real component group]
{Galois cohomology and component group\\ of a real reductive  group}

\author{Mikhail Borovoi and Dmitry A. Timashev}

\address{Raymond and Beverly Sackler School of Mathematical Sciences,
Tel Aviv University, 6997801 Tel Aviv, Israel}
\email{borovoi@tauex.tau.ac.il}

\address{
Lomonosov Moscow State University, Faculty of Mechanics and
Mathematics, Department of Higher Algebra, 119991 Moscow, Russia}
\email{timashev@mccme.ru}

\thanks{Borovoi was partially supported
by the Israel Science Foundation (grant 870/16).
Timashev was supported by the Russian Foundation
for Basic Research (grant 20-01-00091) and by the Ministry of Education and Science of the Russian Federation
in the framework of the program of the Moscow Center for Fundamental and Applied Mathematics (agreement 075-15-2019-1621).}

\keywords{Real reductive group, real Galois cohomology, real component group}

\subjclass{Primary: %
  11E72
, 20G10
, 20G20
}

\dedicatory{To the memory of Arkadi\u\i\ L\!'\hs vovich Onishchik}

\newcommand{\into}{\hookrightarrow}
\newcommand{\onto}{\twoheadrightarrow}
\newcommand{\lra}{\longrightarrow}

\newcommand{\labelt}[1]{\xrightarrow{\makebox[1.2em]{\scriptsize ${#1}$}}}
\newcommand{\labelto}[1]{\xrightarrow{\makebox[1.5em]{\scriptsize ${#1}$}}}
\newcommand{\labeltoo}[1]{\xrightarrow{\makebox[2em]{\scriptsize ${#1}$}}}
\newcommand{\labeltooo}[1]{\xrightarrow{\makebox[2.7em]{\scriptsize ${#1}$}}}

\newcommand{\longisoto}{{\ \labelt{\raisebox{-1.ex}{$\sim$}}\ }}
\newcommand{\isoto}{\longisoto}

\newcommand{\hs}{\kern 0.8pt}
\newcommand{\hssh}{\kern 1.2pt}

\newcommand{\hm}{\kern -0.8pt}

\newcommand{\emm}{\bfseries}

\newcommand{\Lie}{{\rm Lie}}

\newcommand{\Ad}{\mathds{A}}
\newcommand{\Cd}{\mathds{C}}
\newcommand{\Hd}{\mathds{H}}
\newcommand{\Qd}{\mathds{Q}}

\newcommand{\Zd}{\mathds{Z}}
\newcommand{\Rd}{\mathds{R}}

\newcommand{\Gd}{\mathds{G}}

\newcommand{\Z}{{\Zd}}
\newcommand{\Q}{{\Qd}}
\newcommand{\R}{{\Rd}}
\def\C{{\Cd}}
\newcommand{\G}{{\Gd}}
\renewcommand{\H}{{\Hd}}

\newcommand{\id}{{\rm id}}

\renewcommand{\AA}{{\bf A}}
\newcommand{\BB}{{\bf B}}
\newcommand{\CC}{{\bf C}}

\newcommand{\GG}{{\bf{G}}}
\newcommand{\HH}{{\bf H}}

\renewcommand{\SS}{{\bf S}}
\newcommand{\TT}{{\bf T}}
\newcommand{\UU}{{\bf U}}

\newcommand{\ZZ}{{\bf Z}}

\newcommand{\X}{{{\sf X}}}

\def\Hom{{\rm Hom}}

\def\Aut{{\rm Aut}}
\def\Gal{{\rm Gal}}
\def\Lie{{\rm Lie\,}}
\def\coker{{\rm coker\,}}
\def\im{{\rm im\,}}

\def\Ad{{\rm Ad}}
\newcommand{\inn}{{\rm inn}}

\def\pp{{p}}
\def\qq{{q}}

\newcommand{\GL}{{\rm{GL}}}

\newcommand{\Spin}{{{\rm Spin}}}
\def\Dtil{{\widetilde{D}}}

\newcommand{\GGtil}{\widetilde{\GG}}

\def\vev{\varepsilon^\vee}

\def\half{{\tfrac{1}{2}}}
\def\ihalf{{\tfrac{\ii}{2}}}

\newcommand{\uu}{{\rm u}}
\newcommand{\red}{{\rm red}}

\newcommand{\gf}[2]{\genfrac{}{}{0pt}{}{#1}{#2}}

\newcommand{\tauhat}{{\hat\tau}}

\newcommand{\DDD}{{\sf D}}

\newcommand{\Orb}[1]{{\rm Orb}({#1})}
\newcommand{\Orbb}[1]{{\rm Orb}\big({#1}\big)}

\newcommand{\Wtil}{{\widetilde{W}_0}}
\newcommand{\BRD}{{\rm BRD}}

\newcommand{\Rm}{{\mathcal R}}
\def\Sm{{\mathcal S}}
\newcommand{\What}{{\widehat{W}_0}}

\newcommand{\upgam}{{\hs^\gamma\hm}}

\newcommand{\sss}{{\rm ss}}
\newcommand{\ssc}{{\rm sc}}

\newcommand{\Ga}{\Gamma}

\newcommand{\Km}{{\mathcal K}}

\def\cc{\raise 1.7pt \hbox{\Tiny{$\bullet$}}}

\newcommand{\Gtil}{{\widetilde G}}

\newcommand\gtil{{\tilde g}}
\newcommand{\der}{{\rm ss}}

\newcommand{\Ho}{{\mathrm{H}\kern 0.4pt}}
\newcommand{\Zl}{{\mathrm{Z}\kern 0.2pt}}
\newcommand{\Bd}{{\mathrm{B}\kern 0.2pt}}
\newcommand{\HoT}{\widehat{\mathrm{H}}}

\newcommand\piR{{\pi_0^\R\hshs}}

\newcommand{\vphi}{\varphi}

\newcommand{\vt}{{\vartheta}}
\newcommand{\ve}{\varepsilon}
\newcommand{\vs}{{\varsigma}}

\newcommand{\Exp}{{\mathcal{E}}}
\newcommand{\Exphat}{\widehat{\mathcal{E}}}

\newcommand{\wt}{\widetilde}
\newcommand{\ov}{\overline}

\newcommand{\tl}{\mathfrak{t}}
\newcommand{\ttl}{\boldsymbol{\tl}}
\newcommand{\sll}{\mathfrak {s}}

\newcommand{\hl}{\mathfrak{h}}
\newcommand{\gl}{\mathfrak{g}}

\newcommand{\ggl}{\boldsymbol{\gl}}
\newcommand{\ssl}{\boldsymbol{\mathfrak{s}}}
\newcommand{\ul}{\mathfrak{u}}
\newcommand{\uul}{\boldsymbol{\ul}}

\newcommand{\Nm}{{\mathcal N}}

\newcommand{\GmC}{{\G_{{\rm m},\C}}}
\newcommand{\GmR}{{\G_{{\rm m},\R}}}

\def\piR{\pi_0^\R\hs}
\def\Hon{\Ho^1\hs}
\def\pia{{\pi_1^{\rm alg}}}

\def\prm{\prime\,}
\def\dprm{\prime\prime\,}

\newcommand{\scong}{{\scriptstyle\cong}}
\newcommand{\mm}{{m}}

\newcommand{\ab}{{\rm ab}}

\renewcommand{\top}{{\rm top}}
\newcommand{\pit}{{\pi_1^\top}}

\newcommand{\Bm}{{\mathcal B}}
\newcommand{\RG}{{\bf G}}

\newcommand{\parpr}{{\partial\hs'}}

\newcommand{\tw}{{\mathcal T}}

\newcommand\rtil{{\tilde\rho}}

\newcommand{\RW}{{\rm R}}
\newcommand{\ev}{{\rm ev}}

\newcommand{\ltil}{{\tilde l}}
\newcommand{\Ztil}{{\wt Z}}

\begin{document}

\date{\today}

\begin{abstract}
Let $\GG$ be a connected reductive group over the field of real numbers $\R$.
Using results of our previous joint paper,
we compute combinatorially the first Galois cohomology set $\Ho^1(\R,\GG)$
in terms of reductive Kac labelings.
Moreover, we compute the group of connected components $\pi_0\GG(\R)$ of the real Lie group $\GG(\R)$
and the maps in exact sequences containing $\pi_0\GG(\R)$ and $\Ho^1\hm(\R,\GG)$.
\end{abstract}

\maketitle

\tableofcontents

\section*{Introduction}
\label{s:Intro}

In this article,
by an $\R$-group we mean an algebraic group,
not necessarily linear or connected, over the field of real numbers $\R$.
We denote  $\Gamma=\Gal(\C/\R)=\{1,\gamma\}$,
the Galois group of $\C$ over $\R$,
where $\gamma$ is the complex conjugation.

\subsec{}
For the definition of the  first (nonabelian)
Galois cohomology set $\Hon\GG\coloneqq{}\Ho^1(\R,\GG)$
of a real algebraic group $\GG$
see Serre's book \cite[Section I.5]{Serre};
see also Section \ref{s:Nonab-R} below.
Galois cohomology can be used to answer many natural questions;
see Serre \cite[Section III.1]{Serre}.

The Galois cohomology set $\Hon \GG$ is finite; see Example~(a) in Section III.4.2
and Theorem~4 in Section III.4.3 of Serre's book \cite{Serre}.
Moreover, when $\GG$ is nonabelian, the set $\Hon \GG$ has a canonical neutral element, but no natural group structure,
and one is tempted to conclude that to compute $\Hon\GG$ is the same as to compute the cardinality $\#(\Hon \GG)$.
However, for applications in classification problems over $\R$,
 one needs {\em explicit cocycles} representing the cohomology classes;
see, for instance, \cite[Section 8]{Djokovic} or \cite[Section 3.3]{BGL}.

The first Galois cohomology sets of the classical groups $\GG$ are well known.
The  cardinalities of the Galois cohomology sets $\Hon\GG$ were computed by Adams and Ta{\"\i}bi \cite{AT}
for ``most'' of the  absolutely  simple $\R$-groups $\GG$,
in particular, for all  simply connected absolutely simple $\R$-groups
and all adjoint absolutely simple $\R$-groups.
In \cite{BE}, explicit cocycles were computed for all {\em simply connected, absolutely simple} $\R$-groups.
Thus  $\Ho^1\hs \GG$ is known for all  {\em simply connected semisimple} $\R$-groups $\GG$;
see \cite[Introduction]{BT} for details.
On the other hand, a slight modification of the method of Kac \cite{Kac} in the version of
 Onishchik and Vinberg \cite[Section 4.4]{OV} and
Gorbatsevich, Onishchik, and Vinberg \cite[Section 3.3]{GOV}
gives $\Ho^1\hs\GG$ (explicit cocycles) for all {\em absolutely simple, adjoint} $\R$-groups $\GG$.
Thus one obtains   $\Ho^1\hs\GG$  for all {\em adjoint semisimple} $\R$-groups $\GG$;
see \cite[Introduction]{BT} for details.

In \cite[Theorem 6.8]{Borel-Serre} (see also \cite[Section III.4.5, Theorem 6 and Example (a)]{Serre}),
Borel and Serre computed  $\Hon\GG$ for a {\em compact} group $\GG$.
For a compact connected $\R$-group $\GG$  (which is automatically reductive),
they constructed a canonical bijection $\TT(\R)^{(2)}/W\isoto \Hon\GG$,
where $\TT\subseteq\GG$ is a maximal torus,
$\TT(\R)^{(2)}$ is its subgroup of real points of order dividing~2, and $W=W(\GG,\TT)$ is the Weyl group.
Generalizing the result of Borel and Serre, the first-named author \cite{Borovoi88} computed
$\Hon\GG$ for a connected reductive $\R$-group, not necessarily compact
(see also \cite[Theorem 9]{Borovoi-arXiv}).
Note  that, similarly to the formula of Borel and Serre, the result of \cite{Borovoi88}
describes $\Hon\GG$ as the set of orbits of a certain Coxeter group $W_0$ of large order
in a set of large cardinality.
For example, if the adjoint group $\GG^\ad\coloneqq \GG/ Z(\GG)$
is an inner form of a classical simple compact group of absolute rank $\ell$,
then $W_0=W$ has order at least $\ell!$\hs.
However, the cardinality of $\Hon\GG$ is usually much smaller.
Therefore, it is a challenging problem to find a transparent efficiently computable description
of Galois cohomology of connected reductive $\R$-groups
that allows one to write down easily representatives of cohomology classes.

In \cite{BT}, using ideas  of Kac \cite{Kac0}, \cite{Kac},
ideas and results of Gorbatsevich, Onishchik, and Vinberg \cite{OV}, \cite{GOV},
and  the result of \cite{Borovoi88},
we computed the Galois cohomology set $\Hon\GG$ for all {\em semisimple} $\R$-groups $\GG$,
not necessarily simply connected or adjoint, via {\em Kac labelings}.
In the present article,  using results of \cite{BT}, we compute $\Hon\GG$
for a connected {\em reductive} $\R$-group $\GG$ via {\em reductive Kac labelings}.
The formulas for $\Hon\GG$ involve some combinatorial constructions and notation;
we refer to Section \ref{s:H1} for the statement and proof of our result.
Note that we obtain $\Hon\GG$ as the set of orbits of a certain finite abelian group $F_0$ of small order
acting on a relatively small set of reductive Kac labelings explicitly described in combinatorial terms.
If $\ZZ^\ssc$ is the center of the universal cover $\GG^\ssc$ of the commutator subgroup $[\GG,\GG]$ of $\GG$,
then the order of $F_0$ is a divisor of the order of the group $\ZZ^\ssc(\C)$; see Remark~\ref{r:F0|C}.
Our description of $\Hon\GG$ allows one to enumerate the cohomology classes
and to write down explicit cocycles representing them.
Note that computation of Galois cohomology of {\em any  connected} linear algebraic $\R$-group
reduces to the reductive case (see Subsection~\ref{ss:structure} below).

The set $\Hon\GG$ has certain additional structures: it is a functor of $\GG$,
there is a twisting map (see Serre \cite[Section I.5.3]{Serre}),
and there is an action of the abelian group $\Hon\hm Z(\GG)$
on $\Hon\GG$ (see Serre, \cite[Section I.5.7]{Serre}).
We discuss these additional structures in Section~\ref{s:Additional}.
We also compute the abelian cohomology group $\Ho^1_\ab\hs\GG$
and the abelianization map
\[\ab^1\colon \Hon\GG\to\Ho^1_\ab\hs\GG\]
introduced in \cite{Borovoi-Memoir}; see Proposition \ref{p:H1ab} and Theorem \ref{t:H1ab}.
One needs the abelian group $\Ho^1_\ab\hs\GG$ and the abelianization map in order to describe
the Galois cohomology of a reductive group over a {\em  number field};
see \cite[Theorem 5.11]{Borovoi-Memoir}.

\subsec{}
For a connected reductive $\R$-group $\GG$, consider the group of connected components
$\piR\GG\coloneqq \pi_0\hs\GG(\R)$
of the real Lie group $\GG(\R)$.
In the case of an absolutely simple $\R$-group $\GG$ of adjoint type, the group
$\piR\GG$ was tabulated in the papers \cite{Matsumoto},
\cite{Thang}, and \cite{AT}.
For a general connected reductive group $\GG$, the only known to us result
is that of Matsumoto \cite[Corollary of Theorem 1.2]{Matsumoto},
see also Borel and Tits \cite[Theorem 14.4]{Borel-Tits},
saying that if $\TT_s$ is a {\em maximal split} torus of $\GG$,
then the canonical homomorphism $\piR\TT_s\to\piR\GG$ is surjective.
From this result it follows that $\piR\GG\simeq(\Z/2\Z)^d$,
where $d\le{\rm rank}_\R(\GG)\coloneqq \dim\TT_s$.
In particular, the group $\piR\GG$ is abelian.

In this article, we compute $\piR\GG$ for a connected reductive $\R$-group $\GG$
in terms of the {\em algebraic fundamental group} $\pia\GG$
introduced in \cite{Borovoi-Memoir}.

Let $\GG^\ssc$ denote the universal cover of the commutator subgroup $[\GG,\GG]$ of $\GG$.
Let $\rho\colon \GG^\ssc\onto[\GG,\GG]\into \GG$ denote the natural homomorphism.
Let $\TT\subseteq\GG$ be a maximal torus.
We set $\TT^\ssc=\rho^{-1}(\TT)\subseteq\GG^\ssc$, which is a maximal torus in $\GG^\ssc$.

\begin{definition}[\hs{\cite{Borovoi-Memoir}}\hs]
The {\em algebraic fundamental group of $\GG$} is
\[\pia\GG=\coker \big[\rho_*\colon \X_*(\TT^\ssc)\to\X_*(\TT)\big],\]
where $\X_*(\TT)\coloneqq\Hom(\GmC,T)$
denotes the cocharacter group of the complex torus $T\coloneqq\TT\times_\R\C$,
and similarly for $\X_*(\TT^\ssc)$.
\end{definition}
As an abstract group, $\pia\GG$ is isomorphic (non-canonically)
to the topological fundamental group of the complex Lie group $\GG(\C)$.
The Galois group $\Gamma$ naturally acts on $\pia\GG$,
and the $\Gamma$-module $\pia\GG$ is well defined
(does not depend on the choice of $\TT$);
see \cite[Lemma~1.2]{Borovoi-Memoir}.

\begin{construction}
We wish to compute $\piR\GG$.
Let $\GG^\ad\coloneqq\GG/Z(\GG)$ denote the corresponding semisimple group of adjoint type.
Set $\TT^\ad=\TT/Z(\GG)\subset\GG^\ad$.
Write $C=\pia\GG^\ad=\X_*(\TT^\ad)/\X_*(\TT^\ssc)$.
The homomorphism $\Ad\colon \GG\to\GG^\ad$ induces a homomorphism
\[\Ad_*\colon\Ho^0\hs\pia\GG\lra \Ho^0\hs\pia\GG^\ad=\Ho^0\hs C,\]
where $\Ho^0\hs\pia\GG\coloneqq\HoT^0(\Gamma,\pia\GG)$ (zeroth Tate cohomology),
and similarly for $\Ho^0\hs C$.
We note that there is a $\Gamma$-anti-equivariant isomorphism of $\Gamma$-groups $C\isoto\ZZ^\ssc$
that induces a canonical isomorphism $\Ho^0\hs C\isoto \Ho^1\hs \ZZ^\ssc$,
where $\ZZ^\ssc=Z(\GG^\ssc)$, the center of $\GG^\ssc$;
see Definition \ref{d:anti} and Lemma \ref{l:anti} below.
The inclusion homomorphism $\iota\colon \ZZ^\ssc\into\GG^\ssc$ induces a map
$\iota_*\colon \Ho^1\hs\ZZ^\ssc\to\Ho^1\hs\GG^\ssc$.
Consider the composite map
\[\phi\colon\ \Ho^0\hs\pia\GG\labeltoo{\Ad_*}\Ho^0\hs C\isoto
      \Ho^1\hs \ZZ^\ssc \labelto{\iota_*}  \Ho^1\hs \GG^\ssc.\]
We write  \,$(\Ho^0\hs\pia\GG)_1\coloneqq \ker\phi\subseteq \Ho^0\hs\pia\GG$,
the preimage of the neutral element $[1]\in \Ho^1\hs \GG^\ssc$.
\end{construction}

\begin{theorem}\label{t:pi0R-0}\
\begin{enumerate}
\item[\rm(i)] The subset $(\Ho^0\hs\pia\hs\GG)_1\subseteq\Ho^0\hs\pia\GG$ is a subgroup.
\item[\rm (ii)] There is a canonical group isomorphism
    $\psi\colon (\Ho^0\hs\pia\hs\GG)_1\isoto \piR\GG$.
\end{enumerate}
\end{theorem}

Moreover, we show that the subgroup $(\Ho^0\hs\pia\hs\GG)_1$ is the stabilizer
of a Kac labeling $q\in\Km(\Dtil)$ defining the real form $\GG^\ssc$
of the complex semisimple group $G^\ssc$, under a certain action of the group $\Ho^0\hs\pia\GG$
on the affine Dynkin diagram $\Dtil$ of $\GG$; see Section \ref{s:pi0R}.
See Definition~\ref{ss:Kac-labeling} for the notion of a Kac labeling and the notation $\Km(\Dtil)$.

\subsec{}\label{ss:structure}
We describe the structure of the article.
In Sections \ref{s:Abelian}--\ref{s:uni}
we gather old and new results on group cohomology and hypercohomology
of abstract $\Gamma$-groups and on Galois cohomology of linear $\R$-groups.
In particular, for an $\R$-torus $\TT$, in Theorem \ref{p:X*}
we compute the Tate cohomology groups $\Ho^k\hs\TT$ for $k\in\Z$,
and in Corollary \ref{c:H0-X(T)} we compute the component group $\piR \TT$.
Moreover, for an $\R$-{\em quasi-torus} $\AA$
($\R$-group of multiplicative type), in Theorem \ref{t:quasi}
we compute the Tate cohomology groups $\Ho^k\hm\AA$.
Furthermore, for a short exact sequence of $\R$-groups
\begin{equation*}
 1\to \GG_1\labelto{i} \GG_2\labelto{j} \GG_3\to 1
\end{equation*}
(not necessarily connected or linear), in Section \ref{s:Nonab-R}
we construct an exact sequence
\begin{align}\label{e:G'GG''-pi-coh-Int}
\piR\GG_1\labelto{i_*} \piR\GG_2\labelto{j_*} \piR\GG_3\labelto{\delta^0}
      \Hon\GG_1\labelto{i_*} \Hon\GG_2 \labelto{j_*} \Hon\GG_3.
\end{align}
In Section \ref{s:uni} we show that if $\GG$ is a connected linear algebraic $\R$-group,
$\GG^\uu$ is its unipotent radical,
$\GG^\red\coloneqq \GG/\GG^\uu$ is the corresponding reductive $\R$-group,
and $\GG\to\GG^\red$ is the canonical homomorphism,
then the induced maps $\piR\hs\GG\to\piR\hs\GG^\red$ and
$\Ho^1\hs\GG\to\Ho^1\hs\GG^\red$ are bijective.
This reduces computing the Galois cohomology and component group
of a connected linear algebraic $\R$-group
to the case of a connected {\em reductive} $\R$-group; see Remark \ref{r:non-red}.

In Section \ref{s:H1} we state and prove Theorem \ref{t:H1}
that computes $\Hon \GG$  for a reductive $\R$-group $\GG$
in terms of {\em reductive Kac labelings}.
In Section \ref{s:pi0R} we prove Theorem \ref{t:pi0R-0}.
These are  our main results.
In Section \ref{s:Additional} we discuss additional structures on $\Hon \GG$:
functoriality, twisting, action of $\Hon\hm Z(\GG)$, and abelianization.
In Section \ref{s:connecting} we compute the connecting map $\delta^0$
in the exact sequence \eqref{e:G'GG''-pi-coh-Int} in the case
when  $\GG_2$ and $\GG_3$ are connected reductive $\R$-groups,
and $\GG_1$ is either a connected reductive $\R$-group or an $\R$-quasi-torus
(an $\R$-group of multiplicative type).

In Section \ref{s:examples} we consider examples:
we compute   $\Hon\GG$ and  $\piR\GG$ for certain connected reductive $\R$-groups $\GG$.

In Appendix \ref{s:Indecomposable} we give a short elementary proof of
the known classification of $\Gamma$-lattices
(finitely generated free abelian groups with $\Gamma$-action).

\subsec{Notation and conventions}\label{ss:not-conv}
\begin{itemize}
\item[\cc] $\Z$ denotes the ring of integers.
\item[\cc] $\Q$, $\R$, and $\C$ denote the fields of rational numbers,
           of real numbers, and of complex numbers, respectively.
\item[\cc] $\ii\in\C$ is such that $\ii^2=-1$.
           (Our results do not depend on the choice of $\ii$.)
\item[\cc] $\Gamma=\Gal(\C/\R)=\{1,\gamma\}$, the Galois group of $\C$ over $\R$,
           where $\gamma$ is the complex conjugation.
           The action of $\gamma$ on an element $s$ of a set $S$ is denoted by $s\mapsto\upgam s$.
\item[\cc] We denote real algebraic groups  by boldface letters $\GG$, $\TT$, \dots,
           their complexifications by respective
           Italic (non-bold) letters \,$G=\GG\times_\R \C$,  \,$T=\TT\times_\R\C$, \dots,
           the corresponding {\em complex} Lie algebras
           by respective lowercase Gothic letters $\gl=\Lie G$, \,$\tl=\Lie T$, \dots,
           and the corresponding {\em real} Lie algebras
           by respective boldface lowercase Gothic letters \,$\ggl(\R)=\Lie \GG$, \,$\ttl(\R)=\Lie \TT$, \dots.
\item[\cc] $\GG(\R)$ denotes the set of real points of a real algebraic group $\GG$,
           and $\GG(\C)$ denotes the set of complex points.
           By abuse of notation we identify $G$ with $\GG(\C)$. In particular,
           we write $g\in G$ for $g\in\GG(\C)$.
\item[\cc] For a homomorphism $\varphi:G\to H$ of algebraic (or Lie) groups,
           the differential at the unity $d\varphi:\gl\to\hl$ is a homomorphism of Lie algebras.
           By abuse of notation, we often write $\varphi$ instead of $d\varphi$.
\item[\cc] $G^0$ denotes the identity component of an algebraic (or Lie) group $G$.
\item[\cc] $\GG$ is an $\R$-group, not necessarily linear, connected, or reductive.
           In Section~\ref{s:uni}, $\GG$ is linear.
           Starting from Section~\ref{s:H1}, $\GG$ is connected and reductive.
           Moreover, in Section~\ref{s:H1} $\GG$ is compact.
\item[\cc] $Z(\GG)$ denotes the center of $\GG$.
\item[\cc] $\GmC$ and $\GmR$ denote the multiplicative groups over $\C$ and $\R$, respectively.
\item[\cc] $\X^*(\TT)=\Hom(T,\GmC)$, the group of complex characters
           of an algebraic $\R$-group~$\TT$.
           When $\TT$ is a torus, we regard $\X^*(\TT)$
           as a lattice in the dual space $\tl^*$ of $\tl$,
           in view of the canonical embedding $\X^*(\TT)\into\tl^*$, $\chi\mapsto d\chi$.
\item[\cc] $\X_*(\TT)=\Hom(\GmC,T)$, the group of complex cocharacters of an $\R$-torus $\TT$.
           We regard $\X_*(\TT)$ as a lattice in $\tl$,
           in view of the canonical embedding $\X_*(\TT)\into\tl$, ${\nu\mapsto d\nu(1)}$.
\item[\cc] $A\simeq B$ means that two groups (or algebraic groups) $A$ and $B$ are isomorphic.
\item[\cc] $A\cong B$ means that $A$ and $B$ are {\em canonically} isomorphic.
\item[\cc] By an {\em exact commutative diagram} we mean a commutative diagram with exact rows.
\end{itemize}

\section{Abelian cohomology}
\label{s:Abelian}

\subsec{}\label{ss:H1-abelian}
Let $A$ be a $\Ga$-module, that is, an abelian group written additively,
endowed with an action of $\Gamma=\{1,\gamma\}$.
We consider the first cohomology group $\Ho^1(\Ga,A)$.
We write $\Ho^1\hm A$ for $\Ho^1(\Ga,A)$.
Recall that
\[\Ho^1\hm A=\Zl^1\hm A/\Bd^1\hm A,\quad\
\text{where}\quad\ \Zl^1\hm A=\{a\in A\mid\upgam a=-a\}, \quad
\Bd^1\hm A=\{\upgam a'-a'\mid a'\in A\}.\]

We define the second cohomology group $\Ho^2A$ by
\[ \Ho^2\hm A=\Zl^2\hm A/ \Bd^2\hm A,\quad\
\text{where}\quad\ \Zl^2\hm A=A^\Gamma\coloneqq \{a\in A\mid \upgam a=a\},\quad\
\Bd^2\hm A=\{\upgam a'+a'\mid a'\in A\}.\]

For $k\in\Z$ we define the coboundary operator
\[d^k\colon A\to A,\quad a\mapsto\upgam a+(-1)^{k+1} a.\]
In other words, $d^k=\gamma+(-1)^{k+1}\in \Z[\Gamma]$,
where $\Z[\Gamma]=\Z\oplus\Z\gamma$  is the group ring of $\Gamma$.
We calculate:
\[d^k\circ d^{k-1}=\big(\gamma+(-1)^{k+1}\big)\big(\gamma+(-1)^{k}\big)
     =\big(\gamma-(-1)^{k}\big)\big(\gamma+(-1)^k\big)=\gamma^2-(-1)^{2k}=0.\]
We define the {\em Tate cohomology groups} $\HoT^k\hm A$ for all $k\in \Z$ by
\[ \HoT^k\hm A=\Zl^k\hm A/\Bd^k\hm A,\]
where
\begin{align*}
\Zl^k\hm A=\ker d^k=\big\{a\in A \ \big|\ \upgam a=(-1)^k a\big\},\quad\
\Bd^k\hm A=\im d^{k-1}=\big\{\upgam a'+(-1)^k  a'\ \big|\  a'\in A\big\}.
\end{align*}

Then clearly
\[\HoT^k\hm A=\Ho^1\hm A\ \ \text{when $k$ is odd,\quad and}\quad
      \HoT^k\hm A=\Ho^2\hm A\ \ \text{when $k$ is even.}\]

In this article, for any $k\in\Z$ we write $\Ho^k\hm A$ for $\HoT^k\hm A$.
In particular,
\[\Ho^0\hm A=\Zl^0\hm A/\Bd^0\hm A=A^\Gamma\hm/\big\{\upgam a'+a' \ \big|\  a'\in A\big\},\ \,\text{and not}\ A^\Gamma.\]
If $z\in\Zl^k\hm A$, we write $[z]=z+\Bd^k\hm A\in \Ho^k\hm A$ for the cohomology class of $z$.

\begin{remark}
In the standard exposition, our definitions become theorems; see
\cite[Chapter XII, Section 7, p. 251]{CE} or \cite[Theorem 5 in Section 8]{AW}.
\end{remark}

\begin{lemma}[{See, for instance, \cite[Section 6, Corollary 1 of Proposition 8]{AW}}]
\label{l:2-xi}
For any $k\in \Z$ and $\xi\in \Ho^k\hm A$, we have $2\xi=0$.
\end{lemma}

\begin{proof}
Let $\xi=[z]$,  $z\in \Zl^k\hm A$. Then $z=(-1)^k\cdot\hm\upgam z$.
Hence
\[2z=z+(-1)^k\cdot\hm\upgam z=\big(\gamma+(-1)^k\big)
     \upgam z=d^{k-1}(\upgam z)\in\Bd^k \hm A.\]
Thus $2\xi=[2z]=0$.
\end{proof}

\begin{corollary}\label{c:2-inv}
If $A$ is a $\Gamma$-module such that the endomorphism
\[2\colon A\to A,\quad a\mapsto 2a\]
is invertible, then $\Ho^k\hm A=0$ for all $k$.
\end{corollary}

\subsec{}\label{ss:ABC-abelian}
Let
\[0\to A\labelto{i} B\labelto{j} C\to 0\]
be a short exact sequence of $\Ga$-modules.
It gives rise to a cohomology exact sequence
\begin{align}\label{e:Ga-ABC}
\cdots\longrightarrow\Ho^{k-1} C\labeltoo{\delta^{k-1}}\Ho^k\hm A\labelto{i_*^k} \Ho^k B
      \labelto{j_*^k} \Ho^k C\labelto{\delta^k}\Ho^{k+1}\hm A\longrightarrow\cdots
\end{align}
We recall the formula for $\delta^k$. We identify $A$ with $i(A)\subseteq B$.
Let $[c]\in \Ho^k C$, $c\in \Zl^k C$.
We lift $c$ to some $b\in B$ and set $a=d^k\hs b$.
Then $a\in \Zl^{k+1}\hm A$. We set $\delta^k[c]=[a]\in \Ho^{k+1}\hm A$.
In particular, we have
\begin{equation*}
\delta^0[c]=[\upgam b-b]\ \,\text{for}\ c\in \Zl^0\hs C,\qquad
      \delta^1[c]=[\upgam b +b]\ \,\text{for}\ c\in \Zl^1\hs C.
\end{equation*}

\begin{definition}\label{d:anti}
Let $A,A'$ be two $\Gamma$-modules.
By a {\em $\Gamma$-anti-equivariant} homomorphism $A\to A'$
we mean a homomorphism of abelian groups
\[\varphi\colon A\to A'\quad\text{such that}\quad
      \upgam\varphi(a)=-\varphi(\upgam a)\ \text{ for all }a\in A.\]
\end{definition}

\begin{lemma}[obvious]
\label{l:anti}
Let $\varphi\colon A\to A'$ be a $\Gamma$-anti-equivariant homomorphism of $\Gamma$-modules.
Then for any $k$ in $\Z$, the homomorphism $\varphi$ restricts to homomorphisms
\[\Zl^k\hm A\to \Zl^{k+1}\hm A'\quad\ {and}\quad\  \Bd^k\hm A\to \Bd^{k+1}\hm A'\]
and induces a homomorphism on cohomology
\[\varphi_*^k\colon \Ho^k\hm A\to \Ho^{k+1}\hm A'.\]
If, moreover, $\varphi$ is an isomorphism of abelian groups,
then $\varphi_*^k$ is an isomorphism for each $k$.
\end{lemma}

\begin{definition}
Let $A$ be a $\Gamma$-module. We denote by $(\ii)A$ the abelian group
consisting  of formal expressions $\{(\ii)a\mid a\in A\}$
with the addition law
$$(\ii)a+(\ii)a'=(\ii)(a+a').$$
Then $(\ii)(-a)=-(\ii)a$.
We define a $\Gamma$-action on $(\ii)A$ by
$$\upgam\big((\ii)a\big)=-(\ii)\hs \upgam\hm a.$$
\end{definition}

We have a canonical isomorphism of $\Gamma$-modules
\[(\ii)(\ii)A\isoto A,\quad (\ii)(\ii)a\mapsto -a.\]

\begin{corollary}[\hs from Lemma \ref{l:anti}\hs]
\label{c:ii, k+1}
For any $\Gamma$-module $A$, the canonical $\Gamma$-anti-equivariant isomorphism
\[A\to(\ii)A,\quad a\mapsto (\ii)a\ \ \text{for }a\in A,\]
induces canonical isomorphisms
\[ \Ho^k\hm A\isoto \Ho^{k+1}\hs(\ii) A\]
for all $k\in \Z$.
\end{corollary}

\section{Hypercohomology}
\label{s:hyper}

\begin{definition}
A {\em short complex of $\Gamma$-modules}
is a morphism of $\Gamma$-modules $A_{1}\labelt{\partial} A_0$.
\end{definition}

\subsec{}
For a short complex of $\Gamma$-modules $A_{1}\labelt{\partial} A_0$ and for
$k\in\Z$, we define a differential
\[D^k\colon\, A_1\oplus A_0\to A_1\oplus A_0,\quad D^k(a_1,a_0)=
      \big(-d^{k+1} a_1,\, d^k a_0-\partial a_1\big)\ \ \text{for}\ \, (a_1,a_0)\in A_1\oplus A_0\hs.\]
We calculate:
\begin{align*}
D^k\big( D^{k-1}(a_1,a_0)\big)&=D^k\big(\hs -d^k a_1,\ d^{k-1} a_0-\partial a_1\hs\big)\\
    &=\big(d^{k+1}d^k a_1,\ d^k d^{k-1} a_0-d^k\partial a_1+\partial\hs d^k a_1\big)=(0,0).
\end{align*}
Thus $D^k\circ D^{k-1}=0$.

We define the {\em  $k$-th Tate hypercohomology group} $\H^k(A_{1}\labelt{\partial} A_0)$ by
\begin{equation*}
\H^k(A_{1}\labelt{\partial} A_0)=\Zl^k(A_{1}\labelt{\partial} A_0)\hs/\hs\Bd^k(A_{1}\labelt{\partial} A_0),
\end{equation*}
where
\begin{equation*}
\begin{aligned}
&\Zl^k(A_{1}\labelt{\partial} A_0)=\ker D^k=\big\{(a_{1},a_0)\in A_{1}\oplus A_0,\ \big|\
      d^{k+1} a_1=0,\,  d^k a_0=\partial a_{1}\big\},\\
&\Bd^k(A_{1}\labelt{\partial} A_0)={\rm im\,} D^{k-1}=
    \big\{(-d^k a'_1,\  d^{k-1} a'_0-\partial a'_1\hs) \ \big|\ (a'_1, a_0')\in A_1\oplus A_0\big\}.
\end{aligned}
\end{equation*}
For simplicity we write $\H^k(A_{1}\to A_0)$ instead of $\H^k(A_{1}\labelt{\partial} A_0)$.
If $(a_1,a_0)\in\Zl^k(A_{1}\to A_0)$, we write $[a_1,a_0]\in\Ho^k(A_{1}\to A_0)$
for the cohomology class of $(a_1,a_0)$.
\begin{examples}\
\begin{enumerate}
\item We have an isomorphism
\[\Ho^k \hm A_0\longisoto \H^k(0\to A_0),\quad [a_0]\mapsto [0,a_0].\]

\item We have an isomorphism
\[\H^k(A_1\to 0)\longisoto \Ho^{k+1}\hm A_1, \quad [a_1,0]\mapsto [a_1].\]
\end{enumerate}
\end{examples}

The correspondence $(A_{1}\to A_0)\rightsquigarrow \H^k(A_{1}\to A_0)$ is a functor
from the category of short complexes of $\Gamma$-modules to the category of abelian groups.
Moreover, a short exact sequence of short complexes of $\Gamma$-modules
\[0\to(A_{1}\to A_0)\labelto\alpha (B_{1}\to B_0)\labelto\beta (C_{1}\to C_0)\to 0\]
gives rise to a hypercohomology exact sequence
\begin{equation}\label{e:exact-complexes}
\cdots\ \H^k(A_{1}\to A_0)\labelto{\alpha_*^k} \H^k(B_{1}\to B_0)\labelto{\beta_*^k}
                         \H^k(C_{1}\to C_0)\labelto{\delta^k}\H^{k+1}(A_{1}\to A_0)\ \cdots
\end{equation}

We specify the maps $\delta^k$ in the sequence \eqref{e:exact-complexes}.
Let $(c_1,c_0)\in \Zl^k(C_0\to C_1)\subseteq C_1\oplus C_0$.
We lift $(c_1,c_0)$ to some $(b_1,b_0)\in B_1\oplus B_0$ and set $(a_1,a_0)=D^k(b_1,b_0)$.
Then $(a_1,a_0)\in \Zl^{k+1}(A_1\to A_0)$,
and we set $\delta^k[c_1,c_0]=[a_1,a_0]\in\H^{k+1}(A_1\to A_0)$.

\begin{example}
Applying \eqref{e:exact-complexes} to the short exact sequence of complexes
\[0\to (0\to A_0)\labelto{\lambda} (A_{1}\to A_0)\labelto{\mu} (A_{1}\to 0)\to 0\]
with \,$\lambda(0,a_0)=(0,a_0)$, \,$\mu(a_1,a_0)=(a_1,0)$,
\,we obtain an exact sequence
\begin{equation}\label{e:coho-hyper}
\cdots\longrightarrow \Ho^k\hm A_{1} \labelto{\partial_*^k} \Ho^k\hm A_0\labelto{\lambda^k_*} \H^k(A_{1}\to A_0)
     \labelto{\mu_*^k} \Ho^{k+1}\hm A_{1}\,\labelto{\partial_*^{k+1}}\, \Ho^{k+1}\hm A_0\longrightarrow\cdots
\end{equation}
\end{example}

\begin{lemma}\label{l:lambda-mu}
The maps  $\lambda_*^k$, $\mu_*^k$, and  $\partial_*^{k+1}$,
in \eqref{e:coho-hyper} are the following:
\begin{align*}
&\lambda_*^k\colon\,\Ho^k\hm A_0\hs\to\hs \H^k(A_{1}\to A_0),\qquad    [a_0]\mapsto [0,a_0],\\
&\mu_*^k\colon\,\H^k(A_{1}\to A_0)\hs\to\hs \Ho^{k+1}\hm A_{1},\quad   [a_1,a_0]\mapsto [a_1],\\
&\partial_*^{k+1}\colon \Ho^{k+1}\hm A_1\to \Ho^{k+1}\hm A_0,\qquad\quad
             [a_1]\mapsto[\partial a_1].
\end{align*}
\end{lemma}

\begin{proof}
We only need to compute the map $\partial_*^{k+1}$.
Let $a_1\in \Zl^{k+1}\hm A_1$.
We lift ${(a_1,0)\in \Zl^k(A_1\to 0)}$ to $(a_1,0)\in A_1\oplus A_0$,
and apply $D^{k}$ to the lift.
We obtain
\[ D^k(a_1,0)=(-d^{k+1}a_1,-\partial a_1)=(0,-\partial a_1)\in \Zl^{k+1}(0\to A_0).\]
Identifying $\Ho^{k+1}(0\to A_0)\cong\Ho^{k+1}\hm A_0$, we obtain
\[ [0,-\partial a_1]=-[\partial a_1]=[\partial a_1]\in \Ho^{k+1}\hm A_0\hs,\]
where $-[\partial a_1]=[\partial a_1]$ by Lemma \ref{l:2-xi}.
\end{proof}

\begin{lemma}[well-known]
\label{l:quasi-inj}
Let $\partial\colon A_1\to A_0$ be an {\emm injective} homomorphism of $\Gamma$-modules.
Then the canonical homomorphism
\[j_*^k\colon\H^k(A_{1}\labelt{\partial} A_0)\lra \H^k(0\to \coker \partial)= \Ho^k\hs\coker \partial,
  \quad[a_1,a_0]\mapsto[0, a_0+\im\partial]\mapsto[a_0+\im\partial]\]
induced by the canonical morphism of complexes
\[j\colon (A_1\to A_0)\lra (0\to\coker \partial),\quad (a_1, a_0)\mapsto (0, a_0+\im\partial)\]
is an isomorphism.
\end{lemma}

\begin{proof}
We prove the surjectivity. Let\,
$a_0+\im\partial\in \Zl^k \hs\coker \partial$.
Then\, $d^k(a_0+\im\partial)=0$, that is, $d^k a_0=\partial a_1$ for some $a_1\in A_1$.
We have
\[\partial\hs d^{k+1}a_1=d^{k+1}\partial a_1=d^{k+1}d^k a_0=0.\]
Since $\partial$ is injective, we conclude that $d^{k+1}a_1=0$.
We see that
\[\big(a_1,a_0\big)\in\Zl^k(A_1\to A_0)\quad\text{and}\quad j_*^k\big[a_1,a_0\big]=[a_0+\im\partial],\]
which proves the surjectivity.

We prove the injectivity. Let
\[(a_1,a_0)\in \Zl^k(A_1\to A_0),\quad [a_1,a_0]\in\ker j_*^k.\]
Then
\[ a_0+\im\partial=d^{k-1}(a_0'+\im\partial)\quad\text{for some}\ a_0'\in A_0\hs,\]
that is,
\[a_0=d^{k-1} a_0'-\partial a_1' \quad\text{for some}\ a_0'\in A_0,\ a_1'\in A_1,\]
and, moreover,
\[(0,0)=D^k(a_1,a_0)=(-d^{k+1}a_1,\hs d^k a_0-\partial a_1).\]
We see that
\[\partial a_1=d^k a_0=d^k(d^{k-1} a_0'-\partial a_1')=-d^k\partial a_1'=-\partial\hs d^k a_1'.\]
Since $\partial$ is injective, we conclude that $a_1=-d^k a_1'$.
It follows that
\[(a_1,a_0)=(-d^k a_1',\hs d^{k-1}a_0'-\partial a_1')=D^{k-1}(a_1', a_0').\]
Thus $[a_1,a_0]=0$ and the homomorphism $j_*$ is injective,
which completes the proof of the lemma.
\end{proof}

\begin{example}
Applying \eqref{e:exact-complexes} to the short exact sequence of complexes
\[0\to(\ker\hm\partial\to 0)\lra (A_{1}\labelt{\partial} A_0)\lra (\im \partial\into A_0)\to 0,\]
where $\H^k(\im \partial\into A_0)\cong \Ho^k\hs\coker\hm\partial$ by Lemma \ref{l:quasi-inj},
we obtain an exact sequence
\begin{equation}\label{e:coho-hyper-1}
\cdots\ \Ho^{k-1}\hs\coker\hm\partial\to\Ho^{k+1}\ker\hm\partial\to \H^k(A_{1}\labelt{\partial} A_0)
                   \to\Ho^k\coker\hm\partial\to\Ho^{k+2}\ker\hm\partial\ \cdots
\end{equation}
\end{example}

\begin{definition}
A morphism of short complexes
\begin{equation}\label{e:quasi}
\varphi\colon\,(A_{1}\labelt{\partial} A_0)\,\to\, (A'_{1}\labelt{\parpr} A'_0)
\end{equation}
is called a {\em quasi-isomorphism} if the induced homomorphisms
\[\ker\partial\to\ker\parpr \quad \text{ and }\quad \coker\partial\to\coker\parpr\]
are isomorphisms.
\end{definition}

\begin{examples}\label{Ex:inj-surj}\
\begin{enumerate}
\item If $ A_{1}\into A_0$ is injective, then
  \   $(A_{1}\into A_0)\,\to\,\big( 0,\coker[ A_1\hm\into\hm A_0]\big)$ \ is a quasi-isomorphism.
\item If $ A_{1}\onto A_0$ is surjective, then
   \  $\big(\ker\,[ A_1\hm\onto\hm A_0]\to 0\big)\,\to\,(A_{1}\onto A_0)$ \ is a quasi-isomorphism.
\end{enumerate}
\end{examples}

\begin{proposition}[well-known]
\label{l:quasi}
  A quasi-isomorphism of complexes of $\Gamma$-modules  \eqref{e:quasi}
induces isomorphisms on the hypercohomology
\[\varphi_*^k\colon\hs \H^k(A_{1}\to A_0)\isoto \H^k(A'_{1}\to A'_0).\]
\end{proposition}

\begin{proof}[Idea of an elementary proof]
Using \eqref{e:coho-hyper-1},
we obtain from \eqref{e:quasi} an exact commutative diagram
\begin{equation*}
\xymatrix{
 \Ho^{k-1}\coker\partial\ar[r]\ar[d]_\cong &\Ho^{k+1}\ker\partial\ar[r]\ar[d]_\cong
                                           &\H^k(A_{1}\labelt{\partial} A_0)\ar[r]\ar[d]^-{\varphi^k_*}
                                           &\Ho^{k}\coker\partial\ar[r]\ar[d]^\cong  &\Ho^{k+2}\ker\partial\ar[d]^\cong \\
 \Ho^{k-1}\coker\parpr\ar[r]               &\Ho^{k+1}\ker\parpr\ar[r]                &\H^k(A_{1}'\labelt{\parpr} A_0')\ar[r]
                                           &\Ho^{k}\coker\parpr\ar[r]                &\Ho^{k+2}\ker\parpr
}
\end{equation*}
in which the middle vertical arrow $\varphi^k_*$ is an isomorphism  by the five lemma.
\end{proof}

\begin{example}\label{ex:HH(inj-sur)}
We have  $\H^k(A_{1}\to 0)\cong\Ho^{k+1}\hm A_{1}$.
Hence, if a homomorphism $\partial\colon A_{1}\to A_0$ is surjective, then
by Example \ref{Ex:inj-surj}(2) and Proposition  \ref{l:quasi} we have
\[\H^k(A_{1}\labelt\partial A_0)\,\cong\,\H^k(\ker\partial\to 0)\,\cong\,\Ho^{k+1}\ker\partial.\]
\end{example}

\subsec{}
Let
\[0\to A\labelto{i} B\labelto{j} C\to 0\]
be a short exact sequence of $\Ga$-modules, where we identify $A$ with $i(A)\subseteq B$.
The quasi-isomorphism
\[i_\#\colon (A\to 0)\,\to\, (B\to C)\]
induces isomorphisms
\[i_\#^k\colon \Ho^{k+1}\hm A\to \H^k(B\to C),\quad [a]\mapsto [a,0].\]
 Combining the cohomology exact sequences \eqref{e:Ga-ABC} and \eqref{e:coho-hyper},
we obtain a diagram with exact rows
\begin{equation}\label{e:coho-hyper-2}
\begin{aligned}
\xymatrix@C=15truemm@R=9truemm{
\Ho^k C\ar[r]^-{\delta^k}\ar@{=}[d] &\Ho^{k+1}\hm A\ar[r]^-{i_*^{k+1}}\ar[d]^-{i_\#^k} &\ \Ho^{k+1} B\ar@{=}[d]\\
\Ho^k C\ar[r]^-{\lambda_*^k}                      &\H^k(B\to C)\ar[r]^-{\mu_*^k}         &\ \Ho^{k+1} B
}
\end{aligned}
\end{equation}

\begin{lemma}\label{l:coho-hyper}
The diagram \eqref{e:coho-hyper-2} is commutative.
\end{lemma}

\begin{proof}
The right-hand rectangle clearly commutes,
and therefore it sufficed to show that the left-hand rectangle commutes.
We perform a calculation.
Let $c\in \Zl^k C$. We lift $c$ to some $b\in B$.
Then $i_\#^k\big(\delta^{k}[c]\big)=\big[d^k(b),0\big]$ and
$\lambda_*^k[c]=[0,c]=\big[0,j(b)\big]$,
whence
\[i_\#^k\big(\delta^k[c]\big)=-\lambda_*^k[c]+\big[d^k b, j(b)\big].\]
Since  $\big(d^k b, j(b)\big)=D^{k-1}(-b,0)\in \Bd^k(B\labelto{j} C)$,
we see that $i_\#^k\big(\delta^k[c]\big)=-\lambda_*^k[c].$
By  Lemma \ref{l:2-xi} we have $-[c]=[c]$, and hence
$i_\#^k\big(\delta^k[c]\big)=\lambda_*^k[c]$, as required.
\end{proof}

\section{Galois cohomology of tori and quasi-tori}
\label{s:Tori-quasi-tori}

\subsec{}
Let $\TT$ be an $\R$-torus.
Consider the {\em cocharacter group}
\[ \X_*(\TT)=\Hom(\GmC,T).\]
The group $\Gamma$ acts on $\X_*(\TT)$ by
\[(\upgam\nu)(z)=\upgam\big(\nu(\hs^{\gamma^{-1}}\! z)\big)\ \ \,
    \text{for}\  \,\nu\in\X_*(\TT),\ z\in\C^\times\]
(where in our case $\gamma^{-1}=\gamma$).
We see that $\X_*(\TT)$  is a $\Gamma$-lattice, that is,
a finitely generated free abelian group with $\Gamma$-action.

Let $L$ be a nonzero $\Gamma$-lattice.
We say that $L$ is {\em indecomposable} if it is not
a direct sum of its two nonzero $\Gamma$-sublattices.
Clearly, any $\Gamma$-lattice is a direct sum
of indecomposable lattices.

\begin{proposition}[well-known]
\label{p:indecomp}
Up to isomorphism, there are exactly three indecomposable $\Gamma$-lattices:
\begin{enumerate}
\item $\Z$ with trivial action of $\gamma$;
\item $\Z$ with the action of $\gamma$ by $-1$;
\item $\Z\oplus \Z$ with the action of $\gamma$ by the matrix \SmallMatrix{0&1\\1&0}.
\end{enumerate}
\end{proposition}

\begin{proof}
See  Casselman \cite[Theorem 2]{Casselman} or Appendix \ref{s:Indecomposable} below.
\end{proof}

\subsec{}
Let $\vphi \colon \TT\to\SS$ be a homomorphism of $\R$-tori;
then $\vphi_*\colon\X_*(\TT)\to\X_*(\SS)$ is a homomorphism of $\Gamma$-lattices.
In this way we obtain an equivalence between
the category of $\R$-tori and the category of $\Gamma$-lattices.

We say that an $\R$-torus is {\em indecomposable}
if it is not a direct product of its two nontrivial subtori.
Clearly, every $\R$-torus is a direct product of indecomposable $\R$-tori.
It is also clear that a torus $\TT$ is indecomposable
if and only if its cocharacter lattice $\X_*(\TT)$ is indecomposable.

\begin{corollary}[of Proposition \ref{p:indecomp};
      see, for instance, Voskresenski\u{\i} {\cite[Section 10.1]{Vos}}]
\label{c:indecoposable-tori}
Up to isomorphism, there are exactly three indecomposable $\R$-tori:
\begin{enumerate}
\item $\TT^1_s=\GmR=(\C^\times,\, z\mapsto \bar z)$, with group of $\R$-points $\R^\times$,
      a one-dimensional split torus;
\item $\TT^1_c=(\C^\times,\, z\mapsto\bar z \hs^{-1})$,
     with group of $\R$-points $U(1)=\{z\in\C^\times\mid z\bar z=1\}$,
     a one-dimensional compact torus;
\item $\RW_{\C/\R}\GmC=\big(\hs\C^\times\times\C^\times,\, (z_1,z_2)\mapsto
       (\bar z_2,\bar z_1)\hs\big)$, with group of $\R$-points $\C^\times$.
\end{enumerate}
\end{corollary}

\begin{notation}
Let $\AA$ be a commutative algebraic $\R$-group, written multiplicatively. For $k\in \Z$, write
\[\Ho^k\AA=\Ho^k(\R,\AA)\coloneqq \Ho^k\big(\Gamma,\AA(\C)\big).\]
\end{notation}

\begin{theorem}\label{p:X*}
Let $\TT$ be an $\R$-torus. Consider the canonical $\Gamma$-equivariant {\emm evaluation homomorphism}
\begin{equation}\label{e:e-X*-T}
\ev\colon \X_*(\TT)\to\TT(\C),\quad  \nu\mapsto \nu(-1)\ \ \text{for }\nu\in \X_*(\TT).
\end{equation}
Then for any $k\in \Z$, the induced homomorphism
\begin{equation}\label{e:e*-X*-T}
 \ev_*^k\colon \Ho^k\hs\hs\X_*(\TT)\to \Ho^k\hs \TT,\quad [\nu]\mapsto\big[\nu(-1)\big]
\end{equation}
is an isomorphism.
\end{theorem}

We give three different proofs of this theorem.

\begin{proof}[A proof using reduction to the split and compact cases]
One checks immediately that the theorem holds for $\TT^1_s$ and $\TT^1_c$.
It follows that the theorem holds for the split tori and for the compact tori.

Now let $\TT$ be an arbitrary $\R$-torus, and let $\TT_0\subseteq\TT$ denote its maximal compact subtorus.
Then $\TT/\TT_0$ is a split torus.
We have an exact commutative diagram
\[
\xymatrix{
0\ar[r] &\X(\TT_0)\ar[r]\ar[d] &\X_*(\TT) \ar[r]\ar[d] &\X_*(\TT/\TT_0)\ar[r]\ar[d] &0\\
1\ar[r] &\TT_0(\C) \ar[r] &\TT(\C) \ar[r] &(\TT/\TT_0)(\C)\ar[r] &1
}
\]
in which the vertical arrows are the evaluation homomorphisms.
We obtain a commutative diagram in which the rows
are the corresponding cohomology exact sequences.
Since the theorem holds for the compact torus $\TT_0$ and for the split torus $\TT/\TT_0$,
by the five lemma it holds for $\TT$ as well.
\end{proof}

\begin{proof}[A proof using the exponential map]
We write $\tl=\Lie \TT$ and identify $\X_*(\TT)$ with an additive subgroup of $\tl$ via
the embedding $\nu\mapsto d\nu(1)$ for $\nu\in\X_*(\TT)$.
The  short exact sequence of $\Gamma$-modules
\begin{equation*}
0 \longrightarrow \ii\hs\X_*(\TT) \labelto\iota \tl \labelto{\Exp} T \longrightarrow 1,
\end{equation*}
where $\Exp(x)=\exp{2\pi x}$ for $x\in\tl$,
gives rise to a cohomology long exact sequence
\begin{equation*}
\cdots\lra\Ho^k\tl \labelto{\Exp_*^k} \Ho^k\hs\TT \labelto{\delta^k}
    \Ho^{k+1}\hs\ii\hs\X_*(\TT) \labelto{\iota_*^{k+1}} \Ho^{k+1}\tl \longrightarrow \cdots,
\end{equation*}
in which $\Ho^k\tl=0$ for all $k\in\Z$ by Corollary \ref{c:2-inv}.
Hence the connecting homomorphism $\delta^k$ is an isomorphism.

Let $\ii\nu\in \Zl^{k+1}\hs\ii\hs\X_*(\TT)$;
then $\upgam\nu=(-1)^k\nu$.
It follows that $\nu(-1)\in\Zl^k\hs\TT$.
We show that the isomorphism $\delta^k$ sends  $\big[\nu(-1)\big]$ to $[\ii\nu]$.
Indeed, take
\[x=\ii\nu/2\in\tl,\quad\   t=\nu(-1)=\exp\pi\ii\nu=\Exp(x).\]
Then  $\delta^k$ sends $[t]$ to the class of
\[d^k x=\upgam(\ii\nu/2)+(-1)^{k+1}\ii\nu/2=-\ihalf\big(\upgam\nu+(-1)^k\nu\big)=(-1)^{k+1}\ii\nu.\]
By Lemma \ref{l:2-xi} we have $\big[(-1)^{k+1}\ii\nu\big]=[\ii\nu]$.
Thus $\delta^k$ indeed sends  $\big[\nu(-1)\big]$ to $[\ii\nu]$.

We consider the $\Gamma$-{\em anti-equivariant} isomorphism
\[ \X_*(\TT)\to \ii\hs\X_*(\TT), \quad \nu\mapsto\ii\nu,\]
which by Lemma \ref{l:anti} induces an isomorphism
\[\Ho^k\hs\X_*(\TT)\longisoto \Ho^{k+1}\hs\ii\hs\X_*(\TT).\]
The composite isomorphism
\[\Ho^k\hs\X_*(\TT)\longisoto \Ho^{k+1}\hs\ii\hs\X_*(\TT)\labeltooo{(\delta^k)^{-1}}\Ho^k\hs\TT\]
sends   $[\nu]\in \Ho^k\hs\X_*(\TT)$ to $[\ii\hs\nu]\in \Ho^{k+1}\hs\ii\hs\X_*(\TT)$
and then to $\big[\nu(-1)\big]\in\Ho^k\hs\TT$, which completes the proof of the theorem.
\end{proof}

\begin{proof}[A proof using the Tate--Nakayama theorem]
Write $A=\C^\times$; then
\[\Ho^1(\Gamma,A)=\{1\}\quad\text{and}\quad  \#\Ho^2(\Gamma, A)=2.\]
Let $a\in\Ho^2(\Gamma, A)$ denote the only nontrivial element of this group.
The triple $(\Gamma,A,a)$ trivially satisfies  the assumptions of the Tate--Nakayama theorem;
see \cite[Section IX.8, Theorem 14]{Serre-LF}.
By this theorem, for any torsion free $\Gamma$-module $M$ and for any integer $k$,
the cup product with $a$ induces an isomorphism
\[ \Ho^{k-2}(\Gamma,M)\longisoto \Ho^k(\Gamma, M\otimes A).\]

We take $M=\X_*(T)$; then $\TT(\C)\cong\X_*(T)\otimes\C^\times=M\otimes A$, and we obtain an isomorphism
\[\Ho^k(\Gamma,\X_*(T))\cong \Ho^{k-2}(\Gamma,\X_*(T))\hs\longisoto\hs \Ho^k(\Gamma,\X_*(T)\otimes \C^\times)=\Ho^k\hs \TT.\]
One can check that this isomorphism is induced by the evaluation homomorphism  \eqref{e:e-X*-T}.
\end{proof}

\begin{proposition}
\label{p:H0-T}
Let $\TT$ be an $\R$-torus. The  canonical surjective homomorphism
\[ \Zl^0\hs\TT\to\Ho^0\hs\TT,\]
where $\Zl^0\hs\TT=\TT(\R)$, induces an isomorphism
\[ \pi_0\TT(\R)\isoto \Ho^0\hs\TT.\]
\end{proposition}

\begin{proof}
The homomorphism $\tl\to\ttl(\R),\ \, x\mapsto x+\upgam x,$
is surjective, and therefore the image  $\Bd^0\hs\TT$ of the homomorphism of real Lie groups
\[ d\hs^{-1}\colon \TT(\C)\to \TT(\R),\quad t\mapsto t\cdot\hm\upgam t\]
is open.
Since the complex Lie group $\TT(\C)$ is connected, its image $d^{-1}\big(\TT(\C)\big)=\Bd^0\hs\TT$ is connected.
It follows that $\Bd^0\hs\TT=\TT(\R)^0$, and the proposition follows.
\end{proof}

\begin{corollary}[\hs{\cite[Section 5]{Casselman}}\hs]
\label{c:H0-X(T)}
The homomorphism \eqref{e:e-X*-T} induces isomorphisms
\[\Ho^0\hssh\X_*(\TT)\isoto \Ho^0\hssh\TT\isoto\pi_0\TT(\R).\]
\end{corollary}

\subsec{}
Let $\TT_{1}\labelto{\partial}\TT_0$ be a short complex of $\R$-tori.
Consider the short complex of $\Gamma$-modules $\X_*(\TT_{1})\labelto{\partial_*}\X_*(\TT_0)$.
Formula \eqref{e:e-X*-T} permits us to define the evaluation morphism
of short complexes of $\Gamma$-modules
\begin{equation}\label{e-quasi-tori}
\ev\colon\,\big(\hs\X_*(\TT_{1})\to\X_*(\TT_0)\hs\big)\,\lra\, \big(\hs\TT_{1}(\C)\to\TT_0(\C)\hs\big),
     \quad\  (\nu_1,\nu_0)\mapsto\big(\nu_1(-1),\nu_0(-1)\big)
\end{equation}
which in general is not a quasi-isomorphism.

\begin{proposition}\label{p:quasi-tori}
The morphism of short complexes \eqref{e-quasi-tori}
induces isomorphisms on hyper\-cohomology
\[\ev_*^k\colon\,\H^k\big(\hs\X_*(\TT_{1})\hm\to\hm\hm\X_*(\TT_0)\hs\big)\,
     \longisoto\, \H^k(\TT_{1}\hm\to\hm\hm\TT_0),\quad\ \
      [\nu_1,\nu_0]\mapsto\big[\nu_1(-1),\nu_0(-1)\big].\]
\end{proposition}

\begin{proof}
Using the sequence \eqref{e:coho-hyper} for the short complexes
$\big(\X_*(\TT_1)\to\X_*(\TT_0)\big)$ and $(\TT_1\to\TT_0)$,
we obtain an exact commutative diagram
\[
\xymatrix@C=6.5truemm{
\Ho^k\X_*(\TT_{1})\ar[r]\ar[d]_-{\cong} &\Ho^k\X_*(\TT_{0})\ar[r]\ar[d]_-{\cong}
     &\H^k\big(\hs\X_* (\TT_{1})\hm\hm\to\hm\hm\hm\X_* (\TT_{0})\big)\ar[r]\ar[d]^-{\ev_*^k}
     &\Ho^{k+1}\X_* (\TT_{1})\ar[r]\ar[d]^-{\cong}  &\Ho^{k+1}\X_*(\TT_{0})\ar[d]^-{\cong}\\
\Ho^k\TT_{1}\ar[r]                     &\Ho^k\TT_{0}\ar[r]
     &\H^k( \TT_{1}\hm\to\hm\hm \TT_{0})\ar[r]
     &\Ho^{k+1} \TT_{1}\ar[r]                    &\Ho^{k+1}\TT_{0}
}
\]
in which by Theorem \ref{p:X*} the four vertical arrows
labeled with $\scong$ are isomorphisms.
By the five lemma, the fifth vertical arrow
in the diagram (labelled $\ev_*^k$) is an isomorphism as well.
\end{proof}

\subsec{} Following Gorbatsevich, Onishchik, and Vinberg \cite[Section 3.3.2]{GOV},
we say that a {\em quasi-torus} over $\R$
is a commutative algebraic $\R$-group $\AA$ such that all elements of $\AA(\C)$ are semisimple.
In other words, $\AA$ is an $\R$-group of multiplicative type;
see Milne \cite[Corollary 12.21]{Milne-AG}.
In other words, $\AA$ is an $\R$-subgroup of some $\R$-torus $\TT$;
see, for instance, \cite[Section 2.2]{BGR}.
Set $\TT'=\TT/\AA$; then $\AA$ is the kernel of a surjective homomorphism of tori $\TT\to\TT'$.
The following theorem computes the Galois cohomology
of a quasi-torus $\AA$ in terms of the lattices $\X_*(\TT)$ and $\X_*(\TT')$.

\begin{theorem}\label{t:quasi}
Let $\AA$ be an $\R$-quasi-torus, the kernel
of a surjective homomorphism of $\R$-tori $j \colon \TT\to\TT'$.
Then there are  canonical isomorphisms
\[ \vt^{k}\colon\hs\H^k\big(\X_*(\TT)\to\X_*(\TT')\big)\,\isoto\,\Ho^{k+1}\hm\AA.\]
\end{theorem}

\begin{proof}
We have a short exact sequence
\begin{equation*}
1\to\AA\labelto{i}\TT\labelto{j}\TT'\to 1,
\end{equation*}
whence we obtain a quasi-isomorphism
\[i_\#\colon\hs (\AA\to 1)\lra(\TT\to\TT'),\]
an induced isomorphism
\[i_\#^k\colon \Ho^{k+1}\AA\isoto \H^k(\TT\to\TT'),\]
and a composite  isomorphism
\begin{equation}\label{e:vt}
\xymatrix@C=15truemm{
\vt^{k}\colon\hs\H^k\big(\X_*(\TT)\to\X_*(\TT')\big) \ar[r]^-{\ev_*^k} & \H^k(\TT\to\TT')
     \ar[r]^-{(i_\#^k)^{-1}} & \Ho^{k+1}\hm\AA,
}
\end{equation}
as required.
\end{proof}

\subsec{}\label{ss:quasi}
We specify the map \eqref{e:vt}.
Let
\[(\nu,\nu')\in \Zl^k\big(\X_*(\TT)\to\X_*(\TT')\big)\]
be a $k$-hypercocycle.
We consider its image
$\big(\nu(-1),\nu'(-1)\big)\in  \Zl^k(\TT\to\TT')$.
Since the homomorphism $j$ is surjective, we can lift $\nu'(-1)\in T'$
to some element $t\in T$.
Then
\[\big(\nu(-1),\nu'(-1)\big)\hs\sim\hs \big(\nu(-1)\cdot d^k(t)^{-1},\hs 1\big).\]
Thus
\[\vt^{k}[\nu,\nu']=\big[\nu(-1)\cdot d^k(t)^{-1}\big]
      =\big[\nu(-1)\cdot\hm\upgam t^{-1}\cdot t^{(-1)^k}\hs\big]\in \Ho^{k+1}\AA.\]

\begin{lemma}\label{l:A-X(T)-X(T')}
The following exact diagram, in which the arrows in the top and middle rows
are from \eqref{e:coho-hyper},
and the arrows in the bottom row are from \eqref{e:Ga-ABC}, is commutative:
\begin{equation}\label{e:A-X(T)-X(T')}
\begin{aligned}
\xymatrix@C=13truemm{
\Ho^k\hs\X_*(\TT')\ar[r]\ar[d]_-\cong  &\H^k\big(\hs\X_* (\TT)\hm\hm\to\hm\hm\X_* (\TT')\big)\ar[r]\ar[d]^-{\ev_*^k}_-\cong
                                                                     &\Ho^{k+1}\hs\X_* (\TT)\ar[d]^-\cong\\
\Ho^k\hs\TT'\ar[r]\ar@{=}[d]           &\H^k( \TT\hm\to\hm \TT')\ar[r]
                                                                     &\Ho^{k+1}\hs \TT\ar@{=}[d]\\
\Ho^k\hs\TT'\ar[r]                     &\Ho^{k+1}\hm\AA\ar[r]\ar[u]_-{i_\#^k}^-\cong
                                                                     &\Ho^{k+1}\hs \TT
}
\end{aligned}
\end{equation}
\end{lemma}

\begin{proof}
The top half of the diagram is clearly commutative,
and the bottom half is commutative by Lemma \ref{l:coho-hyper}.
\end{proof}

\section{Nonabelian cohomology for abstract $\Gamma$-groups}
\label{s:Nonab-abstract}

\subsec{} Let $A$ be a $\Ga$-group (written multiplicatively),
that is, a group (not necessarily abelian) endowed with an action of $\Ga$.
We consider the first cohomology $\Ho^1(\Ga,A)$.
We write $\Ho^1\hm A$ for $\Ho^1(\Ga, A)$. Recall that
\[ \Ho^1\hm A=\Zl^1\hm A/\!\sim,\quad\ \text{where}\ \quad
    \Zl^1\hm A=\big\{a\in A \ \big|\  a\cdot\hm\upgam a=1\big\},\]
and two 1-cocycles (elements of $\Zl^1\hm A$) $a_1,\,a_2$ are equivalent  or {\em cohomologous}
(we write ${a_1\sim a_2}$) if there exists $a'\in A$ such that
\[a_2=a'\cdot a\cdot (\upgam a')^{-1}.\]
If $a\in \Zl^1\hm A$, we write $[a]\in \Ho^1\hm A$ for the cohomology class of $a$.
The set $\Ho^1\hm A$ has a canonical {\em neutral element} $[1]$,
the class of the cocycle $1\in \Zl^1\hm A$.
The correspondence $A\rightsquigarrow \Ho^1\hm A$ is a functor
from the category of $\Gamma$-groups to the category of pointed sets.

If the group $A$ is abelian, then
\[\Ho^1\hm A=\Zl^1\hm A/\Bd^1\hm A,\]
where the abelian subgroups $\Zl^1\hm A$ and $\Bd^1\hm A$
were defined as in Subsection \ref{ss:H1-abelian}.
Thus $\Ho^1\hm A$ is naturally an abelian group in this case.

\begin{construction}
Let
\begin{equation}\label{e:ABC}
 1\to A\labelto{i} B\labelto{j} C\to 1
 \end{equation}
be a short exact sequence of  $\Gamma$-groups.
Then we have a cohomology exact sequence
\[
1\to A^\Ga\labelto{i} \Bd^\Ga\labelto{j} C^\Ga\labelto{\delta} \Ho^1\hm A
     \labelto{i_*} \Ho^1 B\labelto{j_*} \Ho^1 C;
\]
see Serre \cite[I.5.5, Proposition 38]{Serre}.
We recall the definition of the map $\delta$
in our case of the group $\Gamma$ of order 2.

Let $c\in C^\Gamma$. We lift $c$ to an element $b\in B$
and set $a=b^{-1}\cdot\hm\upgam b\in B$.
It is easy to check that in fact $a\in\Zl^1\hm A\subseteq A$.
We set $\delta(c)=[a]\in \Ho^1\hm A$.
\end{construction}

\begin{construction}
Assume that in \eqref{e:ABC} the subgroup $A$ is {\em central} in $B$.
Then we have a cohomology exact sequence
\[
1\to A^\Ga\labelto{i} \Bd^\Ga\labelto{j} C^\Ga\labelto{\delta} \Ho^1\hm A\labelto{i_*}
       \Ho^1 B\labelto{j_*} \Ho^1 C \labelto{\delta^1} \Ho^2 A;
\]
see Serre \cite[I.5.7, Proposition 43]{Serre}.
We recall the definition of the map $\delta^1$ in our case of the group $\Gamma$ of order 2.

Let  $c\in \Zl^1\hs C\subseteq C$; then $c\cdot\hm\upgam c=1$.
We lift $c$ to some element $b\in B$; then $b\cdot\hm\upgam\hs b\in A$.
We set $a=b\cdot\hm\upgam\hs b$.
We have $\upgam a=\upgam\hs b\cdot b$.
Since $\upgam a\in A$ and $A$ is central in $B$, we have
\[\upgam a=\upgam\hs b^{-1}\cdot\hm\upgam a\cdot\hm\upgam\hs b=
        \upgam\hs b^{-1}\cdot\hm\upgam\hs b\cdot b\cdot\hm\upgam\hs b=b\cdot\hm\upgam\hs b=a.\]
Thus  $a\in \Zl^2A$, and we set
$\delta^1[c]=[a]=[b\cdot\hm\upgam\hs b]\in \Ho^2\hm A$.
\end{construction}

Note that when the groups $A$, $B$, and $C$ are abelian,
we have $\delta(c)=\delta^0[c]$ for $c\in C^\Gamma=\Zl^0\hs C$,
where $\delta^0\colon \HoT^0\hs C\to \Ho^1\hm A$
is the homomorphism of Section \ref{ss:ABC-abelian}.
Moreover, then our map $\delta^1$ coincides with the homomorphism $\delta^1$
of Section \ref{ss:ABC-abelian}.

\section{Nonabelian Galois cohomology of real algebraic groups}
\label{s:Nonab-R}

\begin{notation}
Let $\GG$ be an algebraic $\R$-group, not necessarily abelian. We write
\[\Ho^1\hs\GG=H^1(\R,\GG)\coloneqq \Ho^1\big(\Gamma,\GG(\C)\big).\]

The group $\GG(\C)$ is a complex Lie group and $\GG(\R)$ is a real Lie group.
If $\GG$ is connected in the Zariski topology, then $\GG(\C)$
is connected in the usual Hausdorff topology,
but $\GG(\R)$ is not necessarily connected even if $\GG$ were connected.
Let $\GG(\R)^0\subseteq\GG(\R)$ denote the identity component,
which is clearly normal in $\GG(\R)$ and open  for the Hausdorff topology.
We write
\[ \piR\GG=\pi_0\hs \GG(\R)\coloneqq \GG(\R)/\GG(\R)^0.\]
If $g\in \GG(\R)$, we write $[g]=g\hs \GG(\R)^0\in\piR\GG$ for the class of $g$.
\end{notation}

\subsec{}
Let
\begin{equation*}
 1\to \AA\labelto{i} \BB\labelto{j} \CC\to 1
 \end{equation*}
be a short exact sequence of real algebraic groups (not necessary linear or connected).
Then we have a short exact sequence of $\Gamma$-groups
\[ 1\to \AA(\C)\labelto{i} \BB(\C)\labelto{j} \CC(\C)\to 1,\]
whence a cohomology exact sequence
\begin{equation}\label{e:H-ABC}
1\to \AA(\R)\labelto{i} \BB(\R)\labelto{j} \CC(\R)\labelto{\delta}
      \Ho^1\hm\AA\labelto{i_*^1} \Ho^1\BB\labelto{j_*^1} \Ho^1\CC.
\end{equation}

\begin{proposition}\label{p:pi-0}\
\begin{enumerate}
\renewcommand{\theenumi}{\roman{enumi}}

\item The map $\delta\colon \CC(\R)\to \Ho^1\hm \AA$ induces a map
\[\delta^0\colon\piR \CC\to \Ho^1\hm \AA.\]

\item The following sequence is exact:
\begin{equation}\label{e:pi-0}
\piR\AA\labelto{i_*^0} \piR \BB\labelto{j_*^0}\piR \CC\labelto{\delta^0}
     \Ho^1\hm\AA \labelto{i_*^1} \Ho^1 \BB \labelto{j_*^1} \Ho^1 \CC .
\end{equation}
\end{enumerate}
\end{proposition}

\begin{proof}
We show that $j\big(\BB(\R)^0\big)=\CC(\R)^0$.
Indeed, since $\BB(\R)^0$ is connected, we have $j\big(\BB(\R)^0\big)\subseteq \CC(\R)^0$.
On the other hand, since the homomorphism $j\colon \BB(\C)\to \CC(\C)$ is surjective,
we see that the differential
\[ dj\colon \Lie \BB\to \Lie \CC\]
is surjective (over $\C$, and hence over $\R$),
and therefore, the image $j\big(\BB(\R)^0\big)\subseteq \CC(\R)^0$ contains an open neighborhood of $1$.
It follows that $j\big(\BB(\R)^0\big)=\CC(\R)^0$.

We prove (i).
We define the map $\delta^0$ by $\delta^0[c]=\delta(c)$ for $c\in \CC(\R)$.
We show that the map $\delta^0$ is well defined.
Indeed, let $c_1,c_2\in \CC(\R)$, $c_2=c_0\hs c_1$ for some $c_0\in \CC(\R)^0$.
Then $c_0=j(b_0)$ for some $b_0\in \BB(\R)^0\subseteq \BB(\R)$,
and hence $\delta(c_1)=\delta(c_2)$,
as required.

We prove (ii).
We show that the sequence \eqref{e:pi-0} is exact at $\piR \CC$.
Indeed, let $[c_1],[c_2]\in \piR \CC$, $c_1,c_2\in \CC(\R)$.
Assume that $\delta^0[c_1]=\delta^0[c_2]$.
Then $\delta(c_1)=\delta(c_2)$, and hence, $c_2=j(b) c_1$ for some $b\in \BB(\R)$.
Thus $[c_2]=j_*^0[b]\cdot[c_1]$, which shows
that  the sequence \eqref{e:pi-0} is exact at $\piR \CC$.
Moreover, the map $\delta^0$ induces a map
\[\piR \CC\hs/\hs j_*^0(\piR \BB)\,\to\, \Ho^1\hm\AA.\]

We show that the sequence \eqref{e:pi-0} is exact at $\piR \BB$.
Let $[b]\in\piR \BB$, where $b\in \BB(\R)$.
Assume that $j_*^0[b]=1$.
Then $j(b)\in \CC(\R)^0$.
It follows that $j(b)=j(b_0)$ for some $b_0\in \BB(\R)^0$.
Then $b=i(a)\cdot b_0$ for some $a\in \AA(\R)$, and hence
$[b]=i_*^0[a]\cdot [b_0]=i_*^0[a]$, as required.

The exactness of  \eqref{e:pi-0} at $\Ho^1\hm\AA$ and at $\Ho^1\hs\BB$
follows from the exactness of \eqref{e:H-ABC}.
\end{proof}

\section{Taking quotient by the unipotent radical}
\label{s:uni}

\subsec{}
Let $\UU$ be a unipotent $\R$-group.
We may  assume that $\UU$ is embedded into $\GL_{n,\R}$ for some natural number $n$.
We consider the exponential and logarithm maps
\begin{align*}
&\exp\colon \ul\to U,\quad x\,\mapsto\, 1+x+x^2/2!+x^3/3!+\cdots \ \,\text{for $x\in \ul$;} \\
&\log\colon U\to \ul,\quad u\,\mapsto\, y-y^2/2+y^3/3-\cdots\ \
 \ \text{for $u=1+y\in U$.}
\end{align*}
Since $\UU$ is unipotent, these two maps are regular (polynomial).
They are mutually inverse and defined over $\R$.

\begin{lemma}[well-known]
\label{l:uni}
Let $\UU$ be a unipotent $\R$-group. Then:
\begin{enumerate}
\item[\rm(i)] $\piR\UU=\{1\}$;
\item[\rm(ii)] $\Ho^1\hs\UU=\{1\}$.
\end{enumerate}
\end{lemma}

\begin{proof}
(i) The exponential map $\exp\colon \ul\to U$
is an isomorphism of real algebraic varieties,
and hence it induces an isomorphism of real analytic varieties $\uul(\R)\longisoto\UU(\R)$
and an isomorphism of their component groups $\{1\}=\pi_0\, \uul(\R)\longisoto\piR\UU$.

\noindent
(ii) See Serre \cite[Section III.2.1]{Serre} for a proof using induction on $\dim \UU$
(over an arbitrary perfect field).
Here we give a constructive proof over $\R$ in one step.
Let $z\in \Zl^1\hs\UU$.
We have $z\cdot\hm\upgam z=1$, whence $\upgam z=z^{-1}$.
Set $u=\exp(-\half\log z)$; then $\upgam u = u^{-1}$ and $u^2=z^{-1}$.
We have $u^{-1}\cdot\hm \upgam u=u^{-2}=z$, whence $z\sim 1$ and $[z]=[1]$, as required.
\end{proof}

\subsec{}
Let $\GG$ be a linear $\R$-group, not necessarily connected.
Let $\GG^\uu$ denote the  unipotent radical of $\GG$,
that is, the largest  normal unipotent  subgroup of $\GG$.
We say that $\GG$ is {\em reductive} if $\GG^\uu=\{1\}$.
For any linear algebraic $\R$-group $\GG$, we set $\GG^\red=\GG/\GG^\uu$,
which is a reductive $\R$-group (not necessarily connected).

\begin{theorem}[Mostow]
\label{t:Mostow}
Let $\GG$ be a linear algebraic $\R$-group, not necessarily connected.
Then the short exact sequence
\[1\to\GG^\uu\lra \GG\labelto{r}\GG^\red\to 1\]
splits, that is, there exists a homomorphism of $\R$-groups (a splitting)
$s\colon \GG^\red\to \GG$ such that $r\circ s=\id_{\GG^\red}.$
\end{theorem}

\begin{proof}
See Mostow \cite[Theorem 7.1]{Mostow},
or Hochschild \cite[Section VIII.4, Theorem~4.3]{Hochschild},
or \cite{grp}, for proofs over any field of characteristic 0.
\end{proof}

\begin{theorem}[well-known]
\label{t:uni}
Let $\GG$ be a linear $\R$-group, not necessarily connected.
Consider the canonical surjective homomorphism
$r\colon\hs \GG\to \GG^\red$.  Then:
\begin{enumerate}
\item[\rm(i)] The induced map $r_*^1\colon\hs\Ho^1\GG\to \Ho^1\hs\GG^\red$ is bijective.
\item[\rm(ii)] The induced homomorphism $r_*^0\colon\hs\piR\GG\to\piR\hs\GG^\red$ is an isomorphism.
\end{enumerate}
\end{theorem}

\begin{proof}
We prove assertion (i). It follows from Sansuc's lemma;
see  \cite[Lemma 1.13]{Sansuc} or \cite[Proposition 3.1]{BDR}.
Here we deduce it from Theorem \ref{t:Mostow}.
Since $r\circ s=\id$, we obtain
\[r_*^1\circ s_*^1=\id\colon\, \Ho^1\hs\GG^\red\to\Ho^1\hs\GG\to\Ho^1\hs\GG^\red,\]
whence we see that the map $r_*^1$ is surjective.

We prove the injectivity.
Write $\UU=\GG^\uu$.
Let $g\in \Zl^1\hs \GG$.
By Serre \cite[I.5.5, Corollary 2 of Proposition 39]{Serre},
the fiber of the map $r_*^1$ over  $r_*^1[g]\in \Ho^1\,\GG^\red$  is in a canonical bijection
with the quotient of the set $\Ho^1\hs\0{\UU}g$ by the group $(\0{\GG^\red}g)(\R)$.
Here the left subscript $g$ denotes the twisting of an $\R$-group structure
by the cocycle $g$; see \cite[Section I.5.3]{Serre}.
Since the $\R$-group $\0{\UU}g$ is unipotent,
by Lemma \ref{l:uni}(ii) the cohomology set  $\Ho^1\hs \0{\UU}g$ is a singleton.
We conclude that the fiber of the map $r_*^1$ over  $r_*^1[g]$ is a singleton.
Thus the map $r_*^1$ is indeed injective.

We prove assertion (ii). We use Theorem \ref{t:Mostow}.
We have an isomorphism of $\R$-varieties
\[\UU\times \GG^\red\to \GG,\quad (u,h)\mapsto u\cdot s(h)\quad\text{for}\ \,u\in U,\, h\in G^\red.\]
Passing to $\R$-points, we obtain an isomorphism of real analytic manifolds
\[\UU(\R)\times \GG^\red(\R)\longisoto \GG(\R),\]
where $\UU(\R)$ is connected by Lemma \ref{l:uni}(i).
We obtain an isomorphism
\[\piR \GG^\red=\piR\UU\times\piR\GG^\red\longisoto \piR\GG,\]
which is clearly inverse to the homomorphism $r^0_*$.
\end{proof}

\begin{remark}\label{r:non-red}
Let $\GG$ be a {\em connected} linear algebraic $\R$-group;
then $\GG^\red$ is a connected reductive group.
In Sections \ref{s:H1} and \ref{s:pi0R} we shall compute $\Ho^1\,\GG^\red$ and $\piR\,\GG^\red$.
By Theorem~\ref{t:uni} this will give us $\Ho^1\hs\GG$ and $\piR\hs\GG$.
\end{remark}

\section{Galois cohomology of a reductive group}
\label{s:H1}

In this section we compute the Galois cohomology
of an arbitrary connected reductive $\R$-group
in transparent combinatorial terms.
We freely use the notation of \cite{BT}.
From now on,  by a semisimple or reductive  algebraic group
we always mean a {\em connected} semisimple or reductive algebraic group,
respectively.
We use the following notation:

\subsec{Notation}\label{ss:Notation-reductive}
\begin{itemize}
\item[\cc] $\GG$ is a connected reductive $\R$-group. In this section $\GG$ is compact (anisotropic).
\item[\cc] $\GG^\sss=[\GG,\GG]$, the commutator subgroup of $\GG$, which is semisimple.
\item[\cc] $\GG^\ssc$ is the universal cover of $\GG^\sss$, which is simply connected.
\item[\cc] $\rho\colon \GG^\ssc\onto\GG^\sss\into\GG$ is the canonical homomorphism.
\item[\cc] $\ZZ=Z(\GG)$, the center of $\GG$.
\item[\cc] $\ZZ^\ssc=Z(\GG^\ssc)=\rho^{-1}(\ZZ)$.
\item[\cc] $\GG^\ad=\GG/\ZZ\cong\GG^\ssc/\ZZ^\ssc$,
            which is a semisimple group of adjoint type.
\item[\cc] $\TT\subset \GG$ is a maximal torus.
\item[\cc] $\TT^\sss=\TT\cap\GG^\sss\subset\GG^\sss$.
\item[\cc] $\TT^\ssc=\rho^{-1}(\TT)\subset \GG^\ssc$.
\item[\cc] $\TT^\ad=\TT/\ZZ\subset\GG^\ad$.
\item[\cc] $\SS=Z(\GG)^0$, the identity component of $\ZZ=Z(\GG)$.
\item[\cc] $\Exp\colon\tl\to T$, $\Exp^\ssc\colon\tl^\ssc\to T^\ssc$,
    $\Exp^\sss\colon\tl^\sss\to T^\sss$, $\Exp^\ad\colon\tl^\ad\to T^\ad$, and $\Exp_S\colon \sll\to S$
    are the scaled exponential maps  given by the formula $x\mapsto \exp 2\pi x$
    (note that we identify $\tl^\ssc=\tl^\sss=\tl^\ad$).
\item[\cc] $A^{(2)}$ denotes the set of elements of order dividing $2$ in a subset $A$ of some group.

\end{itemize}

\subsec{}
As in \cite{BT}, let $\GG=(G,\sigma_c)$ be
a \emph{compact} connected reductive $\R$-group,
with the action of $\gamma$ on $G$ given by an anti-regular involutive automorphism $\sigma_c$. Let
$\TT\subset\GG$ be a maximal torus,
and $B\subset G$ be a Borel subgroup containing $T$.
Let $\Bm=\BRD(G,T,B)$ denote the based root datum of $(G,T,B)$;
see Springer  \cite[Sections 1 and 2]{Springer-Corvallis}.
Recall that
\[\BRD(G,T,B)=(X,X^\vee,\Rm,\Rm^\vee, \Sm,\Sm^\vee)\]
where
\begin{itemize}
\item[\cc] $X=\X^*(T)$ is the character group of $T$, and $X^\vee=\X_*(T)$ is the cocharacter group;
\item[\cc] $\Rm=\Rm(G,T)\subset X$ is the root system, and $\Rm^\vee=\Rm^\vee(G,T)\subset X^\vee$ is the coroot system;
\item[\cc] $\Sm=\Sm(G,T,B)\subset \Rm$ is the system of simple roots,
          and $\Sm^\vee=\Sm^\vee(G,T,B)\subset \Rm^\vee$ is the system of simple coroots with respect to $B$.
\end{itemize}
Let $\tau$ be an involutive automorphism (maybe identity) of $(\GG,\TT,B)$
coming from an automorphism of $\Bm$.
Let $\TT_0$ denote the identity component of the fixed point subgroup $\TT^\tau$,
and  $\theta\in\Aut(\GG,\TT,B)$ be an involutive automorphism
of the form $\theta=\inn(t_\theta)\circ\tau$ with $t_\theta\in T_0$ and $t_\theta^2\in Z$.
Our aim is to compute $H^1(\R,\0\GG\theta)$,
where $\0\GG\theta$ is the corresponding twisted $\R$-group with real structure
(the action of $\gamma$ on $G$)  given by $\sigma=\theta\circ\sigma_c$.

\subsec{}
Recall that $S$ is the connected center and $G^\der$
is the derived subgroup of $G$.
We have a decomposition into an almost direct product $G=G^\der\cdot S$.
There is a chain of isogenies
$$
G^\ssc\times S \longrightarrow G \longrightarrow G^\ad \times \overline{S},
$$
where $G^\ssc$ is the simply connected cover of $G^\der$, $G^\ad$
is the adjoint group of $G$,
and $\overline{S}=S/(S\cap G^\der)$.
Note that $T^\sss=T\cap G^\der$ is a maximal torus in $G^\der$, its preimage $T^\ssc$ in $G^\ssc$
is a maximal torus in $G^\ssc$,
and its image $T^\ad$ in $G^\ad$ is a maximal torus in $G^\ad$.
There is a chain of isogenies
$$
T^\ssc \times S \longrightarrow T=T^\sss\cdot S \longrightarrow T^\ad \times \overline{S}.
$$
The respective inclusion of character lattices reads as
$$
P\oplus\Lambda \supseteq X \supseteq Q\oplus M,
$$
where $X=\X^*(T)$, $\Lambda=\X^*(S)$, and  $M=\X^*(\ov S)$.
As usual, we denote by $P=\X^*(T^\ssc)$ the weight lattice
and by $Q=\X^*(T^\ad)$ the root lattice of the root system $R$.
Note that $P\supseteq Q$ and $\Lambda\supseteq M$.
\begin{lemma}\label{l:M-X-Lambda}
$M=X\cap\Lambda$.
\end{lemma}

\begin{proof}
We have $\ov S=T/T^\sss$,
and therefore $M=\X^*(\ov S)$ consists of the characters of $T$
that become trivial (identically 1) when restricted to $T^\sss$.
If we regard the characters of $T$ as characters of $T^\ssc\times S$,
then $M$ consists of the characters of $T$ that are trivial on $T^\ssc$,
that is, are contained in the direct summand
$\Lambda$ of $P\oplus \Lambda$, as required.
\end{proof}

\subsec{}
The respective inclusion of cocharacter lattices reads as
$$
Q^\vee\oplus\Lambda^\vee \subseteq X^\vee \subseteq P^\vee\oplus M^\vee,
$$
where  $X^\vee=\X_*(T)$, $\Lambda^\vee=\X_*(S)$, and  $M^\vee=\X_*(\ov S)$.
Then $X^\vee$ is dual to $X$, $\Lambda^\vee$
is dual to $\Lambda$, and $M^\vee$ is dual to $M$.
As usual, we denote by $Q^\vee=\X_*(T^\ssc)$ and $P^\vee=\X_*(T^\ad)$
the coroot and coweight lattice, respectively,
so that the lattice $P^\vee$ is dual to $Q$, and the lattice $Q^\vee$ is dual to $P$.
Note $Q^\vee\subseteq P^\vee$ and $\Lambda^\vee\subseteq M^\vee$.

\begin{lemma}\label{l:ABLMY}
Let $A, B, L, M, Y$ be lattices (finitely generated free abelian groups) such that
\[A\supseteq B,\quad\ L\supseteq M,\quad\ A\oplus L\supseteq Y\supseteq B\oplus M.\]
Assume that $[L:M]<\infty$.
Then $Y\cap L=M$ if and only if the natural map of the dual lattices $Y^\vee\to M^\vee$
induced by the inclusion $M\into Y$ is surjective.
\end{lemma}

\begin{proof}
We have $Y\supseteq M$ and  $L\supseteq M$; hence $Y\cap L\supseteq M$.
Moreover,
\begin{equation}\label{e:YLM}
\big[(Y\cap L):M\big]\le[L:M]<\infty.
\end{equation}

If $Y\cap L=M$, then $Y/M=Y/(Y\cap L)$ embeds into $(A\oplus L)/L=A$,
and hence $Y/M$ is torsion free.
It follows that $M$ is a direct summand of $Y$,
and therefore the natural map  $Y^\vee\to M^\vee$ is surjective, as required.

Conversely, if the map $Y^\vee\to M^\vee$ is surjective, then $Y/M$ is torsion free,
and hence $(Y\cap L)/M$ is torsion free.
Now it follows from  \eqref{e:YLM}  that $Y\cap L=M$, as required.
\end{proof}

Since by Lemma \ref{l:M-X-Lambda} we have  $X\cap\Lambda=M$,
by Lemma \ref{l:ABLMY} the lattice $X^\vee$ projects {\em onto}~$M^\vee$.
Since $S$ embeds into $T$, the lattice $X$ projects
{\em onto} $\Lambda$ in the direct sum $P\oplus\Lambda$,
and by Lemma~\ref{l:ABLMY} we have $X^\vee\cap M^\vee=\Lambda^\vee$.

\subsec{}
Consider the almost direct product decomposition $T=T_0\cdot T_1$,
where $\theta$ and $\tau$ act on $T_0$ trivially and on $T_1$ as inversion.
The subtori $T_0$ and $T_1$ of $T=\TT\times_\R\C$ are defined over $\R$,
and we denote the corresponding $\R$-subtori of $\TT$ by $\TT_0$ and $\TT_1$.
We have similar decompositions for $S$, $\overline{S}$, $T^\sss$, $T^\ssc$, and $T^\ad$.

Note that $\0\TT\theta$ is a \emph{fundamental torus} in $\0\GG\theta$, that is, $\0\TT\theta_0=\TT_0$ is a maximal compact torus and $\0\TT\theta$ is the centralizer of $\0\TT\theta_0$ in $\0\GG\theta$. Note also that $\0\TT\theta_1$ is a split $\R$-torus.

We set
\[X_0=\X^*(T_0), \quad \Lambda_0=\X^*(S_0),\quad M_0=\X^*(\ov S_0).\]
Then  $X_0$ is the restriction of $X$ to $T_0$,
$\Lambda_0$ is the restriction of $\Lambda$ to $S_0$,
and $M_0$ is the restriction of $M$ to $\ov S_0$.
Note that $X_0$, $\Lambda_0$, and $M_0$ are identified
with the images of $X$, $\Lambda$, and $M$, respectively,
under the canonical projection $\tl^*\to\tl_0^*$; see \cite[7.3]{BT}.

We also consider the dual lattices $X_0^\vee=\X_*(T_0)$, $\Lambda_0^\vee=\X_*(S_0)$,
and $M_0^\vee=\X_*(\ov S_0)$, which we regard as subgroups of $\tl$.
We have
\[X_0^\vee=(X^\vee)^\tau=\big\{\nu\in X^\vee \ \big|\  \tau(\nu)=\nu\big\},\]
and similarly for $\Lambda_0^\vee$ and $M_0^\vee$;
they are the intersections  with $\tl_0$
of $X^\vee$, $\Lambda^\vee$, and $M^\vee$, respectively.
Finally, let $\wt X_0^\vee$, $\wt\Lambda_0^\vee$, and $\wt M_0^\vee$
denote the images of $X^\vee$, $\Lambda^\vee$, and $M^\vee$, respectively,
under the canonical projection $\tl\to\tl_0,\ \nu\mapsto \big(\nu+\tau(\nu)\big)/2$;
see \cite[7.4]{BT}. Note that
$$
X_0^\vee\subseteq \wt X_0^\vee \subseteq \half X_0^\vee,
$$
and similarly for $\wt\Lambda_0^\vee$, $\wt M_0^\vee$.
There are chains of inclusions:
\begin{gather*}
P_0\oplus\Lambda_0 \supseteq X_0 \supseteq Q_0\oplus M_0, \\
Q_0^\vee\oplus\Lambda_0^\vee \subseteq X_0^\vee \subseteq P_0^\vee\oplus M_0^\vee, \\
\wt Q_0^\vee\oplus\wt\Lambda_0^\vee \subseteq \wt X_0^\vee
            \subseteq \wt P_0^\vee\oplus \wt M_0^\vee.
\end{gather*}

Write
\[_\theta\GG^\ssc=\hs_\theta\GG^{(1)}\times\dots\times \hs_\theta\GG^{(s)},\]
where each $_\theta\GG^{(k)}$ is $\R$-simple. Let
\[\Dtil=\Dtil^{(1)}\sqcup\dots\sqcup\Dtil^{(s)}\]
denote the affine Dynkin diagram of $(\hs_\theta\GG,\hs_\theta\TT,B)$;
see \cite[Section 12]{BT}.
By abuse of notation we write $\beta\in\Dtil$ if $\beta $ is a vertex of $\Dtil$.
The affine Dynkin diagram $\Dtil$ comes
with a family of positive integers $m_\beta$ for $\beta\in\Dtil$;
see \cite[Sections 9, 10, and 11]{BT}.

\begin{definition}[\cite{BT}]
\label{ss:Kac-labeling}
A {\em Kac labeling} of $\Dtil$ is a family of nonnegative integer  numerical labels
$\pp=(p_\beta)_{\beta\in\Dtil}$
at the vertices $\beta$ of $\Dtil$ satisfying
\begin{equation}\label{e:vk-p-0}
  \sum_{\beta\in\Dtil^{(k)}}\!\!\! m_\beta p_\beta=2\ \ \text{for each }k=1,\dots,s.
\end{equation}
We denote the set of Kac labelings of $\Dtil$ by $\Km(\Dtil)$.
\end{definition}

To any $x\in\ttl_0^\sss(\R)$, we assign a family $\pp=\pp(x)=(p_\beta)_{\beta\in \Dtil}$
of \emph{real} numbers $p_\beta=2x_\beta$, where the real numbers $x_\beta$
are the barycentric coordinates of $x$ defined in \cite[Sections 9.3, 10.3, 11.3]{BT}.
This correspondence identifies $\ttl_0^\sss(\R)$ with the subspace of $\R^{\Dtil}$
defined by the equations \eqref{e:vk-p-0}.
We denote the inverse correspondence by $\pp\,\mapsto\, x=x(\pp)$.
Let $\Delta=\Delta^{(1)}\times\dots\times\Delta^{(s)}$
denote the fundamental domain
for the reflection group $\Wtil^\ssc=\wt Q_0^\vee\rtimes W_0$
acting in $\ttl_0(\R)$, where $W_0=\Nm(T_0)/T$
and $\Nm(T_0)$ is the normalizer of $T_0$ in $G$;
see  \cite[Sections 9.3, 10.3, 11.3, and 12.2]{BT} for the description of $\Delta$.
In particular, all $\Delta^{(k)}$ are simplices, and
$x\in\Delta$ if and only if $p_\beta\ge0$ for all $\beta\in\Dtil$.

\subsec{}\label{ss:RG}
Consider the set $\Km(\Dtil)$ of Kac labelings
of the affine Dynkin diagram $\Dtil$.
For any $\pp\in\Km(\Dtil)$
consider the associated point $x(\pp)\in\ihalf P_0^\vee\subset\ttl^\sss_0(\R)$
in the fundamental polyhedron $\Delta\subset\ttl^\der_0(\R)$
of the reflection group $\Wtil^\ssc$;
see \cite[12.7]{BT}.
Consider the scaled exponential maps
$$
\Exp\colon\tl\to T
\quad\text{and}\quad
\Exp^\ad\colon\tl^\ad\to T^\ad,\quad
x\mapsto\exp2\pi x.
$$
Then $\Exp^\ad\big(x(\pp)\big)\in(T_0^\ad)^{(2)}$.
In particular, we may and shall assume that $\inn(t_\theta)=\Exp^\ad\big(x(\qq)\big)$
 and $t_\theta=\Exp\big(x(\qq)\big)$ for some $\qq\in\Km(\Dtil)$; see \cite[12.12]{BT}.
We write
\[\RG(\Bm,\tau,q)=\0\GG\theta\] (the real reductive group
corresponding to $\Bm$, $\tau$, and $q$).
Recall that $\Bm$ is  the based root datum of $(G,T,B)$.

For any  $\mm\in M_0^\vee$,
we write $y(\mm)=\frac\ii2\mm\in\ihalf M_0^\vee\subset\ssl_0(\R)$.
Set
$$
\nu_{\pp,\qq,\mm}=\tfrac2\ii\big(x(\pp)-x(\qq)+y(\mm)\big)\in P_0^\vee\oplus M_0^\vee.
$$

Recall that in \cite[Section 12.7]{BT} we defined a pairing
\[ \langle\,,\rangle_P\colon P_0\times \Km(\Dtil)\to \Q,\quad
            (\lambda,p)\mapsto \langle\lambda,p\rangle_P\coloneqq
            \sum c_\beta p_\beta\ \text{ for }
            \lambda=\sum c_\beta \beta\text{ with }c_\beta\in\Q,\]
where $\beta$ runs over the set of restricted simple roots $\ov\Sm\subset Q_0$.
This pairing induces a well-defined pairing
\[P_0/Q_0\times \Km(\Dtil)\to \Q/\Z.\]
Furthermore, we have a canonical pairing
\[\langle\,,\rangle_\Lambda\colon \Lambda_0\times M_0^\vee\to\C,\quad
       (\lambda, m)\mapsto \langle\lambda,m\rangle_\Lambda
       \  \text{ for }\lambda\in\Lambda_0,\ m\in M^\vee_0,\]
the restriction of the canonical pairing $\sll_0^*\times\sll_0\to\C$.
Since $\Lambda_0\subset M_0\otimes_\Z\Q$, the pairing
$\langle\,,\rangle_\Lambda$ takes values in $\Q$.
If  $\lambda\in M_0$ or $m\in 2\wt\Lambda^\vee_0\subseteq\Lambda_0^\vee$,
then $\langle\lambda,m\rangle_\Lambda\in \Z$.
We see that the pairing $\langle\,,\rangle_\Lambda$ induces a well-defined pairing
\[\Lambda_0/M_0\hs\times  M_0^\vee/2\wt\Lambda_0^\vee\to\Q/\Z.\]

Now if  $\lambda\in X_0\subseteq P_0\oplus\Lambda_0$\hs,
we write $\lambda=\lambda_P+\lambda_\Lambda$ with
$\lambda_P\in P_0,\ \lambda_\Lambda\in\Lambda_0$.

\begin{notation}\label{n:Km-red}
We define the set of {\em reductive Kac labelings}  $\Km(\Dtil,\Lambda,X,\tau,\qq)$
to be the subset of $\Km(\Dtil)\times M_0^\vee/2\wt\Lambda_0^\vee$
consisting of all pairs $\big(\pp,[\mm]\big)$
(with $m\in M_0^\vee$) satisfying
\begin{equation}\label{e:pr=q}
\langle\lambda_P,\pp\rangle_P+\langle\lambda_\Lambda,\mm\rangle_\Lambda\equiv
      \langle\lambda_P,\qq\rangle_P\!\!\pmod{\Z}\quad
      \text{ for all }\ [\lambda]\in X_0/(Q_0\oplus M_0).
\end{equation}
If $\lambda\in Q_0\oplus M_0$, then the congruence \eqref{e:pr=q}
is satisfied for any $\pp,\qq,\mm$,
because $\langle\lambda_P,\pp\rangle_P$,  $\langle\lambda_P,\qq\rangle_P$, and
$\langle\lambda_\Lambda,\mm\rangle_\Lambda$ are integers in this case.
\end{notation}

The finite abelian  group
$$
F_0=\wt X^\vee_0/(\wt Q^\vee_0\oplus\wt\Lambda^\vee_0)\subseteq
\wt P^\vee_0/\wt Q^\vee_0\oplus\wt M^\vee_0/\wt\Lambda^\vee_0
$$
acts diagonally on $\Km(\Dtil)\times M_0^\vee/2\wt\Lambda_0^\vee$,
where the action on $\Km(\Dtil)$ is induced
by the action of $\wt P^\vee_0/\wt Q^\vee_0$ via automorphisms
of the diagram $\Dtil$ described in \cite[\S12]{BT}
and the action on $M_0^\vee/2\wt\Lambda_0^\vee$
is induced by the translation action via the homomorphism
\[\wt M_0^\vee/\wt\Lambda_0^\vee\to M_0^\vee/2\wt\Lambda_0^\vee,
                 \quad  m+\wt\Lambda_0^\vee\,\mapsto\, 2m+2\wt\Lambda_0^\vee
                 \, \in 2\wt M_0^\vee/2\wt\Lambda_0^\vee\subseteq  M_0^\vee/2\wt\Lambda_0^\vee.\]

\begin{theorem}\label{t:H1}
The group $F_0$, when acting on $\Km(\Dtil)\times M_0^\vee/2\wt\Lambda_0^\vee$, preserves
the set of reductive Kac labelings  $\Km(\Dtil,\Lambda,X,\tau,\qq)$.
For $\pp\in\Km(\Dtil),\ m\in M_0^\vee$,
we have $\big(\pp,[\mm]\big)\in\Km(\Dtil,\Lambda,X,\tau,\qq)$
if and only if $\nu_{\pp,\qq,\mm}\in X_0^\vee$.
The map
\begin{gather}
\kappa\colon\,\Km(\Dtil,\Lambda,X,\tau,\qq)\, \longrightarrow\,
          T_0^{(2)} \subset \Zl^1(\R,\0\GG\theta), \notag \\
\big(\pp,[\mm]\big) \mapsto \exp2\pi\big(x(\pp)-x(\qq)+y(\mm)\big)
         =\nu_{\pp,\qq,\mm}(-1) \label{e:mainformula-H1}
\end{gather}
is well defined and induces a bijection
\[\kappa_*\colon\, \Km(\Dtil,\Lambda,X,\tau,\qq)/F_0\hs
          \longisoto\, \Ho^1(\R,\0\GG\theta)\]
between the set of $F_0$-orbits in $\Km(\Dtil,\Lambda,X,\tau,\qq)$
and the first Galois cohomology set $\Ho^1(\R,\0\GG\theta)$.
\end{theorem}

\begin{proof}
By \cite[Prop.\,5.6]{BT}, the inclusion $T_0^{(2)} \subset \Zl^1(\R,\0\GG\theta)$
induces a bijection between $H^1(\R,\0\GG\theta)$
and the orbit set $T_0^{(2)}/N_\tau$ for the group $N_\tau\subset\Nm(T_0)$
acting on $T_0^{(2)}$ by twisted conjugation;
see \cite[Sections 5.1, 5.2, and (4.4)]{BT}.
The twisted conjugation action of $N_\tau$ preserves
the set $\TT_0(\R)$ containing $T_0^{(2)}$; see \cite[Lemma 5.2(iii)]{BT}.

We consider the semidirect product $G\rtimes\langle\tauhat\rangle$,
where $\langle\tauhat\rangle$ is the group of order 1 or 2
acting faithfully on $G$ by conjugation so that $\tauhat$ acts via $\tau$.
The map
\[ \TT_0(\R)\to\TT_0(\R)\cdot\tauhat\subseteq G\cdot\tauhat,
      \quad t\mapsto t\hs t_\theta\cdot\tauhat\]
is  an $N_\tau$-equivariant bijection,
where $N_\tau$ acts on $\TT_0(\R)\cdot\tauhat$ by usual conjugation,
and the subset $T_0^{(2)}$ maps bijectively onto
\[(T_0\cdot\tauhat)^{(2)}_z\coloneqq \big\{g\in T_0\cdot\tauhat\ \big|\ g^2=z \big\},\]
where $z=t_{\theta}^2$; see \cite[Lemma 6.4]{BT}.
The conjugation action of $N_\tau$ on $\TT_0(\R)\cdot\tauhat$ factors
through an effective action of $\What=N_\tau/T^\tau\cong(T_0\cap T_1)\rtimes W_0$,
where $(T_0\cap T_1)$ acts by translations;
see \cite[Sections 8.3, 8.4, and 8.9]{BT}.

Consider the orbit set $\ttl_0(\R)/\Wtil$,
where $\Wtil=\wt X_0^\vee\rtimes W_0$ acts on $\ttl_0(\R)$
by affine isometries in a natural way (see \cite[Section 7.14]{BT}).
By \cite[Lemma 8.13]{BT}, the shifted exponential map
\[\Exphat:\ttl_0(\R)\to\TT_0(\R)\cdot\tauhat,\quad x\mapsto\exp2\pi{x}\cdot\tauhat\]
induces a bijection between the orbit sets  $\ttl_0(\R)/\Wtil$ and
$(\TT_0(\R)\cdot\tauhat)/\What$.

We conclude that the map
\[ \ttl_0(\R)\to\TT_0(\R),\quad x\mapsto tt_{\theta}^{-1},\qquad \text{where}\ \,t=\Exp(x), \]
induces a bijection between the orbit sets $\ttl_0(\R)/\Wtil$ and $\TT_0(\R)/N_\tau$.
Our aim is to identify the subset of $\ttl_0(\R)/\Wtil$
corresponding to $T_0^{(2)}/N_\tau\cong H^1(\R,\0\GG\theta)$
and provide explicit orbit representatives.

Consider the normal subgroup
$$\Wtil^\ssc\times\wt\Lambda^\vee_0=(\wt Q^\vee_0\oplus\wt\Lambda^\vee_0)\rtimes W_0
  \ \subseteq\ \wt X_0^\vee\rtimes W_0=\Wtil\hs$$
and the action of $\Wtil/(\Wtil^\ssc\times\wt\Lambda^\vee_0)\cong F_0$
on $\Delta\times\ssl_0(\R)/\ii\wt\Lambda^\vee_0$, where the group
$F_0\subseteq\ (\wt P_0^\vee/\wt Q_0^\vee)\oplus(\wt M_0^\vee/\wt \Lambda_0^\vee)$
acts on $\Delta$ via the action of $\wt P^\vee_0/\wt Q^\vee_0$
described in \cite[\S12]{BT}, and on $\ssl_0(\R)/\ii\wt\Lambda^\vee_0$
via the translation action of $\ii\wt M^\vee_0/\ii\wt\Lambda^\vee_0$.
The set of orbits of
$\Wtil^\ssc\times\wt\Lambda^\vee_0$ acting on $\ttl_0(\R)$
is identified with $\Delta\times\ssl_0(\R)/\ii\wt\Lambda^\vee_0$,
and the inclusion map
\[\Delta\times\ssl_0(\R)\into \ttl_0(\R)\]
induces a bijection between
the set of orbits of the group
$F_0$ in  $\Delta\times\ssl_0(\R)/\ii\wt\Lambda^\vee_0$
and the  set of orbits of $\Wtil$ in $\ttl_0(\R)$.
We obtain a composite  bijection
\[(\Delta\times \ssl_0(\R)/\ii\wt\Lambda^\vee_0)/F_0\longisoto \ttl_0(\R)/\Wtil
      \longisoto \TT_0(\R)/N_\tau\hs.\]

We see that every $\Wtil$-orbit in $\ttl_0(\R)$ is represented by a vector $x=x'+y$,
where $x'\in\Delta$ and $y\in\ssl_0(\R)$.
The orbit of $x$ corresponds to a cohomology class in $H^1(\R,\0\GG\theta)$
if and only if $tt_{\theta}^{-1}=\Exp\big(x-x(\qq)\big)\in T_0^{(2)}$.
This condition reads as $x-x(\qq)\in\ihalf X_0^\vee$ or, equivalently, as
\begin{equation}\label{e:x=x(q)}
\lambda(x)\equiv\lambda\big(x(\qq)\big)\!\!\!\pmod{\ihalf\Z}\ \ \text{for all }\lambda\in X_0\hs.
\end{equation}

Assume that \eqref{e:x=x(q)} is satisfied.
Since $t_\theta^2\in Z(G)$, we have $\lambda\big(x(q)\big)\in \ihalf\Z$ for all $\lambda\in Q_0$.
We see that for all $\lambda\in Q_0\subseteq X_0$ we have
\begin{equation}\label{e:x-x'-ihalfZ}
\lambda(x')=\lambda(x)\in\ihalf\Z.
\end{equation}
Let $(x_\beta)$ for $\beta\in\Dtil$ denote the barycentric coordinates of $x'$,
and write $p_\beta=2 x_\beta$.
Then from \eqref{e:x-x'-ihalfZ} and the definitions of the barycentric coordinates
in \cite[Sections 9.3, 10.3, and 11.3]{BT} it follows  that all $p_\beta$ are integers.
Since $x\in \Delta$, the numbers $p_\beta$ are nonnegative.
It follows from \cite[Section 12.7]{BT} that the $p_\beta$ satisfy \eqref{e:vk-p-0}.
Thus $p=(p_\beta)\in \Km(\Dtil)$ and $x=x(p)$.
Similarly, we see that for all  $\lambda\in M_0\subseteq X_0$
we have  $\lambda(y)=\lambda(x)\in\ihalf\Z$,
whence $y\in\ihalf M_0^\vee$ and therefore $y=y(\mm)$
for some $\mm\in M_0^\vee$.

Conversely, if $p\in\Km(\Dtil)$, $\mm\in M_0^\vee$,
$x=x(p)+y(\mm)$, and $t=\Exp(x)$,
then $x(p)\in \Delta$ and
\begin{equation*}
\lambda(x)\equiv\lambda\big(x(\qq)\big)\!\!\!\pmod{\ihalf\Z}\quad
\text{for all }\lambda\in Q_0\oplus M_0\hs.
\end{equation*}
Now we see that \eqref{e:x=x(q)} is equivalent to \eqref{e:pr=q},
that is, $tt_{\theta}^{-1}\in T_0^{(2)}$
if and only if $\big(p,[\mm]\big)\in\Km(\Dtil,\Lambda,X,\tau,\qq)$.

The congruences \eqref{e:x=x(q)} can also be written as
$$\lambda\big(x(p)-x(q)+y(\mm)\big)\in\ihalf\Z,\ \text{ that is, }\
    \lambda(\nu_{\pp,\qq,\mm})\in\Z \ \text{ for all }\,\lambda\in X_0\hs.$$
Hence $\big(p,[m]\big)\in\Km(\Dtil,\Lambda,X,\tau,\qq)$ if and only if
$\nu_{\pp,\qq,\mm}\in X_0^\vee$,
and the cocycle in $ T_0^{(2)}\subset\Zl^1\0\GG\theta$ corresponding to $\big(p,[m]\big)$
is
$$
tt_{\theta}^{-1}=\Exp\big(x(\pp)-x(\qq)+y(\mm)\big)=
     \exp\pi\ii\hs\nu_{\pp,\qq,\mm}=\nu_{\pp,\qq,\mm}(-1),
$$
given by formula \eqref{e:mainformula-H1},
where the last equality follows from \cite[(7.2)]{BT}.

We conclude that the Galois cohomology classes in $H^1(\R,\0\GG\theta)$
are represented by the elements $\nu_{\pp,\qq,\mm}(-1)\in T_0^{(2)}$
with $\big(p,[m]\big)\in\Km(\Dtil,\Lambda,X,\tau,\qq)$
defined up to the action of $F_0$ and this correspondence is bijective.
This completes the proof of Theorem \ref{t:H1}.
\end{proof}

\section{Additional structures on Galois cohomology of a reductive group}
\label{s:Additional}

Starting from this section, $\GG$ is a connected reductive $\R$-group, not necessarily compact.

\subsec{}
Let
\[\varphi\colon \GG'\to\GG''\]
be a {\em normal} homomorphism of connected reductive $\R$-groups.
Here ``normal'' means that the image  $\im\varphi$ is normal in $\GG''$.
We write $\HH=\im\varphi$.
We wish to describe the induced map on Galois cohomology.

The normal homomorphism $\varphi$ induces a homomorphism
\[\varphi^\ad\colon\GG^{\prm \ad}\to\GG^{\dprm\ad}\]
with normal image $\HH^\ad$.
Clearly, $\HH^\ad$ is the direct factor of both
$\GG^{\prm \ad}$ and $\GG^{\dprm\ad}$.

Consider the affine Dynkin diagrams $\Dtil'=\Dtil(\GG')$,
$\Dtil''=\Dtil(\GG'')$ and $\Dtil=\Dtil(\HH)$.
Then $\Dtil$ naturally embeds into $\Dtil'$ and into $\Dtil''$.
Let $\qq\in\Km(\Dtil)$ denote a Kac labeling defining
the $\R$-structure of $\HH$ (note that $q$ is defined not uniquely).
Then we may choose Kac labelings $\qq'\in\Km(\Dtil')$
defining the $\R$-structure of $\GG'$ and
$\qq''\in\Km(\Dtil'')$ defining the $\R$-structure of $\GG''$
such that $\qq'|_\Dtil=\qq=\qq''|_\Dtil$.
We may write $\GG'=\RG(\Bm',\tau',q')$,
$\GG''=\RG(\Bm'',\tau'',q'')$, $\HH=\RG(\Bm,\tau,q)$
with the notation of Subsection \ref{ss:RG}.

\begin{proposition}\label{p:H1-funct}
Let $\varphi\colon \GG'\to\GG''$ be a {\emm normal}
homomorphism of reductive $\R$-groups, not necessarily compact,
with image $\HH=\im \varphi$.
Let $\qq\in\Km(\Dtil)$,  $\qq'\in\Km(\Dtil')$, and $\qq''\in\Km(\Dtil'')$ be as above.
Let
\[\big(\pp',[\mm']\big)\in\Km(\Dtil',\Lambda',X',\tau',\qq'),
\ \ \nu_{\pp'\!,\qq'\!,\mm'}(-1)\in (T'_0)^{(2)}\subset \Zl^1\hs\GG',
\ \ \text{and}\ \ \big[\nu_{\pp'\!,\qq'\!,\mm'}(-1)\big]\in \Ho^1\hs\GG' \]
be as in Theorem \ref{t:H1}.
Then
\[\varphi\big(\nu_{\pp'\!,\qq'\!,\mm'}(-1)\big)=\nu_{\pp''\hm,\qq''\hm,\mm''}(-1),\]
and hence
\[\varphi_*\big[\nu_{\pp'\!,\qq'\!,\mm'}(-1)\big]
      =\big[\nu_{\pp''\hm,\qq''\hm,\mm''}(-1)\big]\in\Ho^1\hs\GG'',\]
where $\mm''=\varphi_*(\mm')$ and where $\pp''\in\Km(\Dtil'')$ is such that
\[\pp''|_\Dtil= \pp'|_\Dtil\quad\text{and}\quad \pp''|_{\Dtil''
    \smallsetminus \Dtil}=\qq''|_{\Dtil''\smallsetminus \Dtil}\hs.\]
\end{proposition}

\begin{proof} A straightforward check.\end{proof}

\begin{proposition}\label{p:twisting-ad}
Consider a reductive $\R$-group $\RG(\Bm,\tau,q)$
with the notation of  Subsection~\ref{ss:RG}.
Having fixed $\Bm$ and $\tau$, we shall write $\GG_q=\RG(\Bm,\tau,q)$ for brevity.
We use the notation of Theorem \ref{t:H1}.
Let $q'\in \Km(\Dtil)$, and consider the 1-cocycle
\[ a=\Exp^\ad\big(x(q')-x(q)\big)\in (T_0^\ad)^{(2)}\subset \Zl^1\hs\GG^\ad_q\hs.\]
Consider the twisted group $_a\GG_q$; cf. \cite[Section 14.1]{BT}.
Then there is a canonical isomorphism $\0\GG{a}_q\isoto \GG_{q'}$.
\end{proposition}

\begin{proof} Similar to that of \cite[Proposition 14.2]{BT}. \end{proof}

\subsec{}
Now let $\big(\qq',[m']\big)\in \Km(\Dtil,\Lambda,X,\tau,\qq)$,  that is,
$\qq'\in\Km(\Dtil)$,  $m'\in M_0^\vee$, $[m']\in M_0^\vee/2\wt\Lambda_0^\vee$,
and \eqref{e:pr=q} is satisfied.
With the notation of Theorem \ref{t:H1}, consider the 1-cocycle
\[a=\nu_{\qq'\!,\qq,m'}(-1)=\Exp\bigl(x(\qq')-x(\qq)+y(m')\bigr)
      \in  T_0^{(2)}\subset \Zl^1\hs\GG_\qq\]
and the twisting bijection
$\tw_a\colon \Ho^1\hs\0{\GG_\qq}a\to \Ho^1\hs\GG_\qq$
of Serre \cite[I.5.3, Proposition 35  bis]{Serre}.
By Proposition \ref{p:twisting-ad} we may identify $\0{\GG_\qq}a$ with $\GG_{\qq'}$.
Thus we obtain a map
\begin{equation*}
\tw_a\colon \Ho^1\hs\GG_{\qq'}\to \Ho^1\hs\GG_\qq
\end{equation*}
sending the neutral cohomology class $[1]\in\Ho^1\hs\GG_{\qq'}$  to $[a]\in \Ho^1\hs\GG_\qq$.

\begin{proposition}\label{p:inner-twist}
$$
\tw_a\big[\nu_{\pp'',\qq',m''}(-1)\big]=\big[\nu_{\pp'',\qq,\hs m''+m'}(-1)\big]\qquad
              \text{for all}\ \ \big(\pp'',[m'']\big)\hs\in\hs
              \Km(\Dtil,\Lambda,X,\tau,\qq').
$$
\end{proposition}

\begin{proof}
Note that $\nu_{\pp'',\qq,\hs m''+m'}=\nu_{\pp'',\qq',m''}+\nu_{\qq',\qq,\hs m'}$.
The map $\tw_a$ is induced by the map on cocycles
\[\Zl^1\hs\GG_{\qq'}\to \Zl^1\hs\GG_\qq,\quad a''\mapsto a''a,\]
sending $\nu_{\pp'',\qq',m''}(-1)$ to
\begin{equation*}
\nu_{\pp'',\qq',m''}(-1)\cdot a= \nu_{\pp'',\qq',m''}(-1)\cdot\nu_{\qq',\qq,\hs m'}(-1)
       = \nu_{\pp'',\qq,\hs m''+m'}(-1).\qedhere
\end{equation*}
\end{proof}

\subsec{}\label{ss:Z-action}
Let $\GG=\RG(\Bm,\tau,q)$
with the notation of Subsection \ref{ss:RG}.
We consider the center $\ZZ=Z(\GG)$ of $\GG$.
The group $\Hon\ZZ$
naturally acts on $\Hon\GG$ by
\begin{equation}\label{e:Z-action}
[z]\cdot [g]=[zg] \ \ \text{for}\ \,z\in\Zl^1\hs\ZZ,
\ \, g\in\Zl^1\hs\GG;
\end{equation}
see Serre \cite[Section I.5.7]{Serre}.
We wish to compute this action in our language.

\begin{lemma}\label{l:cocycle-center}
Let $\zeta\in\Ho^1\hs\ZZ$.
Then $\zeta$ can be represented by a cocycle of the form
\begin{equation}\label{e:nuP-nuM}
z=\Exp^\sss(\ii\nu_P)\cdot \Exp_S(\ii\nu_M/2),
\end{equation}
where $\nu_P\in P^\vee$, $\nu_M\in M_0^\vee$,
and the maps $\Exp^\sss\colon \tl^\sss\to T^\sss$, $\Exp_S\colon\sll\to S$
are the restrictions of $\Exp\colon \tl\to T$,
$x\mapsto \exp 2\pi x$, to $\tl^\sss$ and $\sll$, respectively.
\end{lemma}

\begin{proof}
The class $\zeta$ is represented by a cocycle $z=z^\sss\cdot s$, where $z^\sss\in Z(G^\sss)$, $s\in S$.
Then $z^\sss=\Exp^\sss(\ii\nu_P)$, $s=\Exp_S(y)$, where $\nu_P\in P^\vee$, $y\in\sll$.
The cocycle condition reads as
$$
z\cdot\hm\upgam z=\Exp^\sss\big(\ii\nu_P+\ii\tau(\nu_P)\big)\cdot\Exp_S(y+\upgam y)=1,
$$
which is equivalent to
\[\ii\nu_P+\ii\tau(\nu_P)+y+\upgam y\in\ii X_0^\vee.\]
This implies $y+\upgam y=\ii\nu_M\in\ii M_0^\vee$.
Put $y=y_++y_-$, where $\upgam y_+=y_+$ and $\upgam y_-=-y_-$\hs.
Then
$$
s_-\coloneqq\Exp_S(y_-)=\Exp_S(y_-/2)\cdot\hm\upgam\Exp_S(y_-/2)^{-1}
$$
is a coboundary in $S$. Replacing $z$ with $zs_-^{-1}$ yields $y=y_+=\ii\nu_M/2$.
\end{proof}

\subsec{}
Since $\X_*(\TT)=X^\vee$, $\X_*(\TT^\ad)=P^\vee$, and $\TT/\ZZ=\TT^\ad$,
by Theorem \ref{t:quasi} there is a canonical isomorphism
$$\vt^0\colon\hs\H^0(X^\vee\to P^\vee)\isoto\Ho^1\hs\ZZ.$$

\begin{lemma}\label{l:hyper-to-center}
Let $(\nu,\nu')\in\Zl^0(X^\vee\to P^\vee)$, and let $\zeta=\vt^0[\nu,\nu']\in \Ho^1\hs\ZZ$.
Then  $\zeta$ can be represented by the cocycle $z$ of the form \eqref{e:nuP-nuM},
where $\nu_P=-\nu'\in P^\vee$ and $\nu_M$ is the image of
$\nu\in X_0^\vee\subseteq P_0^\vee\oplus M_0^\vee$ under the projection to $M_0^\vee$.
\end{lemma}

\begin{proof}
Since $(\nu,\nu')$ is a $0$-hypercocycle, we have
$$D^0(\nu,\nu') =\big(-d^1\nu,d^0\nu'-\partial\nu\big)
    =\big(\tau(\nu)-\nu,-\tau(\nu')-\nu'-\nu+\nu_M\big)=(0,0),$$
where $\nu_M$ is the image of $\nu\in X^\vee\subseteq P^\vee\oplus M^\vee$
under the projection to $M^\vee$.
Hence $\nu\in X_0^\vee$, $\nu_M\in M_0^\vee$, and $\nu-\nu_M=-\nu'-\tau(\nu')\in P_0^\vee$.

As noted in Subsection~\ref{ss:quasi},
$\vt^0[\nu,\nu']$ is represented by a cocycle $z=\nu(-1)\cdot t \cdot\hm\upgam t^{-1}$,
where $t\in T$ and the image of $t$ in $T^\ad$ is
$\nu'(-1)=\Exp^\ad(\ii\nu'/2)=\Exp^\ad(-\ii\nu'/2)$.
We may take $t=\Exp^\sss(-\ii\nu'/2)$; then
$$z=\Exp(\ii\nu/2-\ii\nu'/2+\ii\tau(\nu')/2)
     =\Exp(-\ii\nu'+\ii\nu_M/2)=\Exp^\sss(-\ii\nu')\cdot \Exp_S(\ii\nu_M/2),$$
as required.
\end{proof}

\subsec{}
Let $\GG^\ssc$ be a simply connected semisimple $\R$-group.
We consider the center $\ZZ^\ssc=Z(\GG^\ssc)$ of $\GG^\ssc$.
Set $C=P^\vee/Q^\vee$.
We embed $P^\vee$ and $Q^\vee$ into $\tl^\ssc$.
The scaled exponential map
\[\Exp^\ssc\colon\tl^\ssc\to T^\ssc,\quad x\mapsto\exp 2\pi x \]
has kernel $\ii Q^\vee$ and induces
a $\Gamma$-equivariant isomorphism of abelian groups
\[ (\ii)C=\ii P^\vee\hm/\hs\ii Q^\vee\longisoto Z^\ssc\]
Using Lemma \ref{l:anti}, we obtain an  isomorphism  on cohomology
\begin{equation*}
 \Ho^0\hs C=\Ho^1\,(\ii)C\isoto \Hon\ZZ^\ssc.
\end{equation*}

In \cite[Section 7.14]{BT} we defined the group $C_0=\wt P^\vee_0/\wt Q^\vee_0$.
We have a canonical surjective homomorphism
\begin{equation}\label{e:C-C0}
C\onto C_0\colon\ \nu+ Q^\vee\hs \longmapsto\hs\half\big(\nu+\tau(\nu)\big)+\wt Q^\vee_0.
\end{equation}

\begin{remark}\label{r:F0|C}
There is a similar surjective homomorphism
$F=X^\vee/(Q^\vee\oplus\Lambda^\vee)\onto F_0$.
On the other hand, $F\subseteq C\oplus M^\vee/\Lambda^\vee$
embeds into $C$ under the natural projection.
Indeed, a coset $[\nu]\in F$ represented by $\nu\in X^\vee$ projects to $[0]\in C$
if and only if $\nu\in Q^\vee\oplus M^\vee$;
but then $\nu=\nu_Q+\nu_M$, where $\nu_Q\in Q^\vee$
and $\nu_M\in X^\vee\cap M^\vee=\Lambda^\vee$,
whence $[\nu]=[0]\in F$. We deduce that $\#F_0$ divides $\#F$ and $\#F$ divides $\#C$.
\end{remark}

\subsec{}
In \cite[Sections 9--11]{BT} we described a canonical action of the group $C_0$
on the twisted  affine Dynkin diagram $\Dtil$.
Thus we obtain a canonical homomorphism
\begin{equation}\label{e:vs}
\vs\colon P\to C\to C_0\to \Aut\hs\Dtil.
\end{equation}

Since $\tau=-\gamma$ on $P$, we have
\[\Bd^0\hs C=\big\{c+\upgam c\ \big|\  c\in C\big\}=\big\{c-\tau(c)\ \big|\ c\in C\big\}
     =\big\{\nu-\tau(\nu)+Q^\vee \ \big|\  \nu\in P^\vee\big\}.\]
We see that the canonical surjective map \eqref{e:C-C0}
sends $\Bd^0\hs C\subseteq C$ to $0$ and thus
induces a homomorphism $\Ho^0\hs C\to C_0$.
Thus we obtain a canonical homomorphism
\begin{equation}\label{e:vs-00}
\vs^0\colon \Ho^0\hs C\to C_0\to \Aut\hs\Dtil.
\end{equation}

\begin{proposition}\label{p:Z-action}
For $\GG$ as in \ref{ss:Z-action},
let $\xi\in\Hon\GG$, $\xi=\kappa_*\big[p,[m]\hs\big]$
with the notation of Theorem \ref{t:H1}.
Let $\zeta=[z]\in\Hon Z(\GG)$,
where $z=\Exp^\sss(\ii\nu_P)\cdot\Exp_S(\ii\nu_M/2)$, with
$\nu_P\in P^\vee$, $\nu_M\in M_0^\vee$  as in Lemma \ref{l:cocycle-center}.
Then
\[\zeta\cdot \xi\hs=\hs \kappa_*\big[p',[m']\hs\big],\]
where $m'=\nu_M+m$, $p'=\vs(\nu_P)(p)$,  and $\vs$ is the homomorphism of \eqref{e:vs}.
\end{proposition}

Proposition \ref{p:Z-action}  computes the action of $\Hon Z(\GG)$ on $\Hon \GG$.

\begin{proof}
The class $\zeta\cdot \xi$ is represented by the cocycle
$$z\cdot\nu_{p,q,m}(-1)=\Exp^\sss\big(\ii\nu_P+x(p)-x(q)\big)\cdot\Exp_S\big(\ii\nu_M/2+y(m)\big).$$
Put $\nu_P=\nu_0+\nu_1$, where
$\nu_0=\half\big(\nu_P+\tau(\nu_P)\big)\in\wt P^\vee_0$ and $\nu_1=\half\big(\nu_P-\tau(\nu_P)\big)$.
Note that $$\upgam\Exp(\ii\nu_1/2)\cdot\Exp(\ii\nu_P)\cdot\Exp(\ii\nu_1/2)^{-1}=\Exp(\ii\nu_0),$$
whence
$$\zeta\cdot \xi=\big[\Exp^\sss\big(\ii\nu_0+x(p)-x(q)\big)\cdot\Exp_S(y(\nu_M+m))\big]
   =\big[\Exp^\sss(x(p')-x(q))\cdot\Exp_S(y(m'))\big]=\big[\nu_{p',q,m'}(-1)\big]$$
by the definition of the action of $C_0$ on
the fundamental domain $\Delta$, as desired.
\end{proof}

\begin{corollary}
\label{c:Stab}
Let $\GG^\ssc$ be a simply connected semisimple $\R$-group with center $\ZZ^\ssc$.
Consider the map
\[\iota_*\colon  \Hon \ZZ^\ssc\to\Hon \GG^\ssc\]
induced by the inclusion map $\iota\colon \ZZ^\ssc\into \GG^\ssc$.
Let
\[\iota_*^0\colon \Ho^0\hs C\isoto \Hon \ZZ^\ssc\lra\Hon \GG^\ssc\]
denote the composite map.
Then
\[\ker \iota_*^0=(\Ho^0\hs C)_q,\]
the stabilizer of $q\in\Km(\Dtil)$ under the action of $\Ho^0\hs C$ on $\Km(\Dtil)$
induced by the action \eqref{e:vs-00} of $\Ho^0\hs C$ on $\Dtil$.
\end{corollary}

\begin{proof}
It follows from the definition of the action \eqref{e:Z-action}
of $\Hon \ZZ^\ssc$ on $\Hon\GG^\ssc$ that
\[\iota_*[z]=[z]\cdot[1]\in  \Hon \GG^\ssc.\]
Therefore, $\ker \iota_*=(\Hon \ZZ^\ssc)_{[1]}$
and $\ker\iota_*^0=(\Ho^0\hs C)_{[1]}$,
the stabilizers of  the neutral element  $[1]\in\Hon \GG^\ssc$,
where $\Ho^0\hs C$ acts on $\Hon \GG^\ssc$
via the isomorphism $\Ho^0\hs C=\Ho^1\,(\ii)C\isoto \Hon \ZZ^\ssc$.
Now the corollary follows from Proposition \ref{p:Z-action}.
\end{proof}

\subsec{}
Let $\GG$ be a connected reductive $\R$-group (not necessarily compact)
and let $\TT\subseteq \GG$ be a maximal torus.
Let $\TT^\ssc=\rho^{-1}(\TT)\subseteq \GG^\ssc$.
By \cite[Definition 2.2]{Borovoi-Memoir},
the $k$-th {\em abelian Galois cohomology group} is defined by
\[\Ho_\ab^k\hs\GG\coloneqq\H^k(\TT^\ssc\to \TT).\]
It is an abelian group, and it does not depend
on the choice of $\TT$ (up to a canonical isomorphism).
In  \cite[Section 3]{Borovoi-Memoir}, the first-named author
defined the {\em abelianization map}
\[\ab^1\colon \Ho^1\hs\GG\to\Ho_\ab^1\hs\GG.\]
The first abelian Galois cohomology group $\Ho_\ab^1\hs\GG$
and the abelianization map play an important role in the description
of Galois cohomology of reductive groups over number fields;
see \cite[Theorem 5.11]{Borovoi-Memoir}.
Here we compute $\Ho_\ab^k\hs\GG$ and the abelianization map.

\begin{proposition}\label{p:H1ab}
Let $\GG$ be a connected reductive $\R$-group.
Then for any $k \in \Z$, there is a canonical isomorphism $\Ho^k\hs\pia\GG\isoto \Ho^k_\ab\hs\GG$.
\end{proposition}

\begin{proof}
Let $\TT\subseteq\GG$ be a maximal torus.
Since the homomorphism
\[\rho_*\colon Q^\vee=\X_*(\TT^\ssc)\lra\X_*(\TT)=X^\vee\]
is injective and $\pia\GG=\coker\rho_*$, by Lemma~\ref{l:quasi-inj}
we have a canonical isomorphism
\[\Ho^k\hs\pia\GG\,\isoto\,\H^k\big(\X_*(\TT^\ssc)\to\X_*(\TT)\big).\]
By Proposition \ref{p:quasi-tori}
we have a canonical isomorphism
\[\H^k\big(\X_*(\TT^\ssc)\to\X_*(\TT)\big)\,\longisoto\,
      \H^k(\TT^\ssc\to\TT)=\Ho^k_\ab\hs\GG,\]
as required.
\end{proof}

\begin{theorem}\label{t:H1ab}
Let $\GG=\RG(\Bm,\tau,q)$ with the notation of Subsection \ref{ss:RG}
(it comes with a fundamental torus $\TT\subseteq \GG$).
With the notation of Theorem \ref{t:H1}, consider the map
\begin{gather*}
\lambda\colon\hs \Km(\Dtil,\Lambda,X,\tau,\qq)\,\lra\, \Hon\pia\GG\\
\big(\pp,[\mm]\big)\hs\longmapsto \hs[\nu_{\pp,\qq,\mm}+Q^\vee].
\end{gather*}
Then the map  $\lambda$ induces a map
\[\lambda_*\colon\hs  \Km(\Dtil,\Lambda,X,\tau,\qq)/F_0\, \lra\, \Ho^1\hs\pia\GG,\]
and the following diagram commutes:
\[
\xymatrix@C=15truemm{
\Km(\Dtil,\Lambda,X,\tau,\qq)/F_0\ar[r]^-{\lambda_*}\ar[d]_-{\kappa_*}  & \Ho^1\hs\pia\GG\ar[d]^-\cong\\
\Ho^1\hs\GG\ar[r]^-{\ab^1} & \Ho^1_\ab\GG
}
\]
where $\kappa_*$ is the bijection of Theorem \ref{t:H1}
and the right-hand vertical arrow is the isomorphism of Proposition \ref{p:H1ab}.
\end{theorem}

Theorem \ref{t:H1ab} computes the abelianization map $\ab^1$.

\begin{proof}
Consider the diagram
\begin{equation*}
\begin{aligned}
\xymatrix@C=15truemm{
\Km(\Dtil,\Lambda,X,\tau,\qq)\ar[r]^-\lambda\ar[d]_\kappa        &\Ho^1\hs\pia\GG\ar[d]^-\cong\\
\Ho^1\hs\TT\ar[r]^-{[t]\mapsto[1,t]}\ar[d]_{i_*}        &\H^1(\TT^\ssc\to\TT)\ar@{=}[d]  \\
\Ho^1\hs\GG\ar[r]^-{\ab^1}                       &\Ho^1_\ab\hs\GG
}
\end{aligned}
\end{equation*}
in which the map $\kappa$ is given by $\big(\pp,[\mm]\big)\mapsto \big[\nu_{p,q,m}(-1)\big]$
and the map $i_*$ is induced by the inclusion map $i\colon \TT\into\GG$.
In this diagram, the top rectangle clearly commutes,
and we know from the definition of the morphism of functors $\ab^1$
in terms of hypercohomology of crossed modules
in \cite[Section 3.10]{Borovoi-Memoir}
that the bottom rectangle commutes as well. Thus the diagram commutes.
Since the composite left-hand vertical arrow is constant on the orbits of $F_0$,
and the right-hand vertical arrows are isomorphisms,
we conclude that the map $\lambda$ is constant on the orbits of $F_0$,
which completes the proof of the theorem.
\end{proof}

\newcommand{\Xv}{X^\vee}
\newcommand{\Xsc}{Q^\vee}
\newcommand{\Xad}{P^\vee}

\section{The group of real connected components of a reductive group}
\label{s:pi0R}

In this section we compute the group of real connected components
$\piR\GG\coloneqq \pi_0\hs\GG(\R)$ for a connected reductive $\R$-group $\GG$
in terms of the algebraic fundamental group $\pia\GG$.

\subsec{}\label{ss:action}
Let $\TT\subseteq \GG$ be a maximal torus, not necessarily fundamental.
Recall that
\[\pia{\GG}=\Xv\!/\hs\Xsc,\quad\text{where}
    \quad \Xv=\X_*(\TT),\ \Xsc=X_*(\TT^\ssc).\]
Recall that
\[\GG^\ad=\GG/Z(\GG),\quad \TT^\ad=\TT/Z(\GG),\quad \Xad=\X_*(\TT^\ad),\quad C=\Xad/\Xsc.\]
The canonical homomorphism $\Ad\colon \GG\to\GG^\ad$ induces a homomorphism of $\Ga$-modules
$$ \pia\GG\to\pia\GG^\ad=C,$$
which in turn induces a homomorphism
\[ \Ad_*\colon \Ho^0\hs\pia\GG\to \Ho^0\hs C.\]
Consider the composite map
\begin{align*}
\phi\colon\, \Ho^0\hs\pia\GG\labeltoo{\Ad_*} \Ho^0\hs C
     \longisoto\Ho^1\hs\ZZ^\ssc\lra \Ho^1\hs\TT^\ssc \lra \Ho^1\hs\GG^\ssc,
\end{align*}
where the third and fourth arrows are induced
by the inclusion homomorphisms $\ZZ^\ssc\into\TT^\ssc\into\GG^\ssc$.
Explicitly:
\[\phi[\nu+\Xsc]=\big[\Exp^\ssc(\ii\nu^\ad)\big],\]
where $\nu^\ad\in\Xad\subset\tl^\ad=\tl^\ssc$ is the image of $\nu\in \Xv$
under the canonical homomorphism $\Xv\to\Xad$.
Let $(\Ho^0\hs\pia\hs\GG)_1$ denote the kernel of $\phi$,
that is, the preimage in $\Ho^0\hs\pia\GG$ of $[1]\in\Ho^1\hs\GG^\ssc$.

Now write $\GG=\RG(\Bm,\tau,q)$ with the notation of Subsection \ref{ss:RG}.
The group $ \Ho^0\hs C$ acts
on the affine Dynkin diagram $\Dtil$ of $\GG^\ssc$
via the homomorphism $\vs^0$ of  \eqref{e:vs-00},
and composing with $\Ad_*$ we obtain an action
of $\Ho^0\hs\pia\GG$ on $\Dtil$, and hence on $\Km(\Dtil)$.

\pagebreak[2]
\begin{theorem}\label{t:pi0R}\
\begin{enumerate}
\item[\rm(i)]
Let $\GG$ be a connected reductive $\R$-group.
Then the subset   $(\Ho^0\hs\pia\hs\GG)_1$
of the abelian group $\Ho^0\hs\pia\hs\GG$ is a subgroup,
and there exists a canonical isomorphism of abelian groups
\[\psi\colon(\Ho^0\hs \pia\GG)_1\longisoto\piR\GG.\]
\item[\rm(ii)]
Write $\GG=\RG(\Bm,\tau,q)$ with the notation of Subsection \ref{ss:RG}.
Then  $(\Ho^0\hs\pia\hs\GG)_1$ is the stabilizer of $q\in\Km(\Dtil)$
under the action of the abelian group $\Ho^0\hs \pia\GG$
on $\Km(\Dtil)$ described in Subsection \ref{ss:action}.
\end{enumerate}
\end{theorem}

The theorem will be proved in Subsections \ref{ss:Proof-pit} and \ref{ss:Proof-piR}.

\begin{example}
Let $\GG=\TT$ be an $\R$-torus.
Then Theorem \ref{t:pi0R} says that $\piR\TT\cong\Ho^0\hs\hs\X_*(\TT)$,
which is Corollary~\ref{c:H0-X(T)}.
\end{example}

\subsec{}\label{ss:map-ker-vs}
We specify the map $\psi$ in Theorem \ref{t:pi0R}(i).
Let $\nu\in \Xv$ be such that
\begin{align}
&\nu+\Xsc\in \Zl^0\hs\pia\GG,\label{e:a}\\
&[\nu+\Xsc]\in\ker\phi.\label{e:b}
\end{align}
Here \eqref{e:a} means that
\begin{equation}\label{e:c}
\upgam\nu-\nu=\nu^\ssc\quad \text{for some }\nu^\ssc\in\Xsc.
\end{equation}
Then  $\nu^\ssc\in \Zl^1\hs\Xsc$.
Set
\[t=\nu(-1)\in T\subset G\quad\text{and}\quad t^\ssc=
        \nu^\ssc(-1)\in \Zl^1\hs\TT^\ssc\subset\Zl^1\hs\GG^\ssc.\]

The following diagram commutes:
\[
\xymatrix@C=12mm{
\raise 3pt \hbox{$\Xsc$}\!\!\!\!
\ar@{_(->}[d]\ar@{^(->}[rd]\\
\Xv\ar[r]^-{\Ad_*}  &\Xad
}
\]
Now it follows from \eqref{e:c} that
\[\upgam\nu^\ad-\nu^\ad=\nu^\ssc,\]
where $\nu^\ad=\Ad_*(\nu)\in\Xad$.
Write $\nu^\ad=\nu_0+\nu_1$, where
$$\nu_0=\half(\nu^\ad-\upgam\nu^\ad)
\quad\text{and}\quad
\nu_1=\half(\nu^\ad+\upgam\nu^\ad).$$
Then $\upgam\nu_0=-\nu_0$, $\upgam\nu_1=\nu_1$, and $2\nu_0=-\nu^\ssc$.

By \eqref{e:b} we have $[\Exp^\ssc(\ii\nu^\ad)]=[1]\in H^1\hs\GG^\ssc$.
Since
\[\Exp^\ssc(\ii\nu_1/2)^{-1}\cdot\Exp^\ssc(\ii\nu^\ad)\cdot\upgam\Exp^\ssc(\ii\nu_1/2)
    =\Exp^\ssc(\ii\nu_0)=\exp(-\pi\ii\nu^\ssc)=t^\ssc,\]
we see that $[t^\ssc]=[1]\in H^1\hs\GG^\ssc$ and therefore
$t^\ssc=(g^\ssc)^{-1}\cdot\upgam g^\ssc$ for some $g^\ssc\in G^\ssc$.
We set
\begin{equation}\label{e:g}
g=\rho(g^\ssc)\cdot t^{-1}\in G.
\end{equation}

We compute $\upgam\hm g$.
By \eqref{e:c} we have $\upgam t=t\cdot\rho(t^\ssc)$.
By construction we have $\upgam\hm{g^\ssc}=g^\ssc\hs t^\ssc$. Thus
\[\upgam\hm g=\rho(\hs\upgam\hm{g^\ssc}\hs)\cdot\hm\upgam t^{-1}
    =\rho(g^\ssc\hs t^\ssc)\cdot\rho(t^\ssc)^{-1}\cdot t^{-1}
    =\rho(g^\ssc)\cdot t^{-1}=g.\]
We conclude that $g\in\GG(\R)$. We set
$\psi[\nu+\Xsc]=[g]\in \pi_0^\R\hs\GG$.

\subsec{}
Let $\pit{G}$ denote the topological fundamental group of $G=\GG(\C)$.
Recall that $\pit{G}$ is the group of equivalence classes of {\em loops},  that is, continuous maps
$l\colon[0,1]\to G$ from the segment $[0,1]$ to $G$ with $l(0)=l(1)=1_G$.
It is well known that $\pit{G}$ a finitely generated abelian group.
Since the Galois group $\Gamma$ acts on $\GG(\C)$ continuously in the usual Hausdorff topology,
it naturally acts on $\pit{G}$.
We usually  write $\pit\GG$ for $\pit{G}$ to stress that it is a $\Gamma$-module
with the $\Gamma$-action induced by the real form $\GG$ of $G$.

\begin{example}\label{ex:GmR}
Take $\GG=\GmR$.
We have a canonical isomorphism of $\Gamma$-modules
\[\ii\Z\labelto{\sim}\pit\hs\GmR=\pit\hs\C^\times,\quad\
   \ii\!\cdot\! n\, \longmapsto\, [t\mapsto\exp 2\pi\ii n t],\ \,n\in\Z,\ t\in[0,1].\]
\end{example}

\begin{proposition}[{\cite[Proposition 1.11]{Borovoi-Memoir}}]
\label{p:Memoir}
Let $\GG$ be a connected reductive $\R$-group.
With the above notation, there is a canonical isomorphism
\begin{align}\label{e:pia-pit}
(\ii)\pia\GG\cong\ii \Xv\!/\hs\ii \Xsc\ \longisoto\ \pit\GG.
\end{align}
\end{proposition}

\begin{proof}
We give an alternative proof of Proposition \ref{p:Memoir}.
In this proof we introduce notions that are necessary for the proof of Theorem \ref{t:pi0R},
and we give the explicit formula \eqref{e:pia-pit-expl-bis} for the isomorphism \eqref{e:pia-pit}.

Recall that $\ssl=\Lie(\SS)$, where $\SS=Z(\GG)^0$.
Consider the $\R$-group $\GGtil\coloneqq\GG^\ssc\times\ssl$,
the product of the simply connected $\R$-group $\GG^\ssc$
and the commutative unipotent $\R$-group $\ssl$.
Observe that the complex Lie group $\Gtil=\GGtil(\C)=G^\ssc\times\sll$ is simply connected.
Consider the surjective $\Gamma$-equivariant homomorphism of complex Lie groups
\[ \rtil\colon\, \Gtil=G^\ssc\times\sll\,\lra\, G^\sss\times S\,\to\, G,\quad \gtil=(g^\ssc,y)\longmapsto
     \rho(g^\ssc)\cdot\Exp_S(y)\ \ \text{for}\  g^\ssc\in G^\ssc,\ y\in\sll.\]
(Note that $\rtil$ is \emph{not} a homomorphism of  algebraic groups!) Since each of the homomorphisms
\[ G^\ssc\to G^\sss,\quad  \sll\to S,\quad G^\sss\times S\to G\]
has discrete kernel, the homomorphism $\rtil$ has discrete kernel as well.
We see that $\rtil$ is a universal covering of $G$.

By the exact homotopy sequence for $\rtil$, the group
$\wt{Z}\coloneqq\ker\rtil$ can be identified with $\pit\GG$;
see, for instance, \cite[Section 1.3, Theorem 4 and Corollary 2]{OV}.
This identification goes as follows. A loop $l\colon [0,1]\to G$
with $l(0)=l(1)=1_G$ can be uniquely lifted to a continuous path $\ltil\colon [0,1]\to \Gtil$
with $\ltil(0)=1_\Gtil$\hs. Then $\tilde\rho\big(\ltil(1)\big)=l(1)=1_G$\hs, whence $\ltil(1)\in\Ztil$.
To $[l]\in\pit \GG$ we associate $\ltil(1)\in\Ztil$.

On the other hand, $\wt{Z}$ is contained in $T^\ssc\times\sll$.
Let $\tilde{t}=(t^\ssc,y)\in T^\ssc\times\sll$,
$t^\ssc=\Exp^\ssc(x)$, $x\in\tl^\ssc$, $y\in\sll$. Then
we have $$\rtil(\tilde{t})=\rho(t^\ssc)\cdot\Exp_S(y)=\Exp^\sss(x)\cdot\Exp_S(y)=\Exp(x+y).$$
Hence $\tilde{t}\in\wt{Z}$ if and only $x+y\in\ii \Xv$. The homomorphism
$$\wt\Exp:\tl=\tl^\ssc\oplus\sll\to T^\ssc\times\sll,\qquad\wt\Exp(x+y)=(\Exp^\ssc(x),y)$$
is a universal covering and $\ker\wt\Exp=\ii\Xsc$.
Hence $\wt{Z}=\wt\Exp(\ii \Xv)\cong\ii \Xv/\ii\Xsc\cong(\ii)\pia\GG$.

We obtain a  $\Gamma$-equivariant isomorphism
\begin{align}
(\ii)\pia\GG\isoto&\wt{Z}\isoto\pit\GG \label{e:pia-pit-expl}  \\
\ii\nu+\ii \Xsc \longmapsto \wt\Exp&(\ii\nu)
    \longmapsto \big[t\mapsto\Exp(t\ii\nu)=\nu(\exp2\pi\ii t)\big]
    \quad  \text{for}\ \nu\in \Xv,\ t\in[0,1],\label{e:pia-pit-expl-bis}
\end{align}
as required. The explicit formula \eqref{e:pia-pit-expl-bis}
for the homomorphism \eqref{e:pia-pit-expl}
stems from an observation that the loop $l:t\mapsto\Exp(t\ii\nu)$ in $G$
lifts to the path $\ltil:t\mapsto\wt\Exp(t\ii\nu)$ in $\Gtil$
starting at $1_\Gtil$ and ending at $\wt\Exp(\ii\nu)$.
\end{proof}

\begin{corollary}
There is a canonical isomorphism
\begin{equation}\label{a:pi1}
\Ho^0\hs\pia\GG\longisoto \Ho^1\hs\pit\GG,\quad
\, [\nu+ \Xsc]\,\mapsto\hs  \big[t\mapsto\nu(\exp 2\pi\ii t)\big]\ \ \text{for}\ \nu\in \Xv.
\end{equation}
\end{corollary}

\begin{proof}
The corollary follows from \eqref{e:pia-pit-expl},
    \eqref{e:pia-pit-expl-bis}, and Corollary \ref{c:ii, k+1}.
\end{proof}

\begin{proposition}\label{t:sc-connected}
Let $\GG^\ssc$ be a {\emm simply connected} semisimple $\R$-group.
Then the real Lie group $\GG^\ssc(\R)$ is connected.
\end{proposition}

\begin{proof}
See Borel and Tits  \cite[Corollary 4.7]{Borel-Tits-C},
or Gorbatsevich, Onishchik and Vinberg  \cite[4.2.2, Theorem 2.2]{GOV},
or Platonov and Rapinchuk \cite[Proposition 7.6 on page 407]{PR}.
\end{proof}

\subsec{\it Proof of Theorem \ref{t:pi0R}(i).}
\label{ss:Proof-pit}
The identification $\pit\GG\cong\wt{Z}$ yields
a short exact sequence of $\Gamma$-groups
\[  1\to \pit\GG\labelto{i} \GGtil \labelto{\rtil} \GG\to 1,\]
which gives rise to a cohomology exact sequence
\[\wt\GG(\R)=\GG^\ssc(\R)\times\ssl(\R)\, \lra\, \GG(\R)\,\lra\, \Ho^1\hs\pit\GG\,
               \labelto{i_*}\, \Ho^1\hs\wt\GG=\Ho^1\hs\GG^\ssc\times\Ho^1\hs\ssl=\Ho^1\hs\GG^\ssc.\]
By Proposition \ref{t:sc-connected} the Lie group $\wt\GG(\R)$ is connected,
and its image in $\GG(\R)$ is connected and open, and hence it is the identity component.
Thus we obtain an injective homomorphism
\[\piR\GG=\coker\big[\wt\GG(\R)\to\GG(\R)\big]\,\into\, \Ho^1\hs\pit\GG\]
whose image is
\begin{equation}\label{e:ker-G-Gssc}
\ker\big[\Ho^1\hs\pit\GG\labelto{i_*} \Ho^1\hs\GG^\ssc\big].
\end{equation}
It follows that the kernel \eqref{e:ker-G-Gssc} is a subgroup.
Using the isomorphism \eqref{a:pi1}, we obtain a group isomorphism
\begin{equation}\label{e:piR-pia-H1}
\piR\GG\longisoto\ker\big[\Ho^0\hs\pia\GG \labelto{\phi} \Ho^1\hs\GG^\ssc\big]=(\Ho^0\hs\pia\GG)_1.
\end{equation}
Since $\Ho^0\hs\pia\GG$ is an abelian group killed by multiplication by 2,
we see that so is $\piR\GG$.

For any $g\in\GG(\R)$, the isomorphism \eqref{e:piR-pia-H1}
sends $[g]\in\piR\GG$ to the cohomology class in $\Ho^0\hs\pia\GG\cong\Ho^1\hs\pit\GG$
corresponding to the 1-cocycle $\tilde{z}=\tilde{g}\hs^{-1}\cdot\hm\upgam\tilde{g}\in\wt{Z}$,
where $\tilde{g}\in\Gtil$ is such that $\rtil(\tilde{g})=g$.
If $g$ is given by \eqref{e:g}, then,
with the notation of Subsection~\ref{ss:map-ker-vs},
we may take $\tilde{g}=g^\ssc \cdot\tilde{t}^{-1}$,
where $\tilde{t}=\wt\Exp(\ii\nu/2)\in T^\ssc\times\sll$.
We have
\[\tilde{z}=\tilde{t}\cdot(g^\ssc)^{-1}\cdot\hm\upgam\hm g^\ssc\cdot\hm\upgam\tilde{t}^{-1}
   =\tilde{t}\cdot t^\ssc\cdot\hm\upgam\tilde{t}^{-1}
   =\tilde{t}\cdot(t^\ssc)^{-1}\cdot\hm\upgam\tilde{t}^{-1}
   =\wt\Exp\big(\ii\nu/2-\ii\nu^\ssc/2+\ii\upgam\hm\nu/2\big)
   =\wt\Exp(\ii\nu),\]
where the last equality follows from \eqref{e:c}.
Then $\tilde z$ represents the cohomology class
$[\nu+\Xsc]\in\Ho^0\hs\pia\GG\cong\Ho^1\hs\pit\GG$.
Thus \eqref{e:piR-pia-H1} is the inverse of the map
$\psi$ of Subsection~\ref{ss:map-ker-vs},
which completes the proof of Theorem \ref{t:pi0R}(i).  \qed

\subsec{\it Proof of Theorem \ref{t:pi0R}(ii).}
\label{ss:Proof-piR}
Now we assume that the maximal torus $\TT$ is fundamental.
Consider the canonical isomorphism $\Ad\colon \GG\to\GG^\ad$ and the commutative diagram
\begin{equation*}
\xymatrix{
\pit\GG\ar[r]^{i}\ar[d]        &\Gtil\ar[d]^-{\text{projection}} \\
\pit\GG^\ad\ar[r]^-{i^\ad}     &G^\ssc
}
\end{equation*}
which induces a commutative diagram
\begin{equation}\label{e:pia-pit-action}
\begin{aligned}
\xymatrix@C=15truemm{
\Ho^0\hs\pia\GG\ar[r]^-\sim\ar[d]_-{\Ad_*^0} &\Ho^1\hs\pit\GG\ar[r]^-{i_*}\ar[d]^-{\Ad_*^1} & \Ho^1\hs\GGtil\ar[d]^\cong \\
\Ho^0\hs\pia\GG^\ad\ar[r]^-\sim              &\Ho^1\hs\pit\GG^\ad\ar[r]^-{i_*^\ad}          & \Ho^1\hs\GG^\ssc
}
\end{aligned}
\end{equation}
The map $\Ho^0\hs C=\Ho^0\hs\pia\GG^\ad\lra  \Ho^1\hs\GG^\ssc$
in the bottom row of the latter diagram comes from the action
of the group $\Ho^0\hs C$ on the cohomology set $\Ho^1\hs\GG^\ssc$,
and by Corollary \ref{c:Stab} the kernel of this map is $(\Ho^0\hs C)_q$\hs.
We see from the diagram \eqref{e:pia-pit-action} that the kernel of the map
$\Ho^0\hs\pia\GG\to \Ho^1\hs\GG^\ssc$ of \eqref{e:piR-pia-H1}
is the preimage in $\Ho^0\hs\pia\GG$ of $(\Ho^0\hs C)_q$
under the homomorphism $\Ad^0_*\colon \Ho^0\hs\pia\GG\to \Ho^0\hs C$,
that is, the stabilizer of $q$ under the action of $\Ho^0\hs \pia\GG$ on $\Km(\Dtil)$
described in Subsection \ref{ss:action}, as required.
\qed

\begin{proposition}\label{p:pi0R-funct}
Let $\varphi\colon \GG'\to\GG''$ be a homomorphism of connected reductive $\R$-groups.
Then the induced homomorphism $\varphi_*^{0}\colon\hs \Ho^0\hs\pia\hs\GG'\to \Ho^0\hs\pia\hs\GG''$
sends the subgroup $(\Ho^0\hs\pia\hs\GG')_1$ into $(\Ho^0\hs\pia\hs\GG'')_1$\hs,
and the following diagram commutes:
\begin{equation*}
\xymatrix@C=17truemm{
(\Ho^0\hs\pia\hs\GG')_1\, \ar[r]^{\varphi_*^{0}}\ar[d]_-{\psi'} &\, (\Ho^0\hs\pia\hs\GG'')_1\ar[d]^-{\psi''}\\
\piR \GG'\,\ar[r]^{\varphi_*}                                        &\, \piR\GG''
}
\end{equation*}
where $\psi'$ and $\psi''$ are the isomorphisms of Theorem \ref{t:pi0R}(i).
\end{proposition}

\begin{proof} A straightforward check. \end{proof}

\section{Connecting maps in exact sequences}
\label{s:connecting}

\subsec{}
Let
\begin{equation*}
 1\to \GG_1\labelto{i} \GG_2\labelto{j} \GG_3\to 1
\end{equation*}
be a short exact sequence of {\em connected reductive} $\R$-groups.
By Proposition \ref{p:pi-0} this sequence gives rise to  an exact sequence
\begin{align}\label{e:G'GG''-pi-coh}
\piR\GG_1\labelto{i_*^0} \piR\GG_2\labelto{j_*^0} \piR\GG_3\labelto{\delta^0}
      \Hon\GG_1\labelto{i_*^1} \Hon\GG_2 \labelto{j_*^1} \Hon\GG_3\hs.
\end{align}
In Theorems \ref{t:pi0R} and \ref{t:H1} we computed all groups and sets in this exact sequence.
We computed  the homomorphisms $i_*^0$ and $j_*^0$ in Proposition \ref{p:pi0R-funct},
and we computed the maps  $i_*^1$ and $j_*^1$ in Proposition \ref{p:H1-funct}.
Here we compute the connecting map $\delta^0\colon  \piR\GG_3\to \Hon\GG_1$.
It turns out that $\delta_0$ factorizes via a homomorphism into $\Hon \hm Z(\GG_1)$.

\subsec{}\label{ss:T,Tssc,X,Xssc}
We have an exact commutative diagram
\[
\xymatrix@R=6mm{
1\ar[r] &\GG_1^\ssc\ar[r]\ar[d]^{\rho_1} &\GG_2^\ssc\ar[r]\ar[d]^{\rho_2} &\GG_3^\ssc\ar[r]\ar[d]^{\rho_3}\ar@/_1pc/[l]_{s^\ssc} &1 \\
1\ar[r] &\GG_1\ar[r]^i\ar[d]^{\Ad_1}     &\GG_2\ar[r]^j\ar[d]^{\Ad_2}     &\GG_3\ar[r]\ar[d]^{\Ad_3}                             &1\\
1\ar[r] &\GG_1^\ad\ar[r]                 &\GG_2^\ad\ar[r]                 &\GG_3^\ad\ar[r]\ar@/^1pc/[l]^{s^\ad}                  &1
}
\]
in which the top and the bottom rows have canonical splittings
\[s^\ssc\colon \GG_3^\ssc\to \GG_2^\ssc,\qquad s^\ad\colon \GG_3^\ad\to \GG_2^\ad,\]
and these splittings are compatible with the composite vertical arrows,
that is, the following diagram commutes:
\begin{equation}\label{e:diag-composite}
\begin{aligned}
\xymatrix@R=6mm@C=10mm{
\GG_2^\ssc\ar[d]_-{\rho_2} &\GG_3^\ssc\ar[d]^-{\rho_3}\ar@/_0.75pc/[l]_{s^\ssc}\\
\GG_2\ar[d]_-{\Ad_2}       &\GG_3\ar[d]^-{\Ad_3}\\
\GG_2^\ad                  &\GG_3^\ad\ar@/^0.75pc/[l]^{s^\ad}
}
\end{aligned}
\end{equation}
We choose a maximal torus $\TT_2\subseteq \GG_2$ (not necessarily fundamental)
and we set $\TT_1=i^{-1}(\TT_2)\subseteq \GG_1$, $\TT_3=j(\TT_2)\subseteq\GG_3$.
We obtain commutative diagrams
\[
\xymatrix@R=6mm{
1\ar[r] &\TT_1^\ssc\ar[r]\ar[d]^-{\rho_1} &\TT_2^\ssc\ar[r]\ar[d]^-{\rho_2} &\TT_3^\ssc\ar[r]\ar[d]^-{\rho_3}\ar@/_1pc/[l]_{s^\ssc} &1 \\
1\ar[r] &\TT_1\ar[r]^i\ar[d]^-{\Ad_1}     &\TT_2\ar[r]^j\ar[d]^-{\Ad_2}     &\TT_3\ar[r]\ar[d]^-{\Ad_3}                             &1 \\
1\ar[r] &\TT_1^\ad\ar[r]                  &\TT_2^\ad\ar[r]                  &\TT_3^\ad\ar[r]\ar@/^1pc/[l]^{s^\ad}                   &1
}
\]
and
\[
\xymatrix@R=6mm{
0\ar[r] &\Xsc_1\ar[r]\ar[d]^-{\rho_1}  &\Xsc_2\ar[r]\ar[d]^-{\rho_2}  &\Xsc_3\ar[r]\ar[d]^-{\rho_3}\ar@/_1pc/[l]_{s^\ssc}   &0 \\
0\ar[r] &\Xv_1\ar[r]^i\ar[d]^-{\Ad_1}  &\Xv_2\ar[r]^j\ar[d]^-{\Ad_2}  &\Xv_3\ar[r]\ar[d]^-{\Ad_3}                           &0 \\
0\ar[r] &\Xad_1\ar[r]                  &\Xad_2\ar[r]                  &\Xad_3\ar[r]\ar@/^1pc/[l]^{s^\ad}                    &0
}
\]
where $\TT_k^\ssc=\rho_k^{-1}(\TT_k)$, $\TT_k^\ad=\Ad_k(\TT_k)$ for $k=1,2,3$, and where
$\Xv_k=\X_*(\TT_k)$  and similarly for $\Xsc_k$ and $\Xad_k$.
In these diagrams,  the rows are exact, but {\em not} the columns.
Note that the top and the bottom rows of these diagrams split canonically,
which can be written for the latter diagram as follows:
\begin{gather*}
\Xsc_2=\Xsc_1\oplus \Xsc_3\quad\text{and}\quad  \Xad_2=\Xad_1\oplus \Xad_3.
\end{gather*}
These splittings are compatible with the composite vertical arrows,
that is, a diagram similar to \eqref{e:diag-composite} commutes, which can be written as follows:
\[(\Ad_2\circ\rho_2)(0,\nu_3^\ssc)=\big(\,0,\,(\Ad_3\circ\rho_3)(\nu_3^\ssc)\,\big)
      \quad\text{for}\ \,\nu_3^\ssc\in \Xsc_3.\]

\begin{construction}
We construct  a canonical coboundary homomorphism
$$\delta_Z\colon\, \H^0(\Xsc_3\to \Xv _3)\,\longrightarrow\,
\H^0(\Xv_1\to \Xad_1).$$
We start with $[ \nu^\ssc_3,  \nu_3]\in \H^0(\Xsc_3\to \Xv_3)$.
Here $( \nu^\ssc_3,  \nu_3)\in  \Zl^0(\Xsc_3\to \Xv_3)$, that is,
\begin{equation}\label{e:*1}
\nu^\ssc_3\in \Xsc_3,\quad \nu_3\in \Xv_3,\quad
\upgam \nu^\ssc_3+ \nu^\ssc_3=0,\quad  \upgam \nu_3- \nu_3=\rho_3( \nu^\ssc_3).
\end{equation}
We lift canonically $ \nu^\ssc_3$ to
$$ \nu^\ssc_2=(0, \nu^\ssc_3)\hs\in\hs \Xsc_1\oplus \Xsc_3= \Xsc_2,$$
and we lift $\nu_3$ to some $\nu_2\in \Xv_2$.
We write
$$\Ad_2(\nu_2)=( \nu^\ad_1, \nu^\ad_3)\hs\in\hs \Xad_1\oplus \Xad_3=\Xad_2,$$
where $ \nu^\ad_3=\Ad_3(\nu_3)\in \Xad_3$ and $ \nu^\ad_1\in \Xad_1$.
We set
$$\nu_1=\upgam \nu_2- \nu_2-\rho_2( \nu^\ssc_2).$$
Since by \eqref{e:*1}  we have
\[\upgam \nu_3- \nu_3=\rho_3( \nu^\ssc_3),\]
we see that  $ \nu_1\in \Xv_1$.
Straightforward calculations show that
\[\upgam \nu_1+\nu_1=0,\quad\ \Ad_1(\nu_1)=\upgam \nu^\ad_1-\nu^\ad_1.\]
We see that $(\nu_1, \nu^\ad_1)\in \Zl^0(\Xv_1\to \Xad_1)$.
We set
$$\delta_Z[ \nu^\ssc_3, \nu_3]=[ \nu_1, \nu^\ad_1]\in \H^0(\Xv_1\to \Xad_1).$$
A straightforward check shows that the map $\delta_Z$ is a well-defined homomorphism.

We have $\pia{\GG_3}=\Xv_3/\rho_3(\Xsc_3)$ with injective $\rho_3$,
whence $\H^0(\Xsc_3\to \Xv_3)\cong\Ho^0\hs\pia{\GG_3}$.
Moreover, by  Theorem \ref{t:quasi} we have a canonical isomorphism
$\H^0(\Xv_1\to \Xad_1)\cong\Ho^1 Z(\GG_1)$;
see Subsection \ref{ss:quasi} for an explicit formula for this isomorphism.
Thus we obtain a canonical homomorphism
\begin{equation*}
\delta_Z\colon \Ho^0\hs\pia{\GG_3}\to \Ho^1 Z(\GG_1)
\end{equation*}

We show that the connecting map $\delta^0 \colon \piR\GG_3\to\Hon\GG_1$
in the exact sequence \eqref{e:G'GG''-pi-coh}
factors via $\delta_Z$.
\end{construction}

\begin{theorem}\label{t:conn-pi0R-H1G}
With the above assumptions and notation, the following diagram commutes:
\begin{equation}\label{e:pi-Z-T}
\begin{aligned}
\xymatrix@C=13truemm{
(\Ho^0\hs\pia\hs\GG_3)_1\ar[r]^{\delta_Z}\ar[d]_-{\psi_3}^-\cong  & \Hon\hm Z(\GG_1)\ar[d]^-{\iota_*}\ar[r]  &\Hon\TT_1\ar[d]\\
\piR\GG_3\ar[r]^{\delta^0}                               &\Hon\GG_1\ar@{=}[r]       &\Hon\GG_1
}
\end{aligned}
\end{equation}
where the map $\iota_*$ is induced by the inclusion map $\iota\colon Z(\GG_1)\into\GG_1$.
\end{theorem}

This theorem computes the connecting map $\delta^0$.

\begin{proof}
Consider $[\nu_3]\in(\Ho^0\hs\pia\hs\GG_3)_1$ and  $\psi_3[\nu_3]\in\piR\GG_3$.
We write:
\begin{align*}
&\nu_3\in \Xv_3,\quad \upgam\nu_3-\nu_3=\rho_3(\nu_3^\ssc),\quad \nu_3^\ssc\in \Xsc_3,\quad
\nu_3^\ssc(-1)=(g_3^\ssc)^{-1}\cdot\hm\upgam g_3^\ssc,\\
&g_3=\rho_3(g_3^\ssc)\cdot\nu_3(-1)\in\GG_3(\R),\quad \psi_3[\nu_3]=[g_3]
\in \piR\GG_3.
\end{align*}
We lift canonically $\nu_3^\ssc$ to $\nu_2^\ssc\in \Xsc_2$
and lift canonically $g_3^\ssc$ to $g_2^\ssc\in G_2^\ssc$,
that is, $\nu_2^\ssc=(0,\nu_3^\ssc)$ and  $g_2^\ssc=(1,g_3^\ssc)$.
We lift $\nu_3$ to some $\nu_2\in \Xv_2$.
We set
\[g_2=\rho_2(g_2^\ssc)\cdot\nu_2(-1)\in G_2.\]
Then $g_2$ is a lift of $g_3\in\GG_3(\R)$.
We set
\[z_1=g_2^{-1}\cdot\hm\upgam g_2.\]
Then $z_1\in\Zl^1\hs\GG_1$, and
\[\delta^0\big(\psi_3[\nu_3]\big)=\delta^0[g_3]
=[z_1]\in\Hon\GG_1.\]
We calculate:
\begin{align*}
z_1=g_2^{-1}\cdot\hm\upgam g_2
&=\nu_2(-1)\cdot\rho_2(g_2^\ssc)^{-1}\cdot\hm\upgam\rho_2(g_2^\ssc)\cdot\hm\upgam\nu_2(-1)\\
&=\nu_2(-1)\cdot\rho_2(\nu_2^\ssc)(-1)\cdot\hm\upgam\nu_2(-1)=\nu_1(-1),
\end{align*}
where $\nu_1=\upgam\nu_2-\nu_2-\rho_2(\nu_2^\ssc)\in \Zl^1 \Xv_1$.
Thus
\begin{equation}\label{e:delta-psi}
\delta^0\big(\psi_3[\nu_3]\big)=[z_1]=\big[\nu_1(-1)\big]\in\Hon\GG_1.
\end{equation}

Recall that $\delta_Z[\nu_3^\ssc,\nu_3]=[\nu_1,\nu_1^\ad]\in  \H^0(\Xv_1\to \Xad_1)$.
By Lemma \ref{l:A-X(T)-X(T')} the image of $[\nu_1,\nu_1^\ad]$ under the composite map
\[  \H^0(\Xv_1\to \Xad_1)\to \Ho^1 Z(\GG_1)\to \Ho^1\hs\TT_1\]
is $\big[\nu_1(-1)\big]\in \Ho^1\hs\TT_1$. Since the right-hand rectangle in the  diagram \eqref{e:pi-Z-T} clearly commutes,
we see that
\begin{equation}\label{e:iota-delta}
\iota_*\big(\delta_Z[\nu_3]\big)=\big[\nu_1(-1)\big]\in \Hon\GG_1.
\end{equation}
By  \eqref{e:delta-psi} and \eqref{e:iota-delta},
the left-hand rectangle  in the  diagram \eqref{e:pi-Z-T} commutes,
which completes the proof of the theorem.
\end{proof}

\subsec{}
Let
\[1\to \AA_1\labelto{i}\GG_2\labelto{j}\GG_3\to 1\]
be a short exact sequence, where $\GG_2$ and $\GG_3$ are connected reductive $\R$-groups,
and $\AA_1=\ker\big[\GG_2\to\GG_3\big]$ is a central subgroup.
Then $\AA_1$ is an $\R$-quasi-torus.
By Proposition \ref{p:pi-0}, the above short exact sequence gives rise to an exact sequence
\begin{equation*}
 \piR \GG_2\labelto{j_*^0}\piR \GG_3\labelto{\delta^0}
     \Ho^1\hm\AA_1 \labelto{i_*^1} \Ho^1 \GG_2 \labelto{j_*^1} \Ho^1 \GG_3\hs.
\end{equation*}
In Theorems \ref{t:pi0R}, \ref{t:quasi}, and \ref{t:H1},
we computed all groups and sets in this  exact sequence.
We computed  the homomorphism  $j_*^0$ in Proposition \ref{p:pi0R-funct},
and we computed the map $j_*^1$ in Proposition \ref{p:H1-funct}.
The map $i_*^1$ factors via $\Ho^1 Z(\GG_2)$
and therefore can be computed using Proposition \ref{p:Z-action}.
In the rest of this section we compute the connecting homomorphism  $\delta^0$.

\subsec{}
Let $\TT_2\subset\GG_2$ be a maximal torus (not necessarily fundamental).
We set $\TT_3=j(\TT_2)\subset\GG_3$.
For $k=2,3$, we define $\TT_k^\ssc$, $\Xv_k$, and $\Xsc_k$ as in Subsection \ref{ss:T,Tssc,X,Xssc}.
Then we have a commutative diagram
\[
\xymatrix@C=15mm{
    \Xsc_2\ar[r]^-{j_*^\ssc} \ar[d]_-{\rho_2}   &\Xsc_3\ar[d]^-{\rho_3}\ar[ld]_-\upsilon  \\
    \Xv_2\ar[r]^-{\quad j_*}                    &\Xv_3
}
\]
in which $j_*^\ssc$ is an isomorphism and where $\upsilon=\rho_2\circ(j_*^\ssc)^{-1}$.
We obtain a morphism of short complexes
\[(\upsilon,\id)\hs\colon\, (\Xsc_3\to \Xv_3)\lra (\Xv_2\to \Xv_3)\]
and the induced homomorphism on hypercohomology
\[\upsilon_*\hs\colon\, \H^0 (\Xsc_3\to \Xv_3)\lra \H^0(\Xv_2\to \Xv_3).\]
We identify $\H^0(\Xsc_3\to \Xv_3)$ with $\Ho^0\pia\GG_3$ by Lemma~\ref{l:quasi-inj}.

\begin{theorem}
The following diagram is commutative:
\begin{equation}\label{e:A1-G2-G3-delta}
\begin{aligned}
\xymatrix@C=15mm@R=7mm{
\Ho^0\pia\GG_3\ar[r]^-{\upsilon_*}                                       &\H^0(\Xv_2\to \Xv_3)\ar[d]^-{\rm ev_*^0}_-\cong\\
\big(\Ho^{\strut0}\pia\GG_3\big)_1\ar@{^(->}[u]\ar[d]_-{\psi_3}^-\cong   &\H^0(\TT_2\to\TT_3)\\
\piR\GG_3\ar[r]^-{\delta^0}                                              &\Ho^1\hm\AA_1\ar[u]^-\cong_-{i_\#^0}
}
\end{aligned}
\end{equation}
in which the left-hand vertical arrow $\psi_3$ is from Theorem \ref{t:pi0R}, and
the right-hand vertical arrows are from the diagram \eqref{e:A-X(T)-X(T')}.
\end{theorem}

This theorem computes the connecting homomorphism $\delta^0$.
See Subsection \ref{ss:quasi} for an explicit formula
for the isomorphism $\H^0(\Xv_2\to \Xv_3)\isoto \Ho^1\hm\AA_1$.

\begin{proof}
Consider $[\nu_3]\in(\Ho^0\hs\pia\hs\GG_3)_1$
and $\psi_3[\nu_3]\in\piR\GG_3$. As in the proof of Theorem~\ref{t:conn-pi0R-H1G},
we write:
\begin{align*}
&\nu_3\in \Xv_3,\quad \upgam\nu_3-\nu_3=\rho_3(\nu_3^\ssc),\quad \nu_3^\ssc\in \Xsc_3,\quad
\nu_3^\ssc(-1)=(g_3^\ssc)^{-1}\cdot\hm\upgam g_3^\ssc,\\
&g_3=\rho_3(g_3^\ssc)\cdot\nu_3(-1)\in\GG_3(\R),\quad \psi_3[\nu_3]=[g_3]
\in \piR\GG_3.
\end{align*}

We lift canonically $\nu_3^\ssc$ to $\nu_2^\ssc\in \Xsc_2$
and lift canonically $g_3^\ssc$ to $g_2^\ssc\in G_2^\ssc$.
We lift $\nu_3$ to some $\nu_2\in \Xv_2$.
Then $j_*(\nu_2)=\nu_3$ and $\rho_2(\nu_2^\ssc)=\upsilon(\nu_3^\ssc)$.
We set
\[g_2=\rho_2(g_2^\ssc)\cdot\nu_2(-1)\in \GG_2(\C).\]
Then $g_2$ is a lift of $g_3\in\GG_3(\R)$.
We set
\[a_1=g_2^{-1}\cdot\hm\upgam g_2.\]
Then $a_1\in\Zl^1\hm\AA_1$, and
\[\delta^0\big(\psi_3[\nu_3]\big)=\delta^0[g_3]
=[a_1]\in\Hon\GG_1.\]
We have
\begin{align*}
a_1=g_2^{-1}\cdot\hm\upgam g_2
&=\nu_2(-1)\cdot\rho_2(g_2^\ssc)^{-1}\cdot\hm\upgam\rho_2(g_2^\ssc)\cdot\hm\upgam\nu_2(-1)\\
&=\nu_2(-1)\cdot\rho_2(\nu_2^\ssc)(-1)\cdot\hm\upgam\nu_2(-1).
\end{align*}
Thus
\begin{equation*}
\delta^0\big(\psi_3[\nu_3]\big)=[a_1]=\big[\nu_2(-1)\cdot\upgam\nu_2(-1)
       \cdot\rho_2(\nu_2^\ssc)(-1)\big]\in\Hon\!\AA_1
\end{equation*}
and
\begin{align*}
(i_\#^{0}\circ\delta^0\circ\psi_3)[\nu_3]
&=\big[\nu_2(-1)\cdot\upgam\nu_2(-1)\cdot\rho_2(\nu_2^\ssc)(-1),\hs 1\big]\\
&=\big[\rho_2(\nu_2^\ssc)(-1),\, j_*(\nu_2)(-1)\big]\hs\in\H^0(\TT_2\to\TT_3).
\end{align*}

On the other hand, we have
\begin{align*}
&\upsilon_*[\nu_3]=[\upsilon(\nu_3^\ssc),\hs \nu_3]\hs\in \H^0(\Xv_2\to \Xv_3),\\
&(\ev_*^0\circ\upsilon_*)[\nu_3]=\big[\upsilon(\nu_3^\ssc)(-1),\hs \nu_3(-1)\big]\in\hs
 \H^0(\TT_2\to\TT_3).
\end{align*}
Since $\rho_2(\nu_2^\ssc)=\upsilon(\nu_3^\ssc)$ and $j(\nu_2)=\nu_3$, we conclude
that the diagram \eqref{e:A1-G2-G3-delta} indeed commutes, as required.
\end{proof}

\section{Examples}
\label{s:examples}

\subsec{}\label{ss:s-q-c-q}
For  {\em even} $l>4$, consider the real spin group
$\GG_q^\ssc=\GG(\DDD_\ell\hs,0,\id,q)$ of type $\DDD_\ell$
with the notation of \cite[Section 12.12]{BT},
where $q\in \Km(\Dtil)$ is a Kac labeling of
the affine Dynkin diagram $\Dtil=\DDD_\ell^{(1)}$
 with the notation of \cite[Table 6]{OV}.
This group comes with a {\em compact} maximal torus $\TT^\ssc$
and a system of simple roots $\Sm=\{\alpha_1,\dots,\alpha_\ell\}$
with the notation of \cite[Table 1]{OV}.
We consider the fundamental coweight
\[ \omega_{\ell-1}^\vee=(\vev_1+\dots+\vev_{\ell-1}-\vev_\ell)/2\,\in P^\vee\subset\tl,\]
where $\vev_1,\dots,\vev_{\ell}$ are the cocharacters dual
to the weights $\ve_1,\dots,\ve_{\ell}$ of the vector representation, cf.\ \cite[Table 1]{OV}.
Set
\[a=\exp 2\pi\ii\hs \omega^\vee_{\ell-1}\in \TT^\ssc(\R)\subset\GG_q^\ssc(\R).\]
Since $\ell$ is even, we have
$2\omega^\vee_{\ell-1}\in Q^\vee=\X_*(T^\ssc)$, and hence  $a^2=1$.

We consider the one-dimensional  split $\R$-torus $\TT^1_s$
and the one-dimensional compact $\R$-torus $\TT^1_c$;
see Corollary \ref{c:indecoposable-tori}.
Consider the elements of order 2
\[-1\in \TT^1_s(\R)\quad\text{and}\quad -1\in \TT^1_c(\R).\]
We consider the reductive $\R$-groups
\[\GG_{s,q}=(\GG^\ssc_q\times\TT^1_s)\hs/\hs\{1,(a,-1)\}\quad \text{and}\quad
\GG_{c,q}=(\GG^\ssc_q\times\TT^1_c)\hs/\hs\{1,(a,-1)\}.\]

\subsec{}
Let $\GG$ be either $\GG_{s,q}$ or $\GG_{c,q}$.
We wish to compute  $\Ho^1\hs\GG$ and $\piR \GG$.
We use Notation~\ref{ss:Notation-reductive}.
Recall that $\SS=Z(\GG)^0$; then $\SS$ is either $\TT^1_s$ or $\TT^1_c$.
Recall that $\GG^\sss=[\GG,\GG]$; then $\GG^\sss=\GG_q^\ssc$.
We have $\GG=\GG^\sss\cdot\SS=\GG^\ssc\cdot \SS$.
We set
\[\TT=\TT^\sss\cdot\SS=\TT^\ssc\cdot\SS=(\TT^\ssc\times\SS)\hs/\hs\{1,(a,-1)\}.\]
Then $\TT$ is a fundamental torus of $\GG$.
We have
\[\X_*(\TT)=\left\langle Q^\vee\oplus\X_*(\SS),\omega^\vee_{\ell-1}+\half\vev\right\rangle,\]
where $\X_*(\SS)=\langle\vev\rangle\simeq\Z$.

We freely use the notation of \cite[Section 16]{BT}.

With the notation of Section \ref{s:H1}, we have  $G^\der=\Spin_{2\ell}(\C)$, $S=\C^\times$.
Let $\ve\in\X^*(S)$ be the basis character of the one-dimensional torus $S$ dual to $\vev$.
Then we have:
\begin{gather*}
\Lambda      = \langle\ve\rangle, \qquad
      X      = \left\langle\ve_i\pm\ve\ (i=1,\dots,\ell),\ \omega_{\ell-1}\pm\tfrac\ell2\ve\right\rangle, \qquad
      M      = \langle2\ve\rangle,      \\
      X^\vee = \bigg\langle\frac{\pm\vev_1\pm\dots\pm\vev_{\ell}\pm\vev}2\ \bigg|\
      \text{with odd number of minuses among  $\pm\vev_i$}\bigg\rangle, \\
\Lambda^\vee = \langle\vev\rangle, \qquad
      M^\vee = \langle\half\vev\rangle.
\end{gather*}

\subsec{}\label{ss:H1-Gsq}
We compute $\Ho^1\hs\GG_{s,q}$.
In this case  $\SS=\TT_s^1$ is a split torus.
The automorphism $\tau$ acts on the weights and coweights as follows:
\[\ve\mapsto-\ve,\ \ \vev\mapsto-\vev,\  \ \ve_i\mapsto\ve_i\hs,\ \  \vev_i\mapsto\vev_i.\]
Therefore, $T_0$ is a maximal torus in $G^\sss$ and
\begin{gather*}
\Lambda_0 = M_0 = 0, \qquad \Lambda_0^\vee = M_0^\vee = 0, \qquad X_0 = P, \\
\wt X^\vee_0 =
\left\langle\frac{\pm\vev_1\pm\dots\pm\vev_{\ell}}2\biggm|
   \text{with odd number of minuses among  $\pm\vev_i$}\right\rangle.
\end{gather*}
Hence  in Theorem \ref{t:H1} we have $\mm=0$  and the cohomology classes
correspond to Kac labelings $p\in\Km(\Dtil)$.

The lattice $X_0$ is generated by the root lattice $Q$ and the weights
\begin{align*}
\frac{\ve_1-\ve_2+\ve_3-\ve_4+\dots+\ve_{\ell-3}-\ve_{\ell-2}+\ve_{\ell-1}+\ve_{\ell}}2
   \, &=\, \frac{\alpha_1+\alpha_3+\dots+\alpha_{\ell-3}+\alpha_{\ell}}2\\
\text{ and }\, \ve_{\ell-1}\, &=\, \frac{\alpha_{\ell-1}+\alpha_{\ell}}2\hs.
\end{align*}

For $p\in\Km(\Dtil)$, we set:
\begin{align*}
r(p)&=p_1+p_3+\dots+p_{\ell-3}+p_\ell\pmod{2},\\
r'(p)&=p_{\ell-1}+p_\ell\hs.
\end{align*}
The congruences \eqref{e:pr=q} are equivalent to
$r(p)\equiv r(q)$, $r'(p)\equiv r'(q)\pmod2$.

The group $F_0$ is generated by the class  $[\omega^\vee_{\ell-1}]$.
It acts on $\Km(\Dtil)$ by the reflection
with respect to the vertical symmetry axis of $\Dtil$:
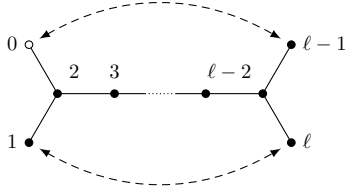
\[
\dynkin[%
edge length=0.75cm,
labels={0,1,,,,,\ell-1,\ell},
labels*={,,2,3,,\ell-2,,},
involution/.style={latex-latex,densely dashed},
involutions={*06;17}]
D[1]{}
\]
We denote by $[p]$ the $F_0$-orbit of $p\in \Km(\Dtil)$.
Then $r(p)$ and $r'(p)$ depend only on $[p]$\hs.

By Theorem \ref{t:H1}, the set $\Ho^1\hs\GG_{s,q}$
is in a canonical bijection with the set of orbits
\[\Orbb{r(q),\hs r'(q)}\coloneqq\big\{\hs[p]\ \,\big|\ \,  p\in\Km(\Dtil),
     \, r(p)\equiv r(q),\, r'(p)\equiv r'(q)\!\!\!\pmod2\hs\big\}.\]
These four sets $\Orb{r,r'}$ (described by representatives of orbits)  are:
\begin{align*}
&\Orb{0,0}:\
&\gf{2}{0}0 \cdots 0\gf{0}{0}, \quad
\gf{0}{2}0 \cdots 0\gf{0}{0}, \quad
\text{and} \quad \
\gf{0}{0}0 \cdots 1 \cdots 0\gf{0}{0}  &\quad
\text{with $1$ \,at $i=2j$},
\\
&\Orb{0,1}:\quad
&\gf{1}{0}0 \cdots 0\gf{1}{0}, \quad
\gf{0}{1}0 \cdots 0\gf{0}{1},\,&
\\
&\Orb{1,0}:\quad
&\gf{1}{1}0 \cdots 0 \gf{0}{0}, \quad
\text{and} \quad
\gf{0}{0}0 \cdots 1 \cdots 0\gf{0}{0} &\quad
\text{with $1$ \,at $i=2j+1$},
\\
&\Orb{1,1}:\quad
&\gf{1}{0}0 \cdots 0 \gf{0}{1}.\,&
\end{align*}
for each integer $i$ (even or odd, respectively) with $1<i\le \ell/2$.
We have
\begin{equation*}
\#\Orb{0,0}=\lfloor \ell/4\rfloor+2, \quad
\#\Orb{0,1}=2,\quad
\#\Orb{1,0}=\lceil \ell/4\rceil,\quad
\#\Orb{1,1}=1.
\end{equation*}
We see that
\[\#\Ho^1\hs\GG_{s,q}\hs=\hs\#\Orbb{r(q),\hs r'(q)}\hs.\]
Thus we know the cardinalities  of $\Ho^1\hs\GG_{s,q}$ for all $q\in\Km(\Dtil)$:

\noindent
If $r'(q)=0$, then
\[\#\Ho^1\hs\GG_{s,q}=
\begin{cases}
\,\lfloor \ell/4\rfloor+2,   & r(q)=0, \\
\,\lceil \ell/4\rceil,     & r(q)=1.
\end{cases}
\]
If $r'(q)=1$, then
\[\#\Ho^1\hs\GG_{s,q}=
\begin{cases}
\,2,     & r(q)=0, \\
\,1,     & r(q)=1,
\end{cases}
\bigg\}=2-r(q).
\]

\subsec{}
We compute $\Ho^1\hs\GG_{c,q}$.
In this case  $\SS=\TT_c^1$ is a compact torus.
The difference with the previous example
is that now $\tau=\id$
and $T_0=T$. It follows that
\begin{gather*}
\Lambda_0=\Lambda,\qquad M_0=M,\qquad X_0=X,\qquad\text{and} \\
\Lambda_0^\vee=\wt\Lambda_0^\vee=\Lambda^\vee,\qquad M_0^\vee=
   \wt{M}_0^\vee=M^\vee,\qquad X_0^\vee=\wt{X}_0^\vee=X^\vee.
\end{gather*}
In particular,
$M_0^\vee/2\wt\Lambda_0^\vee=\langle\half\vev\rangle/
    \langle2\vev\rangle\simeq\Z/4\Z=\{0,1,2,3\}$.
For $m\in M^\vee$, we denote its class in $\Z/4\Z$ by $[m]$,
that is, $[m]=k\bmod4$ if $m=\frac{k}2\vev$.

The lattice $X$ is generated by the root lattice $Q$ and the weights
\begin{align*}
\frac{\ve_1-\ve_2+\ve_3-\ve_4+\dots+\ve_{\ell-3}-\ve_{\ell-2}+\ve_{\ell-1}+\ve_{\ell}}2 \,&=\,
           \frac{\alpha_1+\alpha_3+\dots+\alpha_{\ell-3}+\alpha_{\ell}}2\\
\text{ and }\, \ve_{\ell-1}+\ve \,&=\, \frac{\alpha_{\ell-1}+\alpha_{\ell}}2+\ve.
\end{align*}
For $m\in M^\vee$, we set
$$
r''(m)=2\langle\ve,m\rangle\bmod2=
\begin{cases}
0, & [m]=0\text{ or }2, \\
1, & [m]=1\text{ or }3.
\end{cases}
$$
The congruences \eqref{e:pr=q} are equivalent to
$r(p)\equiv r(q)$, $r'(p)+r''(m)\equiv r'(q)\pmod2$.

The group $F_0$ is generated by the class  $[\omega^\vee_{\ell-1}+\half\vev]$.
It acts on $\Km(\Dtil)$ by the reflection with respect to the vertical symmetry axis of $\Dtil$
and on $M_0^\vee/2\wt\Lambda_0^\vee\simeq\Z/4\Z$ as $0\leftrightarrow2$, $1\leftrightarrow3$.
Note that $r(p)$, $r'(p)$, and $r''(m)$ depend only
on the $F_0$-orbit of $\big(p,[m]\big)\in\Km(\Dtil)\times M_0^\vee/2\wt\Lambda_0^\vee$.

Let $\Orb{r,r',r''}$ denote the set of $F_0$-orbits of $\big(p,[m]\big)$
such that
\[r(p)\equiv r,\quad\, r'(p)\equiv r',\quad\,  r''(m)\equiv r''\!\!\!\!\pmod2.\]
The representatives of the orbits in $\Orb{r,r',r''}$ are $\big(p,[m]\big)$,
where $p$ are the representatives of the orbits in $\Orb{r,r'}$ and
$$[m]=
\begin{cases}
r'',                 &\text{if $p$ is fixed by }F_0, \\
r''\text{ or }r''+2, &  \text{otherwise}.
\end{cases}$$
The cardinalities of these eight orbit sets are:
\begin{align*}
\#\Orb{0,0,0}\hs=\hs\#\Orb{0,0,1}&=
\begin{cases}
2\lfloor \ell/4\rfloor+3, & \ell/2\text{ even}, \\
2\lfloor \ell/4\rfloor+4, & \ell/2\text{ odd},
\end{cases}&\!
\#\Orb{0,1,0}\hs=\hs\#\Orb{0,1,1}&=2,\\
\#\Orb{1,0,0}\hs=\hs\#\Orb{1,0,1}&=
\begin{cases}
2\lceil \ell/4\rceil,   & \ell/2\text{ even}, \\
2\lceil \ell/4\rceil-1, & \ell/2\text{ odd},
\end{cases}&\!
\#\Orb{1,1,0}\hs=\hs\#\Orb{1,1,1}&=2.
\end{align*}
By Theorem~\ref{t:H1}, the set $\Ho^1\hs\GG_{c,q}$ is in a canonical bijection
with the union of two orbit sets $\Orbb{\hs r(q),\hs r'(q),\hs0}\cup\Orbb{\hs r(q),\hs r'(q)-1,\hs1}$.
We obtain
\[\#\Ho^1\hs\GG_{c,q}\hs=\hs\#\Orbb{\hs r(q),\hs r'(q),\hs0}\hs+\hs\#\Orbb{\hs r(q),\hs r'(q)-1,\hs1}\hs.\]
Thus if $r(q)=0$, then
$$
\#\Ho^1\hs\GG_{c,q}\hs=
\begin{cases}
\,2\lfloor \ell/4\rfloor+5, & \ell/2\text{ even} \\
\,2\lfloor \ell/4\rfloor+6, & \ell/2\text{ odd}
\end{cases}
\bigg\}=\ell/2+5.
$$
If $r(q)=1$, then
$$
\#\Ho^1\hs\GG_{c,q}\hs=
\begin{cases}
\,2\lceil \ell/4\rceil+2,   & \ell/2\text{ even} \\
\,2\lceil \ell/4\rceil+1, & \ell/2\text{ odd}
\end{cases}
\bigg\}=\ell/2+2.
$$

\subsec{}\label{ss:-present-vs-old}
In \cite{Borovoi88} (see also \cite{Borovoi-arXiv}),
for any connected reductive $\R$-group $\GG$,
the first-named author constructed a bijection between
$\Ho^1\hs\GG$ and the set of orbits of a certain action of $W_0$ on $\Ho^1\hs\TT$,
where $\TT$ is a fundamental torus of $\GG$,
and $W_0$ is a certain subgroup of the Weyl group $W=W(G,T)$; see \cite[Section 4]{BT}.
We compare our present formula for $\Ho^1\hs\GG$ with the old formula of \cite{Borovoi88}
in the cases $\GG= \GG_{s,q}$ and $\GG= \GG_{c,q}$ as defined in Subsection~\ref{ss:s-q-c-q}.
In both cases we have $W_0=W$, and so $W_0$ is
a Coxeter group of type $\DDD_\ell$ and of order $\ell!\hs2^{\ell-1}$.
Moreover, in both cases the group $F_0$ is generated by an element acting on $\Dtil$
by the reflection with respect to the vertical symmetry axis.

For $\GG=\GG_{s,q}$ we obtain the set $\Ho^1\hs \GG$ of cardinality $\le \ell/4+2$
by computing orbits of a group  of order 2.
On the other hand, with the old formula we obtain
$\#\Ho^1\hs \GG$ as the set of orbits of the group $W_0$ of order $\ell!2^{\ell-1}$
in the set $\Ho^1\hs\TT$ of cardinality $2^{\ell-1}$.

Similarly, for $\GG=\GG_{c,q}$ we obtain the set $\Ho^1\hs \GG$ of cardinality $\ell/2+2$ or $\ell/2+5$
by computing orbits of a group of order 2.
On the other hand, with the old formula we obtain
$\#\Ho^1\hs \GG$ as the set of orbits of the group $W_0$ of order $\ell!\hs2^{\ell-1}$
in the set $\Ho^1\hs\TT$ of cardinality $2^{\ell+1}$.

We conclude that our present method requires less calculations than the old one.

\subsec{}
We compute $\piR \GG$, where $\GG=\GG_{c,q}$.
We have $\pia\GG\simeq\Z$, where $\tau$ acts on  $\pia\GG$ trivially, and $\gamma$ acts as $-1$.
We see that $\Ho^0\pia\GG=0$, and by Theorem \ref{t:pi0R} we have
\[ \piR\GG_{c,q}\cong (\Ho^0\pia\GG)_1=0.\]

\subsec{}
We compute $\piR \GG$, where $\GG=\GG_{s,q}$.
We have $\pia\GG\simeq\Z$, where $\tau$ acts on  $\pia\GG$ as $-1$, and $\gamma$ acts trivially.
We see that $\Ho^0\pia\GG\simeq \Z/2\Z$.
We wish to compute  $(\Ho^0\pia\GG)_1$.

The group $\Ho^0\pia\GG$ acts on $\Km(\Dtil)$ via the homomorphism
$$
\Ad_*\colon\, \Ho^0\pia\GG \longrightarrow \Ho^0C = C = C_0 \simeq \Z/2\Z\oplus\Z/2\Z,
$$
which sends the generator $[\omega^\vee_{\ell-1}+\half\vev]$ of $\Ho^0\pia\GG$ to $[\omega^\vee_{\ell-1}]$.
As already noted in Subsection~\ref{ss:H1-Gsq},
$[\omega_{l-1}^\vee]$ acts on $\Km(\Dtil)$
by reflecting the Dynkin diagram $\Dtil$ with respect to the vertical symmetry axis.

By Theorem~\ref{t:pi0R}, $\piR\GG_{s,q}=(\Ho^0\pia\GG)_1$ is the stabilizer of $q$
under the action of $\Ho^0\pia\GG$ on $\Km(\Dtil)$.
It follows that $\piR\GG_{s,q}$ is nontrivial
if and only if $q$ is fixed by the aforementioned reflection.
This condition is satisfied if and only if
$$
q\, =\quad \gf{0}{0}0 \cdots 1 \cdots 0\gf{0}{0}
\text{\; with $1$ at $\ell/2$ \quad or \quad}
\gf{1}{0}0 \cdots 0\gf{1}{0},
$$
up to the action of $C$, that is,
$\GG^\sss=\GG^\ssc$ is isomorphic to either $\Spin_{\ell,\ell}$ or $\Spin^*_{2\ell}$\,;
see \cite[Table 7]{OV}.
In this case $\piR\GG_{s,q}\simeq\Z/2\Z$.

\appendix

\section{Classification of $\Gamma$-lattices}
\label{s:Indecomposable}

For the reader's convenience, we provide a short elementary proof
of the following known result.

\begin{theorem}[\hs{Curtis and Reiner \cite[Theorem (74.3)]{CR}}\hs]
\label{t:indec}
Let $L$ be a lattice (a finitely generated free abelian group),
and $\tau\colon L\to L$ be an involutive automorphism, that is, $\tau^2=\id$.
Then there exists a basis of $L$ consisting of vectors
$e_i,\ f_j,\ g_k,\ h_k,$ on which $\tau$ acts as follows:
$\tau(e_i) = e_i$, $\tau(f_j) = -f_j$, $\tau(g_k) = h_k$, $\tau(h_k) = g_k$.
\end{theorem}

\begin{proof}
We prove the theorem by induction on the rank $r$ of $L$. The case $r=0$ is trivial.

Induction step: assume that $r\ge 1$.
If $\tau=-\id$, there is nothing to prove.
Otherwise there is a nonzero $\tau$-fixed vector $e$,
which we can choose to be primitive (indivisible).
Consider the quotient group $L/\langle e\rangle$.
Since $e$ is primitive, $L/\langle e\rangle$ is a lattice
(free abelian group) of rank $r-1$.
Clearly, $\tau$ acts on  $L/\langle e\rangle$ as an involution.
By the induction hypothesis, the quotient lattice $L/\langle e\rangle$
has a basis $[e_i],\ [f_j],\ [g_k],\ [h_k]$ with required properties,
where $[v]$ denotes the coset of a vector $v$ in $L$.
We consider the action of $\tau$ on the basis
$e,\ e_i,\ f_j,\ g_k,\ h_k$ of $L$.

Firstly, $\tau(e_i) = e_i$. Otherwise it would be
$\tau(e_i) = e_i+me$ with some  nonzero $m$,
but then $\tau^2(e_i) = e_i+2me\neq e_i$,
which  contradicts the assumption $\tau^2=\id$.

Secondly, $\tau(g_k) = h_k+me$ for some integer $m$ (depending on $k$).
After replacing  $h_k$ with $h_k+me$, we have $\tau(g_k) = h_k$\hs, and hence  $\tau(h_k) = g_k$
(because $\tau$ is involutive).

Finally, $\tau(f_j) = -f_j+me$ for some integer $m$ (depending on $j$).
If $m = 2n$ is even, then  $\tau(f_j-ne) = -f_j+ne$, and after replacing
$f_j$ with $f_j-ne$ we obtain $\tau(f_j) = -f_j$.
If $m = 2n+1$ is odd, the same replacement gives $\tau(f_j) = -f_j+e$.

If this latter case $\tau(f_j) = -f_j+e$ does not appear, the proof is complete.
Otherwise let us fix some $j$, say, $j=0$, such that $\tau(f_0)= -f_0+e$,
and consider all other $j$ for which $\tau(f_j) = -f_j+e$.
After replacing $f_j$ with $f_j-f_0$ for all these other $j$,
we obtain  $\tau(f_j) = -f_j$.
In other words, we may assume that $\tau(f_j) = -f_j+e$ holds for only one $j$.

Now we replace $e$ with $-f_j+e$ and obtain two basis vectors
$g = f_j$ and  $h= -f_j+e$, for which $\tau(g) = h$ and $\tau(h) = g$.
Thus we obtain a basis of $L$ with required properties,
which completes the proof of the theorem.
\end{proof}

See Casselman \cite{Casselman} for another elementary proof of Theorem \ref{t:indec}.


\end{document}